\newif\iffigure
\newif\ifcycle
\figuretrue 

\documentclass[11pt,reqno]{amsart}
\usepackage{amsmath, amssymb, amsthm, verbatim,enumerate,bbm,color,tikz} 
\usepackage{indentfirst}
\usetikzlibrary{shapes}
\usepackage{soul}
\usepackage{caption}
\usepackage{subcaption}

\title{$H$-games played on vertex sets of random graphs}

\author{Gal Kronenberg}
\thanks{School of Mathematical Sciences, Raymond and Beverly Sackler Faculty of Exact Sciences, Tel Aviv University, Tel Aviv, 6997801, Israel. Email: galkrone@mail.tau.ac.il.}
\author {Adva Mond}
\thanks{School of Mathematical Sciences, Raymond and Beverly Sackler Faculty of Exact Sciences, Tel Aviv University, Tel Aviv, 6997801, Israel. Email: advamond@mail.tau.ac.il.}
\author {Alon Naor}
\thanks{School of Mathematical Sciences, Raymond and Beverly Sackler Faculty of Exact Sciences, Tel Aviv University, Tel Aviv, 6997801, Israel. Email: alonnaor@tau.ac.il.}

\date{\today}
\allowdisplaybreaks

\makeatletter
\def\subsection{\@startsection{subsection}{2}%
  \z@{.5\linespacing\@plus.7\linespacing}{.3\linespacing}%
  {\normalfont\bfseries}}
\makeatother

\theoremstyle{plain}
\newtheorem{ourtheorem}{Theorem}[]
\newtheorem{ourcorollary}[ourtheorem]{Corollary}
\newtheorem{ourconjecture}[ourtheorem]{Conjecture}
\newtheorem{theorem}{Theorem}[section]
\newtheorem{lemma}[theorem]{Lemma}
\newtheorem{claim}[theorem]{Claim}
\newtheorem{proposition}[theorem]{Proposition}
\newtheorem{observation}[theorem]{Observation}
\newtheorem{corollary}[theorem]{Corollary}

\newtheorem{question}[theorem]{Question}

\theoremstyle{definition}
\newtheorem{remark}[theorem]{Remark}
\newtheorem{definition}[theorem]{Definition}

\newcommand{\tmi}{T_{\min}^{(i)}}
\newcommand{\tmij}{T_{\min}^{(i,j)}}
\newcommand{\ttt}{TT}
\newcommand{\ssi}[1]{\overrightarrow{s_{#1}}}
\newcommand{\sss}{\overrightarrow{s}}
\newcommand{\yqtss}{Y_{q,t,\sss}}
\newcommand{\xqtss}{X_{q,t,\sss}}
\newcommand{\xqts}{X_{q,t,s}}

\newcommand{\yqts}{Y_{q,t,s}}

\newcommand{\zqts}{Z_{q,t,s}}
\newcommand{\z}[2]{Z_{#1,t,#2}}
\newcommand{\lnn}{\ln n}

\newcommand{\GG}{\Gamma}
\newcommand{\LL}{\Lambda}
\newcommand{\gnp}{G(n,p)}
\newcommand{\ggn}[1]{G \sim G(n,#1)}
\newcommand{\ggnp}{\ggn{p}}
\newcommand{\whp}{w.h.p.\ }
\newcommand{\prgg}{\Pr[\GG \sbst G]}

\newcommand{\md}[2]{\equiv #1~(\mathrm{mod}~#2)}

\newcommand{\A}{\mathcal A}
\newcommand{\B}{\mathcal B}
\newcommand{\cH}{\mathcal H}
\newcommand{\cP}{\mathcal P}
\newcommand{\F}{\mathcal F}
\newcommand{\G}{\mathcal G}
\newcommand{\E}{\mathcal E}
\newcommand{\W}{\mathcal W}
\newcommand{\D}{\mathcal D}
\newcommand{\C}{\mathcal C}

\newcommand{\T}{\mathcal T}
\newcommand{\cS}{\mathcal S}

\newcommand{\eps}{\varepsilon}

\newcommand{\boldm}[1] {\mathversion{bold}$#1$\mathversion{normal}}
\newcommand{\case}[2] {~\\ \noindent\textbf{Case \boldm{#1}: \boldm{#2}.\\*}}

\newcommand{\stm}{\setminus}
\newcommand{\sbst}{\subseteq}

\begin{document}

\begin{abstract}
We introduce a new type of positional games, played on a vertex set of a graph. Given a graph $G$, two players claim vertices of $G$, where the outcome of the game is determined by  the subgraphs of $G$ induced by the vertices claimed by each player (or by one of them). We study classical positional games such as Maker-Breaker, Avoider-Enforcer, Waiter-Client and Client-Waiter games, where the board of the game is the vertex set of the binomial random graph $G\sim G(n,p)$. Under these settings, we consider those games where the target sets are the vertex sets of all graphs containing a copy of a fixed graph $H$, called $H$-games, and focus on those cases where $H$ is a clique or a cycle.
We show that, similarly to the edge version of $H$-games, there is a strong connection between the threshold probability for these games and the one for the corresponding vertex Ramsey property (that is, the property that every $r$-vertex-coloring of $G(n,p)$ spans a monochromatic copy of $H$). Another similarity to the edge version of these games we demonstrate, is that the games in which $H$ is a triangle or a forest present a different behavior compared to the general case.
\end{abstract}

\maketitle

\section{Introduction}\label{sec:intro}
Positional games are finite, perfect information games with no chance moves, played by two players $\A$ and $\B$ (which usually have more informative names, in correspondence to the particular game in discussion). In its most general form, a  positional game is a 4-tuple $(a,b,X,\mathcal F)$, where $a$ and $b$ are two positive integers (called the \emph{bias} of the players), $X$ is a finite set (called the \textit{board}) and ${\mathcal F} \subseteq 2^X$ is a
family of subsets of $X$. The pair $(X,\mathcal F)$ is referred to as the \textit{hypergraph of the game}, and every member in $\mathcal F$ is called a \textit{target set} (according to the game in discussion we sometimes refer to the target sets as either \textit{winning sets} or \textit{losing sets}). The definition of the game is complete by identifying the first player to move (when this is relevant) and by specifying the winning criteria in the game.

The course of the game goes as follows: the two players alternately claim previously unclaimed elements of the board (each such element is called \emph{free}), until there are none left. In each round, $\A$ claims $a$ elements, and $\B$ claims $b$ elements. The last player to play may claim fewer elements than his bias, if not enough free elements remain. The most basic case is $a=b=1$, the so-called \emph{unbiased} game, while for all other choices of $a$ and $b$, the game is called \emph{biased}. Positional games have drawn much attention in the past decade, and numerous papers investigating them have been published. We refer the reader to the extremely thorough book on the subject by Beck~\cite{BeckBook}, and to the more recent book by Hefetz, Krivelevich, Stojakovi\'c and Szab\'o~\cite{HKSSbook}.

The positional games discussed in this paper cannot end in a draw. Hence, and given the nature of positional games in general, each such game must satisfy exactly one of the following: either $\A$ has a strategy to ensure his win (which works against \textbf{any} strategy of $\B$), or $\B$ has such strategy. Thus, we may (and systematically do) refer to every given positional game as either $\A$'s win or $\B$'s win.

It is natural to play positional games on the edge set of a graph $G$. In this case, $X = E(G)$ and $\F$ consists of all edge sets of subgraphs of $G$ satisfying some monotone increasing graph property. Such a property could be, for example, ``being connected and spanning" (the \emph{connectivity} game), ``containing a perfect matching" (the \emph{perfect matching} game), ``containing a Hamilton cycle" (the \emph{Hamiltonicity} game), ``containing a copy of a predetermined fixed graph $H$", and so on. The latter family of games is called $H$-games, and they are the subject of research of this paper.

It turns out that when considering $H$-games, different types of density parameters of graphs are crucial for the analysis. Thus, before continuing with the description of positional games, we define these density parameters in order to later state known and new results for $H$-games. Here, and for the remainder of this paper, for any given graph $H$ we use the standard notation $v(H) := |V(H)|$ and $e(H) := |E(H)|$.

\begin{definition}\label{def:densities}
For a graph $H$, the parameter $d(H) := e(H)/v(H)$ is called the \emph{density} of $H$, and $m(H) := \max\{d(H') \mid H'\subseteq H\}$ is the \emph{maximum density} of $H$. A graph $H$ is called \emph{strictly balanced} if $d(H) > d(H')$ for every $H' \subsetneq H$.

For $i\in \{1,2\}$, if $v(H)\geq i+1$ we define $d_i(H) := \frac{e(H)-i+1}{v(H)-i}$ to be the \emph{$i$-density} of $H$. Similarly to the previous definitions, we define $m_i(H) := \max\{d_i(H') \mid H'\subseteq H,\ v(H')\geq i+1 \}$ to be the \emph{maximum $i$-density} of $H$, and a graph $H$ is called \emph{strictly $i$-balanced} if $d_i(H) > d_i(H')$ for every $H' \subsetneq H$.
%
\end{definition}

Let us now define the main type of positional games we investigate in this paper, called \textit{Maker-Breaker games}. In the $(a:b)$ Maker-Breaker $(X,\mathcal F)$ game, the two players --- who are now called Maker and Breaker --- take turns in claiming the elements of $X$. In every round Maker claims $a$ board elements and Breaker claims $b$ elements. Maker wins the game if he occupies all elements of some target set (a winning set in this case) by the end of the game; if he fails to do so, Breaker wins the game (so indeed, a draw is not possible).

Maker-Breaker games in general, and especially those who are played on graphs, are probably the most studied family of positional games. The most natural choice for the graph whose edges the players claim is $K_n$, the complete graph on $n$ vertices. As it turns out, many natural games played on it, such as the connectivity, perfect matching and Hamiltonicity games, are drastically in favor of Maker: he wins these games in their unbiased version in (almost) the minimal number of moves required to fully claim a winning set (for more details and for similar results, see~\cite{HKSS2009b,HS,Lehman}). Therefore, in order to even out the odds and make these games more interesting, two main approaches are considered, many times simultaneously.

The first is to give Breaker a larger bias than that of Maker, and typically the $(1:b)$ version is considered. An important (and quite easy for observation) property of Maker-Breaker games is that they are \emph{bias monotone}: if Maker wins some game $\F$ with bias $(a:b)$, he also wins this game with bias $(a':b')$, for every $a' \geq a$ and $b' \leq b$. In other words, no player can be harmed by claiming more elements per move. This bias monotonicity enables the definition of the \emph{threshold bias}: for a given hypergraph $\F$, it is the unique integer $b^*$ for which Maker wins the $(1:b)$ game $\F$ if and only if $b < b^*$. For example, it was shown~\cite{CE,GS,K} that for the connectivity, perfect matching and Hamiltonicity games played on $K_n$, the threshold bias is $(1+o(1))n/\ln n$. Bednarska and  \L uczak analyzed $H$-games in~\cite{BL} and showed that for any graph $H$ (under the technical assumption that $H$ contains at least two edges), the threshold bias for the $H$-game played on $K_n$ satisfies $b^* = \Theta(n^{1/m_2(H)})$.

The other main approach towards balancing Maker-Breaker games is to play on sparse graphs. Here the typical case study is that of random graphs, and specifically the binomial random graph $G\sim G(n,p)$, a graph on $n$ vertices where each of the $\binom{n}{2}$ potential edges is included with probability $p$, independently of all other edges. It is well known by the seminal (and more general) result of Bollob\'as and Thomason~\cite{BT} that every monotone increasing graph property $\cP$ has a \emph{threshold probability}. That is, a function $p^* = p^*(n)$ which satisfies
\begin{equation*}
\lim_{n\to \infty} \Pr\left[G\sim G(n,p) \in \cP \right] =
\begin{cases}
1 &p=\omega(p^*),\\
0 &p=o(p^*).
\end{cases}
\end{equation*}

%
%
Now, given a monotone increasing graph property and the bias of the players, it is easy to see that ``being Maker's win" is a monotone increasing graph property as well (Maker's winning strategy for a graph $G$ is applicable to any graph containing $G$ on the same vertex set). This allows us to consider the \emph{threshold probability of the game}, i.e., look for the turning point of the game, where $\ggnp$ goes through a phase transition, from being \whp (with high probability, that is, with probability tending to 1 as $n$ tends to infinity) Breaker's win to being \whp Maker's win.

The first to study positional games on random graphs were Stojakovi\'c and Szab\'o \cite{PGonRandomGraphs}, who analyzed Maker-Breaker games played on $E(G)$ where $G\sim G(n,p)$. In that paper they investigated the threshold probability of the unbiased connectivity, perfect matching, Hamiltonicity, and $K_k$ games. They also considered the biased $(1:b)$ versions of these games, and provided bounds for the threshold bias $b^*$ as a function of $p$. Since then, much progress has been made in understanding positional games played on the edge set of $G\sim G(n,p)$ (see, e.g., \cite{BFHK,CFKL,DK,FGKN,MullerStojakovic,HgameRandomGraphs}). In particular, the study of Maker-Breaker $H$-games in this setting was continued by M\"uller and Stojakovi\'c  \cite{MullerStojakovic}, who found the threshold probability for the unbiased $K_k$-game where $k \ge 4$, by giving a lower bound on the threshold probability matching the upper bound given in~\cite{PGonRandomGraphs}. For the $K_3$-game they provided a \emph{hitting time} result, thus achieving a better understanding of this game, whose threshold probability was already determined in~\cite{PGonRandomGraphs}. We discuss the $K_3$-game and the meaning of hitting time results more thoroughly in Section~\ref{sec:GlobLoc}. In~\cite{HgameRandomGraphs}, Nenadov, Steger, and Stojakovi\'c solved the unbiased Maker-Breaker $H$-game for a large class of graphs $H$. Their main result is the following.

\begin{theorem}[Theorem 2 in \cite{HgameRandomGraphs}]\label{thm:edgeHgame}
    Let $H$ be a graph for which there exists $H'\subseteq H$ such that $d_2(H')=m_2(H)$, $H'$ is strictly 2-balanced  and it is not a tree or a triangle. Then there exist constants $c,C>0$ such that
    $$ \lim_{n\to \infty} \Pr\left[G\sim G(n,p) \text{ is Maker's win in the $(1:1)$ $H$-game} \right] = \begin{cases}
    1 &p\geq Cn^{-1/m_2(H)},\\ 0 &p\leq cn^{-1/m_2(H)}.
    \end{cases}$$
\end{theorem}


This result is very strongly correlated with a result concerning the following \emph{edge Ramsey property}: for graphs $G,H$ and an integer $r\geq 2$, let $G \to (H)^e_r$ be the property that every $r$-edge-coloring of $G$ yields a monochromatic $H$-copy. For $\ggnp$ we have the following.

\begin{theorem}[R\"odl and Ruci\'nski \cite{RR1,RR2,RR3}]\label{thm:edgeRamsey}
    Let $r\geq 2$ be an integer and let $H$ be a graph which is not a forest of stars (and not a path of length 3 if $r=2$). Let $G\sim G(n,p)$. Then there exist constants $c,C\geq 0$ such that
    $$ \lim_{n\to \infty} \Pr\left[G \to (H)^e_r \right] = \begin{cases}
    1 &p\geq Cn^{-1/m_2(H)},\\ 0 &p\leq cn^{-1/m_2(H)}.
    \end{cases}$$
\end{theorem}

The resemblance between the two theorems is not coincidental. In fact, if Maker moves first, the 1-statement of Theorem~\ref{thm:edgeHgame} can be proved almost directly from the 1-statement of Theorem~\ref{thm:edgeRamsey} (and even with the same constant $C$) by applying a simple \emph{strategy stealing} argument. Without going into much details, such an argument states that if the two players have the same bias, the first player can mimic (steal) any strategy of the second player with a slight modification, and thus can always do at least as well as the second player. In this case, since in the end of the unbiased game at least one of the players has an $H$-copy in his graph by the Ramsey property, Maker as a first player can ensure his graph does, and win (the authors of~\cite{HgameRandomGraphs}, however, deduced the 1-statement of Theorem~\ref{thm:edgeHgame} from some stronger claim they have in the paper).

\subsection{Our setting and first results}
Naturally, an equivalent of the above edge Ramsey property is the following \textit{vertex Ramsey property}: for graphs $G$ and $H$, and an integer $r$, let $G \to (H)^v_r$ be the property that every $r$-vertex-coloring of $G$ yields a monochromatic $H$-copy. For $\ggnp$ we have the analogue of Theorem~\ref{thm:edgeRamsey}.

\begin{theorem}[Theorem 1' in~\cite{LRV}]\label{thm:LRV}
    Let $r\geq 2$ be an integer and let $H$ be a graph with at least one edge (containing a path of length 3 if $r=2$). Let $G\sim G(n,p)$. Then there exist constants $c,C\geq 0$ such that
    $$ \lim_{n\to \infty} \Pr\left[G \to (H)^v_r \right] = \begin{cases}
    1 &p\geq Cn^{-1/m_1(H)},\\ 0 &p\leq cn^{-1/m_1(H)}.
    \end{cases}$$
\end{theorem}

It is thus interesting to ask whether a \textbf{vertex} version of Maker-Breaker games on random graphs presents a similar behavior to the vertex Ramsey property, in the same way the \textbf{edge} version of the game resembles the edge Ramsey property.

First, one has to define what a vertex version of the game would be. In this paper, we suggest the following setting. In a Maker-Breaker game on $V(G)$, the players alternately claim vertices of a graph $G$ according to their bias. For a graph property $\mathcal P$, Maker wins the game if the subgraph of $G$ induced by his vertices satisfies $\mathcal P$, otherwise Breaker wins. Note that in this setting, playing the game on the vertex set of a sparse graph is a very natural choice, since the case $G=K_n$ --- which is usually the most basic choice when the players claim edges --- is completely trivial: Maker's graph in the end of the game is always a clique on $n/(a+b)$ vertices, no matter how the players play.

It is important to note that for a fixed graph $H$, the \emph{vertex $H$-game} (that is, the $H$-game played on the vertices of a graph) is bias monotone, as claiming more vertices cannot harm Maker. This is not necessarily the case for other vertex games, see Section~\ref{sec:concluding} for more details. Furthermore, given a graph $H$ and an integer $b\geq 1$, ``being Maker's win in the $(1:b)$ vertex $H$-game" is a monotone increasing graph property. Thus we can study the threshold function for this game, namely the function $p^* = p^*(n,b,H)$ that satisfies
$$ \lim_{n\to \infty} \Pr\left[G\sim G(n,p) \text{ is Maker's win in the $(1:b)$ vertex $H$-game} \right] = \begin{cases}
1 &p=\omega(p^*),\\ 0 &p=o(p^*).
\end{cases}$$

For the remainder of this paper, $p^*_{b,H}$ stands for the threshold probability of being Maker's win in the $(1:b)$ vertex $H$-game played on $\ggnp$; we abbreviate to $p^*$ when there is no risk of confusion. Our main result in this paper is that the $(1:b)$ vertex $H$-game is indeed correlated with the aforementioned vertex Ramsey property whenever a subgraph of $H$ of maximal 1-density is either a clique or a cycle, where the only exception is that this subgraph is a triangle and the game is unbiased.

\begin{ourtheorem}\label{thm:main}
    Let $k,b$ be positive integers such that either $k\geq 4$, or $k=3$ and $b\geq 2$.
    Let $H$ be a graph for which there exists $H'\subseteq H$ such that $d_1(H')=m_1(H)$, and either
    $H'=K_k$ or $H'=C_k$. Then there exist constants $c,C > 0$ such that
    $$ \lim_{n\to \infty} \Pr\left[G\sim G(n,p) \text{ is Maker's win in the $(1:b)$ vertex $H$-game}\right] = \begin{cases}
    1 &p\geq Cn^{-1/m_1(H)},\\ 0 &p\leq cn^{-1/m_1(H)}.
    \end{cases}$$
\end{ourtheorem}


As mentioned, the unbiased game for $H'=K_3$ is excluded from Theorem~\ref{thm:main}. The reason is that if $H=K_3$, Maker wins the $(1:1)$ game playing only on the vertices of a certain fixed graph, which appears in $\ggnp$ \whp for much smaller edge densities than those given in Theorem~\ref{thm:main}. In particular, show that $p^*_{1,K_3} = n^{-7/10} = o(n^{-1/m_1(K_3)})$, as $m_1(K_3) = 3/2$. We provide full details about this case in the next subsection.

Similarly to the edge version of the game, if the game is unbiased and Maker moves first, the 1-statement of the theorem follows from the 1-statement of Theorem~\ref{thm:LRV} by applying strategy stealing. However, in order to prove the 1-statement of the theorem in its full generality we need something stronger.
\begin{theorem}\label{thm:RamseyVertex}
	 Let $r\geq 2$ and let $H$ {be a graph with at least one edge (containing a path of length 3 if $r=2$)}. Then  there exists a constant $C > 0$ such that for $G\sim G(n,p)$, if $p\geq Cn^{-1/m_1(H)}$, then \whp  every subset of $V(G)$ of size $\lfloor n/r \rfloor$ spans an $H$-copy.
\end{theorem}


The proof of Theorem~\ref{thm:RamseyVertex} follows easily from the proof of Theorem~\ref{thm:LRV} in \cite{LRV}; we omit the straightforward details. In Section~\ref{sec:constants} we go into some more details as we better estimate the constant $C$ from the theorem in case $H$ is a clique. In any case, the proof of Maker's side in Theorem~\ref{thm:main} is now immediate.

\begin{proof}[Proof of the 1-statement of Theorem~\ref{thm:main}]
Consider a $(1:b)$ Maker-Breaker vertex $H$-game played on $\ggnp$. In the end of the game Maker's graph is spanned by a $\frac{1}{b+1}$-fraction of the vertices, no matter how he plays. By Theorem~\ref{thm:RamseyVertex}, if $p \ge Cn^{-1/m_1(H)}$, where $C$ is the constant from the theorem corresponding to $H$ and $r = b+1$, then \whp Maker's graph contains an $H$-copy.
\end{proof}

Note that in fact we got that the 1-statement of Theorem~\ref{thm:main} holds for \textbf{any} fixed graph $H$ and not only those specified in the theorem (if $H$ does not meet the requirements of Theorem~\ref{thm:RamseyVertex} then Maker's win is trivial). The proof of the 0-statement of Theorem~\ref{thm:main} will be presented in Sections~\ref{sec:clique} (cliques) and~\ref{sec:cycle} (cycles).

In contrast to the graphs specified in Theorem~\ref{thm:main}, the correlation between the Ramsey property and the game is not maintained in case of forests. Indeed, if $H$ is a forest, then $m_1(H) = 1$. Theorem~\ref{thm:LRV} therefore implies that the threshold function for the corresponding vertex Ramsey property is $p = 1/n$. However, the following theorem shows that the order of magnitude of the threshold function for the vertex $H$-game is significantly smaller.

%

\begin{ourtheorem}\label{thm:MBforest}
Let $H$ be a forest consisting of trees $T_1,\dots,T_k$, and let $b$ be a positive integer.
\begin{enumerate}
\item If $H$ is a tree, i.e. $k=1$, then there exists a tree $T$ such that Maker wins the $(1:b)$ $H$-game played on $V(T)$.
\item For any integer $k \ge 1$, and for every $1 \le i \le k$, let  $T_{\min}^{(i)}$ be a tree of minimal size such that Maker, as a first player, wins when playing the $(1:b)$ $T_i$-game on its vertex set. Let $T_{\max}$ be a tree of maximal size among all trees $T_{\min}^{(i)}$. Then $p^*_{b,H} = n^{-1/m(T_{\max})} = n^{-v(T_{\max})/e(T_{\max})}$.
\end{enumerate}
\end{ourtheorem}

%
We conjecture that as in the edge version of the game, forests and triangles as subgraphs of maximum 1-density (or 2-density in the edge version) are the only exceptions for the very strong connection between the game and the Ramsey property.

\begin{ourconjecture}\label{conj:main}
Let $b\geq 1$ be an integer, and let $H$ be a graph for which there exists $H'\subseteq H$ such that $d_1(H')=m_1(H)$, $H'$ is strictly 1-balanced and is not a single edge, and in case $b=1$ also not a triangle.
Then there exist constants $c,C>0$ such that
$$ \lim_{n\to \infty} \Pr\left[G\sim G(n,p) \text{ is Maker's win in the $(1:b)$ vertex $H$-game}\right] = \begin{cases}
1 &p\geq Cn^{-1/m_1(H)},\\ 0 &p\leq cn^{-1/m_1(H)}.
\end{cases}$$
\end{ourconjecture}

\begin{remark}\label{rmk:MakerSideConj}
    The 1-statement of Conjecture~\ref{conj:main} follows immediately from Theorem~\ref{thm:RamseyVertex}, as shown in the proof of the 1-statement of Theorem~\ref{thm:main}. Thus the interesting part of the conjecture is Breaker's side.
\end{remark}

Note that there are two types of strictly 1-balanced graphs not covered by Conjecture~\ref{conj:main}. In the terminology of the conjecture, observe that $H'$ is an edge if and only if $H$ is a forest, so this case is covered by Theorem~\ref{thm:MBforest}. In the next subsection we consider the unbiased triangle game.
It remains to deal with unbiased $H$-games where $H \neq K_3$ is a graph that satisfies $m_1(H)  = 3/2$, and every strictly $1$-balanced subgraph $H'\sbst H$ with $d_1(H') = 3/2$ is a triangle. Let $\mathcal H$ be the family of all such graphs and let $H\in \mathcal H$. By Corollary~\ref{cor:triangle11} (in the next subsection), if $p = o\left(n^{-7/10}\right)$, then \whp Breaker can prevent Maker from claiming a triangle, thus winning the $H$-game. This implies that $p^*_{1,H} = \Omega\left(n^{-7/10}\right)$. On the other hand, the fact that $p^*_{1,H} = O\left(n^{-2/3}\right)$ is an immediate corollary of Theorem~\ref{thm:RamseyVertex}. We show that if Conjecture~\ref{conj:main} is confirmed, then there exist infinitely many rational values $\alpha\in \left(\frac{10}{7}, \frac32 \right)$ for which there exists a graph $H\in \mathcal H$ such that $p^*_{1,H} = n^{-1/\alpha}$. We actually prove something more general.

%

\begin{ourtheorem}\label{thm:infamily}
Let $H'$ be a graph for which there exists $H''\subseteq H'$ such that $d_1(H'')=m_1(H')=\alpha$, $H''$ is strictly 1-balanced and is not a triangle.
Let $v\in V(H')$ and let $H$ be the graph obtained by connecting $v$ to a triangle via a path of length 4 (see Figure~\ref{fig:infFamily}).

\begin{itemize}
\item In case $\alpha > 10/7$ there exist positive constants $c,C$ such that
$$ \lim_{n\to \infty} \Pr\left[G\sim G(n,p) \text{ is Maker's win in the $(1:1)$ vertex $H$-game}\right] = \begin{cases}
1 &p\geq Cn^{-1/\alpha},\\ 0 &p\leq cn^{-1/\alpha},
\end{cases}$$
where the 0-statement holds if the 0-statement of Conjecture~\ref{conj:main} holds (and for the same $c$).
\item In case $\alpha \le 10/7$ we have
$$ \lim_{n\to \infty} \Pr\left[G\sim G(n,p) \text{ is Maker's win in the $(1:1)$ vertex $H$-game}\right] = \begin{cases}
1 &p = \omega(n^{-7/10}),\\ 0 &p = o(n^{-7/10}).
\end{cases}$$
\end{itemize}
\end{ourtheorem}

Note that the most interesting case is where $\alpha \in \left(\frac{10}{7}, \frac32 \right)$, as it shows that there exists an infinite family of graphs for
which the threshold probability of the game is not determined by a subgraph on which the maximum 1-density is obtained (which in this case is a triangle).
This is the equivalent of the vertex phenomenon demonstrated in Theorem 4 of~\cite{HgameRandomGraphs}.

%
%
%

\subsection{Global vs.~local and the random graph process}\label{sec:GlobLoc}
%
In this subsection we consider a somewhat different model for the random graph on which the game is played. For an integer $n$ let $[n]=\{1, \ldots, n \}$ and $m = \binom n2$, denote the set of edges of $K_n$ by $e_1,\dots,e_m$, and let $\pi \in S_m$ be an arbitrary permutation of $[m]$. For $G_i=([n],\{e_{\pi(1)},\dots,e_{\pi(i)}\})$, the increasing sequence of graphs $\tilde G = \{G_i \}_{i=0}^{m}$ is called a graph process.  The \emph{random graph process} is the graph process obtained by choosing $\pi$ uniformly at random from all possible permutations. This random setting generates a random graph model that is closely related to the standard random graph model $G(n,p)$ we have considered so far (see e.g.\ \cite{JLR}, and more discussion in Section~\ref{sec:pre}).

For a given graph process $\tilde G$, the \emph{hitting time} of a monotone increasing graph property $\mathcal P$ is defined to be $\tau (\tilde G, \mathcal P) = \min\{i ~|~ G_i \in \mathcal P \}$. We would like to examine the hitting time of the property ``being Maker's win in the $(1:1)$ vertex $K_3$-game". It turns out that \whp the graph becomes Maker's win at the same moment a certain fixed graph appears in $G$ for the first time. Before stating the result formally we need to describe this graph and to introduce new notation.

\begin{definition}\label{def:DD}
A \emph{diamond} is a $K_4$-copy with one edge missing. A \emph{Double Diamond}, denoted by $DD$,  consists of two diamonds with vertex sets $\{x,y_1,z_1,z_2\}$ and $\{x,y_2,z_3,z_4\}$ where $z_1z_2$ and $z_3z_4$ are the missing edges (see Figure~\ref{fig:DD}). The intersection of the two diamonds, namely the vertex $x$, is called \emph{the center} of $DD$.
\end{definition}

%
%
%

\iffigure
\begin{figure}
	\centering
	\begin{minipage}{.6\textwidth}
		\centering
		\begin{tikzpicture}[auto, vertex/.style={circle,draw=black!100,fill=black!100, thick,
			inner sep=0pt,minimum size=1mm}]
		\node (a) at (-3,0) [vertex] {};
		\node (b) at (-2.5,0.866) [vertex] {};
		\node (x1) at (-2,0) [vertex] {};
		\node (x2) at (-1,0) [vertex] {};
		\node (x3) at (0,0) [vertex] {};
		\node (x4) at (1,0) [vertex] {};
		\node (x5) at (2,0) [vertex,label=right:$v$] {};
		\node (c) at (3.52,0) [cloud, draw,cloud puffs=10,cloud puff arc=120, aspect=2, inner ysep=1em] {$H'$};
		
		\draw [-] (a) --node[inner sep=0pt,swap]{} (b);
		\draw [-] (a) --node[inner sep=2pt,swap]{} (x1);
		\draw [-] (b) --node[inner sep=0pt,swap]{} (x1);
		\draw [-] (x1) --node[inner sep=0pt,swap]{} (x2);
		\draw [-] (x2) --node[inner sep=0pt,swap]{} (x3);
		\draw [-] (x3) --node[inner sep=0pt,swap]{} (x4);
		\draw [-] (x4) --node[inner sep=0pt,swap]{} (x5);
		\end{tikzpicture}
		\caption{The construction of $H$ from $H'$}\label{fig:infFamily}
	\end{minipage}%
	\begin{minipage}{.4\textwidth}
		\centering
		\begin{tikzpicture}[auto, vertex/.style={circle,draw=black!100,fill=black!100, thick,
			inner sep=0pt,minimum size=1mm}]
		\node (z1) at ( 0.5,0.866) [vertex,label=above:$z_1$] {};
		\node (y1) at ( 0,0) [vertex,label=left:$y_1$] {};
		\node (x) at ( 1,0) [vertex,label=below:$x$] {};
		\node (z2) at ( 0.5,-0.866) [vertex,label=below:$z_2$] {};
		\node (z3) at ( 1.5,0.866) [vertex,label=above:$z_3$] {};
		\node (y2) at ( 2,0) [vertex,label=right:$y_2$] {};
		\node (z4) at ( 1.5,-0.866) [vertex,label=below:$z_4$] {};
		
		\draw [-] (z1) --node[inner sep=0pt,swap]{} (y1);
		\draw [-] (y1) --node[inner sep=2pt,swap]{} (x);
		\draw [-] (z1) --node[inner sep=0pt,swap]{} (x);
		\draw [-] (y1) --node[inner sep=0pt,swap]{} (z2);
		\draw [-] (z2) --node[inner sep=0pt,swap]{} (x);
		\draw [-] (z3) --node[inner sep=0pt,swap]{} (x);
		\draw [-] (z4) --node[inner sep=0pt,swap]{} (x);
		\draw [-] (y2) --node[inner sep=0pt,swap]{} (x);
		\draw [-] (z3) --node[inner sep=0pt,swap]{} (y2);
		\draw [-] (z4) --node[inner sep=0pt,swap]{} (y2);
		\end{tikzpicture}
		\caption{The graph $DD$}\label{fig:DD}
	\end{minipage}
\end{figure}

\fi

Throughout the paper we use the following notation: for an integer $k$ and a fixed graph $H$, let $\G_{kH}$ denote the graph property of containing $k$ (possibly intersecting) copies of $H$. We abbreviate $\G_{1H}$ to $\G_H$.


\begin{ourtheorem}\label{thm:triangle11}
    Let $\mathcal M^1_{K_3}$ and $\mathcal M^2_{K_3}$ be the graph properties of being Maker's win in the $(1:1)$ vertex $K_3$-game, where Maker moves first or second, respectively. For a random graph process $\tilde G$, \whp $\tau(\tilde G,\mathcal M^1_{K_3}) = \tau(\tilde G, \mathcal G_{DD})$ and $\tau(\tilde G,\mathcal M^2_{K_3}) = \tau(\tilde G, \mathcal G_{2DD})$.
\end{ourtheorem}

The following corollary of Theorem~\ref{thm:triangle11} is due to the asymptotic connection between $G(n, p)$ and the random graph process, and the distribution of the number of $H$-copies in the binomial graph. See Proposition~\ref{prop:GnpRGP} and Theorem~\ref{thm:fixedGraphDist} in Section~\ref{sec:pre}.

\begin{ourcorollary}\label{cor:triangle11}
    Let $p=p(n)$ and let $x=np^{10/7}$. Then assuming Maker moves first we have
    $$\lim_{n\to \infty}\Pr\left[G\sim G(n,p )\text{ is Maker's win in the $(1:1)$ vertex $K_3$-game} \right] = \begin{cases}
    0 & x\to 0,\\
    1-e^{-c^7/8} & x\to c \in \mathbb{R}^+,\\
    1 & x\to \infty .
    \end{cases}
    $$
    In particular, the game has a threshold at $p = n^{-7/10}$.
\end{ourcorollary}

Recall that the proof of Maker's side in Theorem~\ref{thm:main} and Conjecture~\ref{conj:main} is trivial: the graph is such that Maker wins no matter how he plays. We say that in these cases Maker wins due to a \emph{global} reason, that is, the structure of the entire graph. This stands in sharp contrast to Maker's side in the vertex $H$-game when $H$ is a forest or a triangle, as shown in Theorems~\ref{thm:MBforest} and~\ref{thm:triangle11}, where Maker applies a straightforward winning strategy on some small, fixed graph $\hat H$. We say that in these cases Maker wins due to a \emph{local} reason, that is, the appearance of $\hat H$ in the random graph.
The question is what is more likely to appear first in the random graph process --- a local reason or a global reason. In this paper we show that when $H$ is either a triangle or a forest, \whp the local reason appears first, and conjecture that these are the only cases (we actually show that a triangle is an exception only in the unbiased game).

The exact same phenomenon --- of Maker winning globally unless $H$ is either a forest or a triangle --- was proven for the unbiased edge version of $H$-games. Theorem~\ref{thm:edgeHgame} shows that the threshold for most graphs is the one matching the global reason, that is, $n^{-1/m_2(H)}$ (this was already shown for cliques in~\cite{PGonRandomGraphs}). The case that $H$ is a forest was considered by Stojakovi\'{c} in~\cite[Lemma 36]{Sthesis}, where he showed (and by that inspired our Theorem~\ref{thm:MBforest}) that in this case Maker wins due to a local reason, and so there exists some constant $\alpha(H) > 1$ such that the threshold function of the game is $n^{-\alpha(H)}$, while $m_2(H) = 1$ (in fact, he only considered trees, but his result may be easily generalized to forests, see Remark~\ref{rmk:milos} in Section~\ref{sec:forests}). The case $H = K_3$ was first considered in~\cite{PGonRandomGraphs}, where the threshold function was shown to be $n^{-5/9}$, since Maker wins locally on a copy of $K_5^-$, the clique on five vertices with one edge missing. Later, in~\cite{MullerStojakovic}, this result was improved to a hitting time result, that is, assuming Maker moves first, in the random graph process \whp the graph becomes his win at the same moment the first copy of $K_5^-$ appears. This is of course the equivalent of Theorem~\ref{thm:triangle11}.

\subsection{Avoider-Enforcer games}
We move on to a different type of positional games, called Avoider-Enforcer games, which are the mis\`{e}re version of Maker-Breaker games. An $(a,b,X,\F)$ Avoider-Enforcer game is played in the same manner as the corresponding Maker-Breaker game: the two players, called Avoider and Enforcer,
alternately claim $a$ and $b$ free elements of $X$ per move, respectively. The difference is that the target sets are now losing sets, and so at the end of the game Avoider loses if he has fully claimed some $F \in \F$, and wins otherwise.

Despite being closely related to Maker-Breaker games,
Avoider-Enforcer games are unfortunately (and perhaps surprisingly)
not bias monotone in general (see e.g.\ \cite{HKSS10},\cite{HKS07}).
Even though one may assume intuitively that no player can be harmed
by claiming fewer elements per move, this is not always the case.
This behavior motivated Hefetz, Krivelevich, Stojakovi\'c and
Szab\'o to propose in~\cite{HKSS10} a bias monotone version for
Avoider-Enforcer games: in the new version Avoider and Enforcer
claim \textbf{at least} $a$ and $b$ elements per move, respectively.
It is easy to see that this new version is indeed bias monotone, and
no player can be harmed from lowering his bias. As it turns out,
this change of rules may change dramatically the outcome of the
game. We refer to the traditional and new sets of rules as the
\emph{strict} and \emph{monotone} rules, respectively, and
accordingly refer to either strict games or monotone games. For every monotone game there exists a threshold bias, defined in a similar way to that of Maker-Breaker games. However, for strict games it is only possible to define lower and upper threshold biases. We do not elaborate on that. For more
information about the differences between the two sets of rules and
about Avoider-Enforcer games in general, see for
example~\cite{BeckBook,HKSS10,HKSSbook,HKS07}.

Avoider-Enforcer games are much less studied than Maker-Breaker games, certainly when considering
$H$-games or games played on random boards as we do in this paper.
Some specific $H$-games were analyzed thoroughly in~\cite{AEstars,HKSS10}, and a more general investigation of $H$-games was performed in~\cite{BedAE,BBGT,FKN}.
In all these papers the $(1:b)$ game played on $K_n$ was considered, with the intention to find the values of the different types of threshold biases.
To the best of our knowledge, Avoider-Enforcer games played on random graphs were only considered
in~\cite{FGKN}, where the authors analyzed the $k$-connectivity, Hamiltonicity, and perfect matching monotone biased games played on the edge set of $\ggnp$, determining the threshold bias as a function of $p$.


It is worth mentioning that when playing on the vertex set of a
graph, ``being Enforcer's win" is trivially
a monotone increasing graph property for both strict and monotone settings (a winning strategy for Enforcer remains such if we add edges to the graph). Thus we can define the threshold
probability for these games, just as in Maker-Breaker games. Note that this is not necessarily the case when
playing the edge version of the game: in general, Avoider-Enforcer
games -- in contrast to Maker-Breaker games -- do not have
hypergraph monotonicity in the following sense. It is possible for Enforcer to win an $(a:b)$ game  $(X, \mathcal F)$, but lose an $(a:b)$ game $(X',\mathcal F')$, even if $X \sbst X'$ and $\F \sbst \F'$. We now state our results, starting with the monotone game.

\begin{ourtheorem}\label{thm:trivialWins}
    Let $H$ be a fixed graph and let $a,b$ be two positive integers.
    Then the threshold probability for Enforcer's win in the monotone $(a:b)$ vertex $H$-game played on $\ggnp$ is $p = n^{-1/m(H)}$.
    Furthermore, if $H$ is strictly balanced, then there exists a
    constant $N = N(a,b,v(H))$, such that in the random graph process,
    $\tau(\tilde G,\E) = \tau(\tilde G,\G_{NH})$ holds w.h.p., where $\E$ denotes the property
    ``being Enforcer's win in the monotone $(a:b)$ vertex
    $H$-game".
\end{ourtheorem}

In Section~\ref{sec:others} we have a short discussion about strict $H$-games in general. However, we state an explicit result only for the
$(1:1)$ triangle-game, for two reasons. First, this is an
interesting game: the threshold probability for it when played on a
random graph is unique comparing to other $H$-games or to biased
triangle games, whether we consider the Maker-Breaker games for both
edge and vertex versions, or Client-Waiter games (which will be
introduced shortly), again for both edge and vertex versions.
Second, it turns out that this game presents an analogous behavior
to that of the Maker-Breaker game in the following way.

Recall that when the random graph process is considered, Maker wins the $(1:1)$ triangle-game as soon as the first or the second $DD$-copy appears in the graph,
depending on the identity of the \textbf{first} player. The reason
for this difference is that both players \textbf{wish to claim} the
center of a $DD$-copy, and so if only one copy exists the first player wins.
Analogously, Enforcer wins in the $(1:1)$ triangle Avoider-Enforcer
game as soon as the first or the second $DD$-copy appears in the
graph, depending on the identity of the \textbf{last} player to
play (that is, the player who claims the last free vertex in the game; not to be confused with the player who plays second). Here both players \textbf{wish to avoid claiming} the center of
a $DD$-copy, and so if only one copy exists the last player to move loses.
It is important to notice that unlike the edge version of positional game
played on random graphs, the identity of the last player is
determined by the identity of the first player and the number of
vertices in the graph, which are both part of the definition of the game. It does not depend on the random graph
process itself or on the number of edges in the graph in any way,
which makes the following theorem, and this whole discussion, well
defined.

\begin{ourtheorem}\label{thm:AEstrict}
    Let $\E_{K_3}^A$ and $\E_{K_3}^E$ be the graph properties of being Enforcer's win in the strict $(1:1)$ vertex $K_3$-game, where Avoider
    or Enforcer, respectively, makes the last move in the game. For a
    random graph process $\tilde G$, \whp $\tau(\tilde G,\E_{K_3}^A) = \tau(\tilde G,\mathcal G_{DD})$ and $\tau(\tilde G,\E_{K_3}^E) = \tau(\tilde G,\mathcal G_{2DD})$.
\end{ourtheorem}

\subsection{Waiter-Client and Client-Waiter games}
Waiter-Client and Client-Waiter games resemble Maker-Breaker and
Avoider-Enforcer games and were introduced by
Beck~\cite{BeckBook,Beck} under the names Picker-Chooser and
Chooser-Picker, respectively. Since the original names of the players were confusing, it is now
conventional to use the new names Waiter and Client as suggested
in~\cite{BHL}. As in Maker-Breaker and Avoider-Enforcer games, the
parameters of these games are a set $X$, a family $\F \sbst 2^X$,
and two positive integers $a$ and $b$ which denote the bias of Client and Waiter, respectively (note that $a$ denotes Client's bias even in the Waiter-Client game, which might be confusing).

The course of every round, however, is different. In an $(a:b)$ Waiter-Client game
$(X,\F)$, in every round Waiter selects $a + b$ previously unclaimed elements of
$X$. Client then chooses $a$ of those elements to claim and the
remaining elements are claimed by Waiter. For the last round of the
game, let $t \le a + b$ be the number of free elements remaining. If $t \le b$ then Waiter claims all these elements, and
otherwise Client claims $t - b$ elements and the rest go to Waiter.
Waiter wins if by the end of the game
Client has claimed all elements of some $F \in \F$, and otherwise
Client wins.
We can think of Waiter as the \emph{builder} of the game and of Client as the \emph{spoiler}.

The definition of $(a:b)$ Client-Waiter games is
similar, but with a few differences. In every round Waiter selects
$t$ free elements of $X$, where $a \le t \le a+b$, from which Client
chooses $a$ to claim, and the rest are claimed by Waiter. In the
last round Client chooses $a$ of the remaining elements to claim (and all others go to Waiter), or
he claims all of them if less than $a$ free elements remain. Client wins if by the end of the game he has
claimed all elements of some $F \in \F$, and otherwise Waiter wins.
In this game we can think of Client as the builder of the game and of Waiter as the spoiler.

The reason that Waiter may offer less than $a+b$ elements per move
in the Client-Waiter game is that otherwise the game would not be
monotone in Waiter's bias, as first observed by
Bednarska-Bzd\c{e}ga~\cite{BedCW} (the game is monotone in Client's
bias even without this relaxation). Waiter-Client games as defined
here are monotone in Waiter's bias only. Since in the study of
Waiter-Client (and Client-Waiter) games the typical case study is that the
bias of Client is 1, no similar adjustment of the rules is usually
considered (although there exists one), including in this paper.

Waiter-Client and Client-Waiter games have drawn much interest in
the last several years, resulting in quite a few papers. We do not
intend to provide a full background on this subject, as we limit
our focus to $H$-games and games on random boards. For more information on the subject we refer the reader to the papers~\cite{BedCW,BHL,DK,HKT}, and to the many other works cited in them.

It is trivial to see that whether Waiter offers exactly $a + b$
elements per move or not, in either of the games, the following
holds. For any monotone increasing graph property $\cP$, and for
both games, when playing an $(a:b)$ game $\cP$ on the vertex set of
a graph, ``being the builder's win" is a monotone increasing graph
property (note that this is not true in the edge version of the
Client-Waiter game if we do not allow Waiter to offer fewer edges
per round, which is another important motivation for this adjustment
of rules). Indeed, if $G \sbst G'$ and $V(G) = V(G')$, then in both
games a winning strategy for the builder in the game played on
$V(G)$ remains such without any changes for the game played on
$V(G')$. Once again, this important --- and non-trivial --- property allows us to define the
threshold probability for these games.

We now present some results, which demonstrate that
Waiter-Client and Client-Waiter games not only resemble
Avoider-Enforcer and Maker-Breaker games, respectively, in the roles
of the players, but also in the outcome of the corresponding
$H$-games. Theorem~\ref{thm:trivialWins} basically shows that for any $a$,$b$ and $H$, Enforcer wins the $(a:b)$ $H$-game as soon as $G$ contains sufficiently many $H$-copies. The following theorem shows that the same holds for Waiter in the Waiter-Client game.

\begin{ourtheorem}\label{thm:WC}
    Let $H$ be a fixed graph and let $a,b$ be two positive integers.
    Then the threshold probability for Waiter's win in the $(a:b)$ Waiter-Client vertex $H$-game played on $\ggnp$ is $p^* = n^{-1/m(H)}$.
    Furthermore, if $H$ is strictly balanced, then there exists a
    constant $N = N(a,b,v(H))$, such that in the random graph process,
    $\tau(\tilde G,\W) = \tau(\tilde G,\G_{NH})$ holds w.h.p., where $\W$ denotes the property
    ``being Waiter's win in the $(a:b)$ Waiter-Client vertex
    $H$-game".
\end{ourtheorem}

Moving to Client-Waiter games, we observe that they feature an almost identical behavior to that of the corresponding Maker-Breaker games, and we have the following perfect analogues of Theorems~\ref{thm:main} and~\ref{thm:MBforest}, and Conjecture~\ref{conj:main}.

\begin{ourtheorem}\label{thm:CWsimple}
    Let $k,b$ be positive integers such that either $k\geq 4$, or $k=3$ and $b\geq 2$.
    Let $H$ be a graph for which there exists $H'\subseteq H$ such that $d_1(H')=m_1(H)$, and either
    $H'=K_k$ or $H'=C_k$. Then there exist constants $c,C > 0$ such that the following holds for the $(1:b)$ Client-Waiter vertex $H$-game.
    $$ \lim_{n\to \infty} \Pr\left[G\sim G(n,p) \text{ is Client's win}\right] = \begin{cases}
    1 &p\geq Cn^{-1/m_1(H)},\\ 0 &p\leq cn^{-1/m_1(H)}.
    \end{cases}$$
\end{ourtheorem}

\begin{ourtheorem}\label{thm:CWforest}
Let $H$ be a forest consisting of trees $T_1,\dots,T_k$, and let $b$ be a positive integer. The following hold for the $(1:b)$ Client-Waiter vertex $H$-game.
\begin{enumerate}
\item If $H$ is a tree, i.e. $k=1$, then there exists a tree $T$ such that Client wins the game played on $T$.
\item For any integer $k \ge 1$, and for every $i \in [k]$, let  $T_{\min}^{(i)}$ be a tree of minimal size such that Client wins the $(1:b)$ $T_i$-game on its vertex set. Let $T_{\max}$ be a tree of maximal size among all trees $T_{\min}^{(i)}$. Then the threshold for Client's win in the $H$-game is $p^* = n^{-1/m(T_{\max})}$.
\end{enumerate}
\end{ourtheorem}

\begin{ourconjecture}\label{conj:CW}
Let $b\geq 1$ be an integer, and let $H$ be a graph for which there exists $H'\subseteq H$ such that $d_1(H')=m_1(H)$, $H'$ is strictly 1-balanced and is not a single edge, and in case $b=1$ also not a triangle.
Then there exist constants $c,C>0$ such that the following holds for the $(1:b)$ Client-Waiter vertex $H$-game.
$$ \lim_{n\to \infty} \Pr\left[G\sim G(n,p) \text{ is Client's win}\right] = \begin{cases}
1 &p\geq Cn^{-1/m_1(H)},\\ 0 &p\leq cn^{-1/m_1(H)}.
\end{cases}$$
\end{ourconjecture}

More discussion about the similarities between Client-Waiter games and Maker-Breaker games can be found in Sections~\ref{sec:gnrlH},~\ref{sec:forests} and~\ref{sec:others}.

We finish with the $(1:1)$ Client-Waiter triangle-game, where once again we observe the
phenomenon of ``the $(1:1)$ triangle-game behaves differently". This game in fact
presents the most interesting behavior of all games considered in
this paper, and, if Conjectures~\ref{conj:main} and~\ref{conj:CW} turn out to
be true, then of all vertex $H$-games.

Recall that in the unbiased \textbf{edge} version, Maker wins the
triangle game locally, but
wins due to a global reason for most other $H$-games. Dean
and Krivelevich showed in~\cite{DK} that for any given graph $H$, the threshold probability
for the Client-Waiter $H$-game is exactly the same as in the
Maker-Breaker $H$-game even when Waiter is allowed an arbitrary fixed
bias. Furthermore, they showed that not only Client wins locally the unbiased triangle
game, but also that the fixed graph for which Client awaits is the same one
Maker waits for, namely $K_5^-$.

However, in the vertex version of these games, Client and Maker need
different fixed graphs to apply their winning strategies on. While Maker
wins as soon as a double diamond appears, Client has to wait further
for the appearance of a \emph{triple diamond}, which we
describe in Section~\ref{sec:others}. In particular, unlike any
other game mentioned in this paper (for both edge and vertex
versions), in the unbiased triangle-game on the vertex set, Waiter is
significantly stronger then Breaker.

But there is more to it. As it turns out, the threshold probability for Client's ``local win" is of
the same order of magnitude as the threshold probability for
Client's ``global win".
Consequently, there is a range of values of $p$ for which all of the following occur with probability bounded away from zero in $\ggnp$: Client wins due to a local reason, Client wins due to a global reason, Waiter wins.
\begin{ourtheorem}\label{thm:CWtriangle}
There exist positive constants $c,C,\alpha$ such that in the $(1:1)$ Client-Waiter vertex $K_3$-game played on $\ggnp$ the following hold.
\begin{enumerate}[$(1)$]
\item Waiter wins \whp for any $p = o\left(n^{-2/3}\right)$;
\item For any constant $0<d < c$ and for $p = dn^{-2/3}$, we have $$\lim_{n \to \infty}\Pr[\text{Waiter wins the game}] \geq \alpha;$$
\item For any constant $d>0$ there exists a constant $\beta = \beta(d) > 0$, such that for $p = dn^{-2/3}$ we have $$\lim_{n \to \infty}\Pr[\text{Client wins the game}] \geq \beta;$$
\item Client wins \whp for any $p \ge Cn^{-2/3}$.
\end{enumerate}
\end{ourtheorem}

\subsection{Organization of the paper}
The organization of the paper is as follows. In Section~\ref{sec:pre} we provide notation and technical preliminaries, and present some random graph results. In Section~\ref{sec:gnrlH} we describe a strategy for the spoiler (either Breaker or Waiter) which can be used in any $H$-game, and additionally focus on clique games. The content of this section is the basis for the remainder of the paper. In Section~\ref{sec:clique} we deal with Maker-Breaker clique games: we first prove the 0-statement of
Theorem~\ref{thm:main} for $H=K_k$, dividing the proof into the two cases $k\geq 4$ (Section~\ref{sec:largeCliques}) and $k=3,~b\geq 2$ (Section~\ref{sec:(1:2)triangle}), and then prove Theorem~\ref{thm:triangle11} (Section~\ref{sec:(1:1)triangle}). In Section~\ref{sec:cycle} we prove the 0-statement of Theorem~\ref{thm:main} for $H=C_k,~k\geq 4$. In Section~\ref{sec:forests} we deal with forests, and in particular prove Theorems~\ref{thm:MBforest} and~\ref{thm:CWforest}. In Section~\ref{sec:infamily} we prove Theorem~\ref{thm:infamily}. Avoider-Enforcer, Waiter-Client and Client-Waiter games are all discussed in Section~\ref{sec:others}, and the proofs for all corresponding theorems are given (except for Theorem~\ref{thm:CWforest}). Finally, in Section~\ref{sec:concluding} we provide some concluding remarks and open problems.


\section{Preliminaries}\label{sec:pre}
Our graph-theoretic notation is standard and follows that of \cite{IntroGraphTheory}.
In particular we use the following. For a graph $G = (V,E)$ and a set $U\sbst V$, let $G[U]$ denote the corresponding vertex-induced subgraph of $G$, and let $N_G(U) = \{v \in V \setminus U
\mid \exists u\in U \textrm{ such that } uv \in E\}$ denote the
external neighborhood of $U$ in $G$. For a vertex $v \in V$ we
abbreviate $N_G(\{v\})$ to $N_G(v)$ and let $d_G(v) = |N_G(v)|$
denote the degree of $v$ in $G$. The maximal degree in $G$ is denoted by $\Delta(G)$. Often, when there is no risk of ambiguity, we omit the
subscript $G$ in the above notation.

Considering a fixed graph $H$, we say that a graph $G$ is \emph{$H$-free} if it does not contain a copy of $H$ as a subgraph. More generally, for a family $\mathcal F$ of fixed graphs, we say that $G$ is \emph{$\mathcal F$-free} if it is $H$-free for every $H\in \mathcal F$. We say that two $H$-copies in a graph $G$ \emph{intersect} if they are not vertex disjoint.

Our results are asymptotic in nature and we assume that $n$ is large enough where needed. We omit floor and ceiling signs whenever these are not crucial.
%

\subsection{Intersecting $H$-copies}
Let $H$ be a connected graph. We now present notation for some specific structures involving intersecting $H$-copies, which will be very useful in our proofs.

\begin{definition}\label{def:HChain}
    A graph $\Gamma$ is an \emph{$H$-chain} of length $t$ if it consists of $t\geq 1$ copies $H_1 , \ldots , H_t$ of $H$, such that
    $$\forall\ 1 \le i < j \le t:\; |V(H_i) \cap V(H_{j})| = \begin{cases} 1 &j - i = 1,\\ 0 &otherwise. \end{cases}$$
\end{definition}

Note that if $H$ is a clique then for every integer $t$ there exists exactly one $H$-chain of length $t$ (up to isomorphism), and therefore we can refer to \textbf{the} $H$-chain of length $t$ in this case.
This is not the case for any other graph, as there are different chains of each length, according to which vertices lie in the intersections between consecutive $H$-copies.

\begin{definition}\label{def:HCycle}
    A graph $\Gamma$ is an
    \emph{$H$-cycle} of length $t$ if it consists of $t \geq 3$ copies
    $H_1, \ldots, H_t$ of $H$, such that
    $$\forall\ 1 \le i,j \le t: \; |V(H_i) \cap V(H_j)| =
    \begin{cases} 1 &|j - i| \md{1}{t},\\ 0 &otherwise. \end{cases}$$
\end{definition}

\begin{definition}
    Given a graph $G$, an edge $e \in E(G)$ belonging to two distinct $H$-copies in $G$ is called \emph{dangerous} with respect to $H$. Since $H$ is always clear from the context, we simply refer to such edges as dangerous.
\end{definition}

\subsection{Strictly balanced and strictly 1-balanced graphs}
\begin{claim}\label{cl:balancedDense}
    For every strictly balanced graph $H$, if $\GG$ is a graph
    consisting of two $H$-copies with a non-empty intersection $H'$,
    then $m(\GG) > m(H) + \frac 1{2v(H)^2}$.
\end{claim}

\begin{proof}
    Let $v(H) = v$, $e(H) = e$, $v(H')= v'$ and $e(H') = e'$. Since $H' \sbst
    H$ and $H$ is strictly balanced, we get $e'/v' <
    e/v$, which implies that $ev' - e'v$ is a positive integer. It
    follows that
    \begin{equation*}
        m(\GG) - m(H) \ge d(\GG) - m(H) = \frac {2e-e'}{2v-v'} - \frac ev =
        \frac {ev' - e'v}{v(2v-v')} > \frac 1{2v^2}. \qedhere
    \end{equation*}
\end{proof}

%
%

\begin{claim}\label{cl:1balanced2connected}
Every strictly 1-balanced graph with at least three vertices is 2-vertex-connected.
\end{claim}

\begin{proof}
Let $H$ be a strictly 1-balanced graph with at least three vertices, and assume for contradiction that $H$ is not 2-vertex-connected. Then there exist two subgraphs $H_1,H_2 \sbst H$, each containing at least two vertices, such that $H = H_1 \cup H_2$ and $|V(H_1) \cap V(H_2)| = 1$. For $i=1,2$, let $v(H_i) = v_i$ and $e(H_i) = e_i$. Since $H$ is strictly 1-balanced we have $d_1(H) > d_1(H_1)$, that is
$$\frac{e_1+e_2}{v_1+v_2-2} > \frac{e_1}{v_1-1}.$$
Rearranging, we get $e_2v_1 - e_1v_2 > e_2 - e_1$. Similarly, from $d_1(H) > d_1(H_2)$ we get $e_1v_2 - e_2v_1 > e_1 - e_2$. Putting the two inequalities together, we get $e_2 - e_1 < e_2v_1 - e_1v_2 < e_2 - e_1$, a contradiction.
\end{proof}

\subsection{Random graph results}\label{sec:preGnp}
For several proofs in this paper we rely on the well-known asymptotic connection between the
random graph model $G(n,p)$ and the random graph process, given in the following proposition. Roughly
speaking, if a typical graph from one of these models satisfies some
monotone increasing graph property $\mathcal P$ with some given
probability, then a  typical graph from the other model (with the
corresponding parameters) also satisfies $\mathcal P$ with the same
probability (see \cite{Bol98,FrBook,JLR} for more details).
%
%
%
%
\begin{proposition}\label{prop:GnpRGP}
    Let $\mathcal P$ be a monotone graph property, let $p=p(n)$, $0<a<1$, let $\tilde G = (G_i)$ be a random graph process and let $G\sim G(n,p)$.
    Then there exist constants $c_1, c_2>0$ such that:
    \begin{itemize}
        \item $\Pr[G\in \mathcal P] = a ~\Rightarrow~ \Pr[G_i\in \mathcal P] = a$ for $i=c_1n^2p$.
        \item $\Pr[G_i\in \mathcal P] = a ~\Rightarrow~ \Pr[G\in \mathcal P] = a$ for $i=c_2n^2p$.
    \end{itemize}
\end{proposition}

In our proofs we often upper bound the probability that $\ggnp$ is not $H$-free (for a given fixed graph $H$), and make use of the simple inequality $\Pr[H \sbst G] \le n^{v(H)} p^{e(H)}$. The next two theorems deal with the appearance of $H$-copies in $\ggnp$, where the order of magnitude of $p$ is compared to $n^{-1/m(H)}$.

\begin{theorem}[Theorem 5.3 in \cite{FrBook}]\label{thm:HinGnp}
    For any fixed graph $H$ and for any fixed positive integer $N$, \whp
    $\ggnp$ contains $N$ vertex disjoint copies of $H$ for every $p =
    \omega(n^{-1/m(H)})$, and is $H$-free for every $p =
    o(n^{-1/m(H)}$).
\end{theorem}

\begin{theorem}[Theorem 3.19 in \cite{JLR}]\label{thm:fixedGraphDist}
	Let $H$ be a strictly balanced graph and let $p=cn^{-1/m(H)}$ for some constant $c >0$. The number of $H$-copies in $\ggnp$ converges in distribution to $Poisson(\lambda)$, the Poisson distribution with parameter $\lambda := c^{e(H)}/|Aut(H)|$, where $Aut(H)$ is the automorphism group of $H$. In particular,
	$$\lim_{n\to \infty}\Pr[\textrm{$\ggnp$ is $H$-free}] = e^{-\lambda}.$$
\end{theorem}

%


Using Theorem~\ref{thm:HinGnp} we obtain the following characterization of the typical structure of sparse random graphs.

\begin{claim}\label{cl:smallTrees}
Let $k$ be a positive integer and let $\ggnp$ for $p =
o\left(n^{-\frac {k+1}{k}}\right)$. Then \whp every connected
component of $G$ is a tree with at most $k$ vertices.
\end{claim}

\begin{proof}
Let $\F$ be the family of all trees with $k+1$ vertices and all
cycles of length at most $k$. Every $H \in \F$ satisfies either
$m(H) = \frac k{k+1}$ (if $H$ is a tree) or $m(H) = 1$ (if $H$ is a
cycle). Since $\F$ is a finite family, a simple union bound on the members of $\F$ implies that $G$ is \whp $\F$-free by
Theorem~\ref{thm:HinGnp}. Hence every connected component of $G$ has
at most $k$ vertices, because every larger component contains a tree
with $k+1$ vertices as a subgraph. The absence of all short cycles
from $G$ completes the proof.
\end{proof}

We conclude this section with two results concerning the appearance of any number of $H$-copies in a random graph process.

\begin{claim}\label{cl:sparseBeforeDense}
    For any fixed graph $H$ and every finite family of fixed graphs $\F
    = \{H_1,\dots,H_k\}$, each satisfying $m(H_i) > m(H)$, and for every
    positive integer $N$, the appearance of $N$ vertex disjoint copies
    of $H$ in a random graph process occurs \whp before the appearance of any member of $\F$.
\end{claim}

\begin{proof}
    Let $\G_{NH}'$ be the graph property of containing $N$ vertex disjoint
    $H$-copies, and let $\G_\F$ be the graph property of not being $\F$-free. By Theorem~\ref{thm:HinGnp}, for $p = n^{-1/m(H)}\lnn$,
    \whp $\ggnp$ contains $N$ vertex disjoint $H$-copies, and by the
    same theorem, for $p = n^{-1/m(H)}\ln^2 n$, \whp $\ggnp$ is $\F$-free
    (since $\F$ is finite and this value of $p$ satisfies $p = o(n^{-1/m(H_i)})$ for every $i \in
    [k]$). Proposition~\ref{prop:GnpRGP} implies that \whp $\tau(\tilde G, \G_{NH}') =
    O(n^{2- 1/m(H)}\lnn)$ and $\tau(\tilde G, G_\F) = \Omega(n^{2-1/m(H)}\ln^2 n)$, which completes the proof. Note that we had to
    consider each property separately since ``containing $N$ vertex
    disjoint $H$-copies and being $\F$-free" is not a monotone graph
    property.
\end{proof}

\begin{corollary}\label{cor:balancedDisjoint}
    For every strictly balanced graph $H$ and for any integer $N$, \whp
    the first $N$ copies of $H$ which appear in a random graph process
    are all vertex disjoint.
\end{corollary}

\begin{proof}
    Let $\F$ be the family of all graphs consisting of two $H$-copies
    with a non-empty intersection. Then $\F$ is finite, and by
    Claim~\ref{cl:balancedDense} every member in it has maximal density
    larger than $m(H)$. The result then immediately follows from Claim~\ref{cl:sparseBeforeDense}.
\end{proof}

\section{General results and tools for $H$-Games}\label{sec:gnrlH}
In this section we provide general tools that will be fundamental for the proofs of the 0-statements. First, we start with a general method called the ``pairing strategy", which will be used (mostly) in the proofs of the $K_3$-games.
Next, we will focus on the proofs of the 0-statements of Theorems~\ref{thm:main} and~\ref{thm:CWsimple}. 
We present a process of deleting vertices and edges from an arbitrary graph $G$, until we get a subgraph $G^* \sbst G$ with the following property: any winning strategy for the spoiler for the $H$-game played on $V(G^*)$ may be extended to a winning strategy for the original game, played on $V(G)$. Then, we characterize the possible connected components of $G^*$. Last, in the light of this characterization, we focus on clique games and provide more specific results.

We discuss Maker-Breaker and Client-Waiter $H$-games, both played on $V(G)$ where $\ggnp$ and $H$ is a fixed, strictly 1-balanced graph on at least three vertices. In particular, $H$ is connected by Claim~\ref{cl:1balanced2connected}. Since the proofs for both theorems are almost identical, in this section we sometimes refer to Breaker and Waiter collectively as ``the spoilers".

\subsection{Pairings}
Consider a $(1:b)$ Maker-Breaker game $(X,\F)$, and let $\A = \{A_1, A_2, \dots\}$ be a family of pairwise disjoint subsets of $X$, each of size at most $b+1$. Suppose that for every $i$, whenever Maker claims an element of $A_i$, Breaker responds by claiming all free elements of $A_i$ (and, if necessary, completes his move according to some strategy). Then at the end of the game Maker occupies at most one element from each member of $\A$. This seemingly trivial strategy, called the \emph{pairing strategy}, turns out to be one of the most useful and basic strategies in positional games. It can be Breaker's entire strategy, or a part of a more involved one. Waiter can use the same strategy in a Client-Waiter game, by offering an entire set $A_i$ in each move (and offer the remaining elements of $X$, if there are any, according to some other strategy).

Pairing strategies are fundamental for our proofs in this paper: our main method is to provide the spoiler (either Breaker or Waiter) with a winning strategy for a game played on a ``simple" graph, which can be extended via a pairing strategy to a winning strategy for the original game. This is explained in more details in Section~\ref{sec:gnrlH}. Furthermore, the winning strategy for this simple graph usually involves another pairing strategy, applied to each connected component separately. Two of the basic connected components we have to deal with when analyzing $K_3$-games are the following.
\begin{itemize}
	\item A \emph{Triple Triangle}, denoted by $\ttt$, is the graph with vertex set $V(\ttt)=\{v_1, v_2, v_3, v_4, v_5 \}$ and edge set $E(\ttt)=\{v_1v_2, v_1v_3, v_2v_3, v_2v_4, v_3v_4, v_3v_5, v_4v_5 \}$, as shown in Figure~\ref{fig:trio}.
	\item For an integer $t \geq 2$ let $DD_t$ be the graph obtained by taking a $K_3$-chain of length $t$ consisting of the triangles $\{a_ib_ic_i \}_{i=1}^t$ where $c_i=a_{i+1}$ for every $0<i<t$, and adding to it two triangles, $a_1c_1y$ and $c_{t-1}c_tx$, where $x,y$ are new vertices (see Figure~\ref{fig:ddt}). Note that we got two diamonds connected by a $K_3$-chain of length $t-2$, and in particular $DD_2 = DD$ (recall Definition~\ref{def:DD}).
\end{itemize}

\iffigure
\begin{figure}
	\centering
	\begin{minipage}{.4\textwidth}
		\centering
		\begin{tikzpicture}[auto, vertex/.style={circle,draw=black!100,fill=black!100, thick,
			inner sep=0pt,minimum size=1mm}]
		\node (v1) at ( 0,0) [vertex,label=below:$v_1$] {};
		\node (v2) at ( 0.5,0.866) [vertex,label=above:$v_2$] {};
		\node (v3) at ( 1,0) [vertex,label=below:$v_3$] {};
		\node (v4) at ( 1.5,0.866) [vertex,label=above:$v_4$] {};
		\node (v5) at ( 2,0) [vertex,label=below:$v_5$] {};
		
		\draw [-] (v1) --node[inner sep=0pt,swap]{} (v2);
		\draw [-] (v1) --node[inner sep=0pt,swap]{} (v3);
		\draw [-] (v2) --node[inner sep=0pt,swap]{} (v3);
		\draw [-] (v2) --node[inner sep=0pt,swap]{} (v4);
		\draw [-] (v3) --node[inner sep=0pt,swap]{} (v4);
		\draw [-] (v3) --node[inner sep=0pt,swap]{} (v5);
		\draw [-] (v4) --node[inner sep=0pt,swap]{} (v5);
		
		\end{tikzpicture}
		\caption{The graph $\ttt$}\label{fig:trio}
	\end{minipage}%
	\begin{minipage}{.6\textwidth}
		\centering
		\begin{tikzpicture}[auto, vertex/.style={circle,draw=black!100,fill=black!100, thick,
			inner sep=0pt,minimum size=1mm}]
		\node (b1) at (-2,0.866) [vertex,label=above:$b_1$] {};
		\node (a1) at (-2.5,0) [vertex,label=left:$a_1$] {};
		\node (y) at (2.5,-0.866) [vertex,label=below:$y$] {};
		\node (c1) at (-1.5,0) [vertex, label={[shift={(0.1,-0.63)}]$c_1$}] {};
		\node (b2) at (-1,0.866) [vertex,label=above:$b_2$] {};
		\node (v3) at (1,0) [vertex] {};
		\node at (0,0) {\ldots};
		\node at (0.5,0) {\ldots};
		\node (c2) at (-0.5,0) [vertex,label=below:$c_2$] {};
		\node (bt1) at (1.5,0.866) [vertex,label=above:$b_{t-1}$] {};
		\node (ct1) at (2,0) [vertex,label={[shift={(-0.2,-0.63)}]$c_{t-1}$}] {};
		\node (bt) at (2.5,0.866) [vertex,label=above:$b_t$] {};
		\node (ct) at (3,0) [vertex,label=right:$c_t$] {};
		\node (x) at (-2,-0.866) [vertex,label=below:$x$] {};
		
		\draw [-] (bt) --node[inner sep=0pt,swap]{} (ct);
		\draw [-] (bt) --node[inner sep=0pt,swap]{} (ct1);
		\draw [-] (ct) --node[inner sep=0pt,swap]{} (ct1);
		\draw [-] (ct) --node[inner sep=0pt,swap]{} (y);
		\draw [-] (y) --node[inner sep=0pt,swap]{} (ct1);
		\draw [-] (ct1) --node[inner sep=0pt,swap]{} (bt1);
		\draw [-] (ct1) --node[inner sep=0pt,swap]{} (v3);
		\draw [-] (bt1) --node[inner sep=0pt,swap]{} (v3);
		\draw [-] (b1) --node[inner sep=0pt,swap]{} (a1);
		\draw [-] (b1) --node[inner sep=0pt,swap]{} (c1);
		\draw [-] (a1) --node[inner sep=0pt,swap]{} (c1);
		\draw [-] (a1) --node[inner sep=0pt,swap]{} (x);
		\draw [-] (x) --node[inner sep=0pt,swap]{} (c1);
		\draw [-] (c1) --node[inner sep=0pt,swap]{} (b2);
		\draw [-] (c1) --node[inner sep=0pt,swap]{} (c2);
		\draw [-] (b2) --node[inner sep=0pt,swap]{} (c2);
		\end{tikzpicture}
		\caption{The graph $DD_t$}\label{fig:ddt}
	\end{minipage}
\end{figure}
\fi


We now provide pairing strategies for several graphs, which are later used in the proofs of the different unbiased $K_3$-games. For each of these graphs we provide a list $\LL$ of pairs, to which we refer as its \emph{natural pairs}. A pairing strategy with respect to these pairs is called \emph{the natural pairing strategy} for the graph. We begin with two graphs for which, if the spoiler uses the natural pairing strategy, he prevents the builder from creating a triangle.

\begin{definition}\label{def:K3cyclePairing}
	Given a $K_3$-cycle of length $t \ge 4$, denote the vertices of its triangles by $\{a_i, b_i, c_i\}_{i = 1}^t$, where $c_t = a_1$ and $c_i = a_{i+1}$ for every $0 < i < t$. The natural pairs for this graph are $\Lambda = \left\{\{a_i,b_i\} \right\}_{i=1}^t$.
\end{definition}

\begin{definition}\label{def:trioPairing}
	The set of natural pairs of the graph $\ttt$, using the labeling of Figure~\ref{fig:trio}, is $\Lambda = \{\{v_2, v_3\}, \{v_4, v_5\} \}$.
\end{definition}

We next deal with the graph $DD_t$. Here we do not have a ``proper" pairing strategy, and in particular Breaker cannot follow it. Moreover, Maker wins the unbiased $K_3$-game played on this graph. However, there is a winning strategy for Waiter in this game, in which he uses a pairing strategy from his second move onwards, where the set of pairs he uses depends on Client's first move.

\begin{claim}\label{cl:DDtPairing}
	For any $t \ge 2$, Waiter wins the unbiased Client-Waiter $K_3$-game played on $DD_t$.
\end{claim}

\begin{proof}
	With respect to the labeling of Figure~\ref{fig:ddt}, let $\Lambda_x = \{a_1,c_1\} \cup \left\{\{b_i,c_i\} \right\}_{i=2}^t$ and $\Lambda_y = \{a_t,c_t\} \cup \left\{\{a_i,b_i\} \right\}_{i=1}^{t-1}$. Waiter's winning strategy goes as follows. In his first move he offers the pair $\{x,y\}$. If Client chooses the vertex $x$ Waiter proceeds with the pairing strategy according to $\LL_x$, and otherwise proceeds with the pairing strategy according to $\LL_y$.
\end{proof}

We finish this part with a pairing strategy for the graph $DD$, which unlike the previous pairing strategies is constructive, and used by the builders Maker and Enforcer in their respective unbiased triangle games, that is, in the proofs of Theorems~\ref{thm:triangle11} (Maker) and~\ref{thm:AEstrict} (Ennforcer).

\begin{observation}\label{obs:DDpairing}
	The set of natural pairs of the graph $DD$, using the labeling of Figure~\ref{fig:DD}, is $\Lambda = \{\{y_1, y_2\}, \{z_1, z_2\}, \{z_3, z_4\} \}$. If $U \sbst V(DD)$ contains the center of $DD$ (the vertex $x$) and at least one vertex from each of the natural pairs of $DD$, then $DD[U]$ contains a triangle.
\end{observation}

\subsection{The $(H,b)$ graph deletion algorithm}

In this subsection we present the process of deleting vertices and edges from an arbitrary graph $G$ in a way that will allow us to extend any winning strategy on the resulting graph to a winning strategy on the original graph. This will be used in the proofs of the 0-statements of Theorems~\ref{thm:main} and~\ref{thm:CWsimple}.

Given a connected graph $H$ and a positive integer $b$, we describe the $(H,b)$ deletion algorithm, applicable to any graph $G$. We first need the following definitions.

\begin{definition}
Let $G,H$ be graphs, where $H$ is connected, and let $b$ be a
positive integer. All the following are defined with respect to $G$,
$H$ and $b$. A \emph{bad vertex} is a vertex $v \in V(G)$ which is
not contained in any $H$-copy in $G$. A \emph{bad edge} is an edge
$e \in E(G)$ which is not contained in any $H$-copy in $G$. A
\emph{bad set} is a set $U \sbst V(G)$ of size $2 \le |U| \le b+1$
such that there is no $H$-copy $\hat{H} \sbst G$ satisfying
$|V(\hat{H}) \cap U| = 1$. A bad set of size 2 is referred to as a \emph{bad pair}. A \emph{small component} is a connected component of
$G$ with at most $(b+1)(v(H)-1)$ vertices.
\end{definition}

\paragraph{\normalfont\bfseries{The algorithm}}
Begin with $G_0 = G$. For as long as possible, obtain $G_{i+1}$ from $G_i$ by performing one  \emph{deletion step}, that is, delete arbitrarily either a bad vertex, a bad edge, a bad set or a small component, where these are all defined with respect to the current graph $G_i$ and the fixed $H$ and $b$. The \emph{output} of the algorithm is a sequence $U = \{U_1,U_2,\dots\}$ of all bad sets
that were deleted during the process (if there were any) in the
order of deletion (that is, $U_i$ was the $i$th bad set to be
deleted), a sequence $W = \{W_1,W_2,\dots\}$ of the vertex sets of the
small components that were deleted (again, if there were any and in
the order of deletion), and the remaining (possibly empty) graph when no deletion can be
made anymore. Note that all of the sets in $U$ and $W$ are
clearly pairwise disjoint.

Given the $(H,b)$ graph deletion algorithm, and before further discussing it, we introduce some more terminology.

\begin{definition}
Given a connected graph $H$ and a positive integer $b$, an
$(H,b)$-stable graph is a graph for which no deletion step of the $(H,b)$
deletion algorithm can be made. For any graph $G$, the
$(H,b)$-core of $G$ is the union of all $(H,b)$-stable subgraphs of
$G$. Whenever $H$ and $b$ are clear from the context, we denote the
$(H,b)$-core of $G$ by $G^*$. For similar reasons, we also omit $b$,
or both $H$ and $b$, when talking about stability.
\end{definition}

\begin{remark}\label{rmk:1stable}
It is immediate to see that if a graph is $(H,b+1)$-stable, then it is also $(H,b)$-stable, and in particular every $(H,b)$-stable graph is also $(H,1)$-stable.
\end{remark}

We now present some properties of the deletion algorithm and its output graph, showing eventually why the spoiler wins the game played on a graph $G$ if he wins the game played on $G^*$.

\begin{claim}\label{cl:oneDeletion}
Let $G \sbst G_1$ be graphs such that $G$ is $(H,b)$-stable, and let $G_2$ be a graph obtained from $G_1$ by applying one deletion step of the $(H,b)$ deletion algorithm. Then $G \sbst G_2$.
\end{claim}

\begin{proof}
By the stability of $G$, for any $v \in V(G)$ there exists an
$H$-copy $\hat{H} \sbst G \sbst G_1$ containing $v$, therefore $v$
is not bad with respect to $G_1$. For the same reason, no edge $e \in E(G)$ is
bad with respect to $G_1$.
Next, every bad set $U\sbst V(G_1)$ must be vertex disjoint from $V(G)$. Otherwise, let $U' = U\cap V(G)$ and note that whether $|U'|=1$ or $2 \le |U'| \le b+1$ there exists an $H$-copy $\hat{H} \sbst G$ such that $|U' \cap V(\hat{H})| = 1$, since $G$ contains no bad vertices and no bad sets. But then $|U \cap V(\hat{H})| = 1$ as well, implying $U$ is not a bad set in contradiction.
Finally, every connected component $\Gamma \sbst G_1$ containing at least one vertex from $V(G)$ must contain a connected component of $G$, and therefore $\GG$ is not small.
We conclude that any deletion step applied to $G_1$ contains no vertices and no edges from $G$, and thus $G \sbst G_2$.
\end{proof}

\begin{corollary}\label{cor:MaxStbl}
Let $G,H$ be graphs where $H$ is connected, and let $b$ be a
positive integer. The $(H,b)$ deletion algorithm applied to $G$
terminates with the $(H,b)$-core of $G$, regardless of the arbitrary
choices made during the process. In particular, $G^*$ is $(H,b)$-stable.
\end{corollary}

\begin{proof}
Let $\hat{G}$ be a graph received by applying the deletion algorithm
on $G$. Since $\hat{G}$ is stable, $\hat{G} \sbst G^*$ trivially
holds. On the other hand, by Claim~\ref{cl:oneDeletion} we
get that every stable subgraph of $G$ is contained in $\hat{G}$, and
thus $G^* \sbst \hat{G}$ as well.
\end{proof}

\begin{claim}\label{cl:G*}
Let $G,H$ be graphs where $H$ is connected, let $b$ be a positive integer, and let $U,W,G^*$ be the
output of an arbitrary application of the $(H,b)$ deletion algorithm
on $G$. Let $S \sbst V(G)$ be such that $S$ contains at most one vertex
from any set $U_i \in U$ and at most $v(H) - 1$ vertices from any
set $W_j \in W$. Then every $H$-copy contained in $G[S]$ is also
contained in $G^*$.
\end{claim}

\begin{proof}
If $G$ is $H$-free there is nothing to prove. Assume then that it is not, and let $\hat{H} \sbst G$ be an $H$-copy not contained in $G^*$. We need
to show that $\hat{H}$ is not contained in $G[S]$ either. By the assumption
on $\hat{H}$, when considering the deletion algorithm that was
applied on $G$, there exists an integer $k \ge 0$ such that $\hat{H}
\sbst G_k$ but $\hat{H} \not\sbst G_{k+1}$. Furthermore, this is due
to a removal of either a bad set $U_i$ or a small component
$\Gamma_j$, as all vertices and edges contained in $\hat{H}$ are
obviously not bad in $G_k$.

In the first case, it follows that $|V(\hat{H}) \cap U_i| \ge 2$,
since $V(\hat{H})$ must intersect $U_i$, and no $H$-copy in $G_k$
contains exactly one vertex from $U_i$ by definition of a bad set.
However, $S$ contains at most one vertex from every bad set, and so
$V(\hat{H}) \not\sbst S$. In the latter case, that is, a small
component $\Gamma_j$ with vertex set $W_j$ was removed, note that
$\hat{H} \sbst \Gamma_j$ by the connectivity of $H$. Since $S$ contains at most $v(H) - 1$
vertices of $W_j$, once again we get $V(\hat{H}) \not\sbst S$.
\end{proof}

\begin{lemma}\label{lem:BwinG'}
Let $G,H$ be graphs where $H$ is connected, let $b$ be a positive integer, and let $G^*$ be the
$(H,b)$-core of $G$. When playing the $(1:b)$ Maker-Breaker or Client-Waiter $H$-game, if the spoiler has a winning strategy when playing the game on $V(G^*)$, then he has a winning strategy when playing the game on $V(G)$.
\end{lemma}

\begin{proof}
We start with the Maker-Breaker game, assume a winning strategy $\cS_B^*$ for Breaker for the game played on $V(G^*)$, and provide him with the following winning strategy $\cS_B$ for the game played on $V(G)$. First, Breaker
runs the deletion algorithm on $G$ and obtains an output $U,W,G^*$.
During every round of the game, denote the last vertex claimed by
Maker by $v$. Breaker responds according to the following cases.

\begin{enumerate} [$(1)$]
\item If $v \in V(G^*)$, Breaker plays according to $\cS_B^*$. If
there are less than $b$ free vertices in $V(G^*)$ before his move,
or if he is supposed to claim some vertices he already occupies, he
completes his move arbitrarily.
\item Otherwise, if $v \in U_i$ for some $U_i \in U$, Breaker claims all
the free vertices in $U_i$, and, if necessary, completes his move
arbitrarily.
\item Otherwise, if $v \in W_j$ for some $W_j \in W$, Breaker claims
$b$ arbitrary free vertices of $W_j$, and, if there were less than
$b$ free vertices in $W_j$, completes his move arbitrarily.
\item Otherwise, Breaker makes an arbitrary move.
\end{enumerate}

The strategy is well defined (recall that all the sets in $U$ and
$W$ are pairwise disjoint, and clearly they are disjoint from
$V(G^*)$ as well), and Breaker can follow it. Case~$(2)$ of $\cS_B$
ensures that Maker claims at most one vertex from every bad set, and
Case~$(3)$ of $\cS_B$ ensures that Maker claims at most
$\Bigl\lceil|W_j| / (b+1) \Bigr\rceil \le v(H) - 1$ vertices of any
$W_j \in W$. Therefore, by Claim~\ref{cl:G*}, Breaker wins the game
if he prevents Maker from claiming any $H$-copy $\hat{H} \sbst G^*$,
and this is guaranteed by $\cS_B^*$.

Now assume a winning strategy $\cS_W^*$ for Waiter in the game played on $V(G^*)$. Waiter uses an analogous strategy to $\cS_B$ in the
following way. He first runs the deletion algorithm and obtains the
list of all bad sets and small components which were deleted during
the process. He then offers all vertices of $G^*$ according to
$\cS_W^*$ (recall that Waiter may offer less than $b+1$ elements per move). Next, For every bad set that was removed he offers all of its vertices in one
move, ensuring that Client claims only one of them. For every small
component he offers arbitrary $b+1$ of its vertices repeatedly until
all vertices of the component have been offered (he may offer less
than $b+1$ vertices when he offers the last free vertices of the
component). By this he ensures that Client claims at most $v(H) - 1$
vertices of any small component. He offers all the remaining
vertices arbitrarily.

It is easy to see that Waiter can follow this strategy and win the
game played on $G$, by Claim~\ref{cl:G*} and by the fact that the
strategy $\cS_W^*$ ensures that by the end of the game Client cannot
fully claim any $H$-copy $\hat{H} \sbst G^*$.
\end{proof}

\subsection{The structure of $G^*$}
By Lemma~\ref{lem:BwinG'}, in order to provide the spoilers with winning strategies for their $(1:b)$ $H$-games played on the vertex set of a graph $G$, it suffices to provide them with winning strategies for the same games played on $V(G^*)$. Since $H$ is connected by assumption, they only need to prevent their opponents from claiming an $H$-copy in every connected component of $G^*$.
Since Waiter can avoid offering vertices from different components in the same round, and since Maker claims only one vertex per move, allowing Breaker to respond in the same component (and complete his move arbitrarily if necessary) we get the following observation.

\begin{observation}\label{obs:playSeparately}
Let $H$ be a connected graph. In any $(1:b)$ Maker-Breaker or Client-Waiter $H$-game played on the vertex set of a graph $G$, the spoiler has a winning strategy in the game if for every connected component $\GG \sbst G$, he has a winning strategy when playing the game on $V(\GG)$.
\end{observation}

For better understanding of the possible structures of the connected components
of $G^*$, we describe an exploration process that can be applied to
every stable component. This is done step by step, by starting with one
vertex and then slowly expand our view of the component by adding to
it one piece of structure at a time. We describe here the general
method, and then in
Sections~\ref{sec:clique},~\ref{sec:cycle} and~\ref{sec:others}
we go into more details according to the graph $H$ and the bias $b$
in discussion.

We first need some terminology. Let $H$ and $\Gamma$ be connected
graphs and let $\Gamma'$ be a connected subgraph of $\Gamma$. The
following definitions are made with respect to $H$, $\GG$ and
$\GG'$. An \emph{internal edge} is an edge
$uv \in E(\Gamma) \stm E(\Gamma')$ such that $u,v \in V(\Gamma')$.
For an $H$-copy $\hat{H} \sbst \Gamma$ such that $V(\hat{H})
\not\sbst V(\Gamma')$ let $U := V(\hat{H}) \cap V(\Gamma')$. We say
that $\hat{H}$ is an \emph{external} copy if $|U| = 1$; we say it is
an \emph{internal} copy if $|U| > 1$ and in addition $E(\hat{H}[U])
\sbst E(\GG')$, that is, $\hat{H}$ contains no internal edges. We do not consider any other kind of $H$-copies.
Whenever we say we add an internal edge $e$ or an $H$-copy $\hat{H}$
(either external or internal) to $\Gamma'$ we mean that we now
expand our view from $\Gamma'$ to $\GG''$, where $\GG'' = \Gamma' \cup \{e\}$ or $\GG''=\Gamma' \cup \hat{H}$, respectively. In either case we get a new connected subgraph of
$\Gamma$. With respect to the transition from $\GG'$ to $\GG''$ as above, a vertex $v \in V(\GG'')$ is called an \emph{existing} vertex if $v \in V(\GG')$ and a \emph{new} vertex otherwise. Existing and new edges are defined analogously.

We are now ready to describe the exploration process, applicable to any connected $(H,1)$-stable graph $\GG$ (and thus by Remark~\ref{rmk:1stable} applicable to any $(H,b)$-stable connected graphs). We start by setting $\Gamma_0 =
v$ for an arbitrary vertex $v \in V(\Gamma)$. Then, while $\Gamma_i
\neq \Gamma$, we expand $\Gamma_i$ to $\Gamma_{i+1}$ by adding
either an internal edge, an external $H$-copy or an internal
$H$-copy, with one restriction: if $\Gamma_i$ is obtained from
$\Gamma_{i-1}$ by adding an external $H$-copy, and if $\GG_{i+1}$ is obtained from $\GG_i$ by adding an $H$-copy $H'$ (either internal or external), then $V(H') \cap \big(V(\Gamma_i) \stm V(\Gamma_{i-1})\big)
\neq \emptyset$ must hold. In other words, after the addition of an external
$H$-copy, which added $v(H) - 1$ new vertices to the subgraph, we
are looking for either an $H$-copy containing at least one of them, or an arbitrary internal edge. Note that this means that adding $t$ consecutive external copies is in fact
adding an $H$-chain of length $t$ to the subgraph; we often use this terminology.

Since each vertex and each edge in an $H$-stable graph must be a part of some $H$-copy, it is evident we can explore every $H$-stable connected component by sequentially adding internal edges and $H$-copies. It only remains to show why any component can be explored under the aforementioned restriction. We use the following claim.

\begin{claim}\label{cl:externalCollapse}
Let $\GG$ be an $H$-stable connected graph. Suppose that during an arbitrary exploration of $\GG$, the $i$th step was the addition of an external $H$-copy $H_1$, consisting of an existing vertex $v$ and a set $U$ of new vertices. Then $\GG$ contains an $H$-copy $H_2 \not\sbst \GG_i$ such that $V(H_2) \cap U \neq \emptyset$.
\end{claim}

\begin{proof}
Note that $|U| \ge 2$, and that any two vertices in $U$ not contained in any other $H$-copy in $\GG$ but $H_1$ form a bad pair, which contradicts the stability of $\GG$. It therefore suffices to show that $H_1$ is the only $H$-copy in $\GG_i$ containing any vertices from $U$.

Indeed, recall that $H$ is 2-vertex-connected by Claim~\ref{cl:1balanced2connected},  and let $H' \sbst \GG_i$ be an $H$-copy satisfying $U_1 := V(H') \cap U \neq \emptyset$. Now, if $H' \neq H_1$ then $U_2 := V(H') \stm V(H_1) \neq \emptyset$. But since there are no edges between $U_1$ and $U_2$ in $\GG_i$, we have that $H' \stm \{v\}$ is disconnected, a contradiction.
\end{proof}

It is now immediate to see that the exploration process is well defined: in the terminology of Claim~\ref{cl:externalCollapse}, after the addition of $H_1$ we can either add $H_2$ if possible, or add an internal edge required for $H_2$ otherwise (the addition of an internal edge $e \not\in E(H_2)$ is also a legal step). Claim~\ref{cl:externalCollapse} also leads to the following useful corollary.

\begin{corollary}\label{cor:noExtrnlCps}
In every exploration of an $H$-stable connected graph $\Gamma$, the
last step cannot be an addition of an external $H$-copy. In
particular, if after step $i$ in the exploration of $\Gamma$ we
argue that we can only continue by adding external $H$-copies, then
$\Gamma = \Gamma_i$.
\end{corollary}

Our general approach in our analysis is to show that when
exploring a connected component of $G^*$, we may consider only those in which the number of times we can
add internal edges or internal $H$-copies is very limited, as we do not expect $G$ to contain any dense components. We then investigate the possible stable components that can be constructed under these restrictions. We therefore use the following classification of connected components.

\begin{definition}\label{def:Xqts}
For every connected graph $H$ and three non-negative integers
$q,t,s$,  let $X^H_{q,t,s}$ be the family of all connected graphs (not necessarily $H$-stable) which can be obtained by adding exactly $q$ internal edges, $t$
external $H$-copies, and $s$ internal $H$-copies during the exploration process.
When $H$ is clear from the context we abbreviate to $\xqts$. Throughout this paper we assume that $q,t,s$ are non-negative, even if we do not state that explicitly. Furthermore, we assume $t > 0$ as $\GG_1$ is always obtained from $\GG_0$ by the addition of an external $H$-copy, regardless of $H$, $q$, $t$ or $s$.
\end{definition}

Note that in general, a connected component $\Gamma$ could be
explored in many different ways (all of them terminating with
$\Gamma$, of course). The number of times each addition type is
used during the exploration may vary between different
explorations, so the families $\xqts$ are not pairwise
disjoint. However, when referring to a graph $\GG \in \xqts$, we
\textbf{always} consider an exploration of $\GG$ in such a way that each
addition type is applied exactly $q$, $t$, or $s$ times,
respectively.

When possible, it is extremely beneficial for the analysis to start the exploration of a given component with two $H$-copies whose edge sets intersect, as defined below (recall that the first step in any exploration, of any connected graph, is the addition of an external $H$-copy).

\begin{definition}\label{def:greedy}
A \emph{greedy} exploration of a connected component $\GG$ is any exploration in which the second step is the addition of an internal copy which contains at least one existing edge.
\end{definition}

Note that if there exists a greedy exploration for $\GG$, then $\GG$ must contain at least one dangerous edge. However, the existence of such an edge in $\GG$ does not guarantee the existence of greedy explorations. Indeed, even if $H_1$ and $H_2$ share an edge and the exploration starts with $H_1$, it is possible that internal edges need to be added prior to $H_2$. In addition, it is possible that there exist integers $q,t,s$ such that $\GG \in \xqts$ only for non-greedy explorations of $\GG$.

%
%

\subsection{Results for clique games}\label{sec:genCliques}
In this subsection we consider $K_k$-games where $k \ge 3$ is some fixed integer (this is implicit for the remainder of the section). In particular, since $m_1(K_k) = k/2$, we consider the random graph $\ggnp$ for $p = O(n^{-2/k})$. We provide general results which are later used in the analysis of several games in Sections~\ref{sec:clique} and~\ref{sec:others}.
We start by showing that any component whose exploration contains too many additions of internal edges and internal $K_k$-copies is unlikely to appear in $G$.

\begin{claim}\label{cl:qsBound}
Let $0 < a < 1$ (not necessarily a constant), let $\ggn {an^{-2/k}}$, let $q,t,s$ be three integers satisfying $2q + (k-2)s > k$, and let $\GG \in \xqts$. Then $\prgg \le n^{-1/k}$.
\end{claim}

\begin{proof}
Consider an arbitrary exploration of $\GG$, and for every $1 \le
i \le s$ let $r_i$ denote the number of existing vertices contained
in the $i$th internal clique added during the exploration. We have:
$$v(\Gamma) = 1 + t(k-1) + \sum_{i=1}^{s} (k - r_i)$$
and
$$ e(\Gamma) = q + t{k \choose 2} + \sum_{i=1}^{s} \left({k
\choose 2} - {r_i \choose 2} \right),$$ and therefore
\begin{align*}
\Pr [\Gamma \sbst G] & \le n^{v(\Gamma)} p^{e(\Gamma)}\\
& = a^{e(\Gamma)}n^ {1 + t(k-1) + \sum_{i=1}^{s} (k - r_i) - \frac{2}{k} \left[q + t{k \choose 2} + \sum_{i=1}^{s} \left({k \choose 2} - {r_i \choose 2}\right)\right]}\\
& = a^{e(\Gamma)}n^{1 - \frac{1}{k} \left[2q + \sum_{i=1}^{s} \left((kr_i - k^2) + (k^2 - k) - r_i(r_i - 1) \right) \right]}\\
& = a^{e(\Gamma)}n^{1 - \frac{1}{k} \left[2q + \sum_{i=1}^{s}
(k-r_i)(r_i-1) \right]}.
\end{align*}
By definition of an internal copy we have $2 \le r_i \le k-1$ for every $i$, and since
$(k-x)(x-1) \ge k-2$ for every $2 \le x \le k-1$ we obtain
\begin{equation} \label{eq:xqts}
\Pr [\Gamma \sbst G] \le a^{e(\Gamma)}n^{1 - \frac{1}{k} \left[2q
+ (k-2)s\right]},
\end{equation}
which completes the proof by the assumptions on $a$, $q$ and $s$.
\end{proof}

We next provide an upper bound on the maximal length of $K_k$-chains we expect $G$ to contain.

\begin{claim}\label{cl:NoChains}
Let $\ggn{cn^{-2/k}}$ for some constant $0 < c < 1$, and let $\GG$ be the $K_k$-chain of length $t = -\frac{1}{\ln c} \ln n$ (note that $\ln c$ is negative). Then \whp $G$ is $\GG$-free.
\end{claim}

\begin{proof}
By observing that $v(\Gamma) = 1 + t(k-1)$ and $e(\Gamma) = t\binom
k2$ we get
\begin{equation*}
\Pr[\Gamma \sbst G] \leq n^{v(\Gamma)} p^{e(\Gamma)} = n^{1 + t(k-1)
- \frac 2k t\binom k2} c^{t\binom k2} \leq nc^{-\frac{2}{\ln c} \ln
n} = \tfrac{1}{n} = o(1).\qedhere
\end{equation*}
\end{proof}

Claim~\ref{cl:NoChains} shows that we can limit our focus to
components containing no long $K_k$-chains. For every three integers
$q,t,s$, let $\yqts$ denote the family of all members of $\xqts$
containing no $K_k$-chains of length more than $-\frac{1}{\ln c} \ln
n$ as subgraphs, and let $Y_{q,s} = \bigcup_{t \ge 1} \yqts$.

\begin{claim}\label{cl:yqsIsSmall}
$|Y_{q,s}| = O\left((\ln n)^{3q + ks}\right)$ holds for any fixed integers $q,s$.
\end{claim}

\begin{proof}
Fix $q$ and $s$ and let $\GG \in Y_{q,s}$. In the exploration of
$\GG$ we start with a single vertex and then add elements
$A_1,B_1,A_2,B_2,\dots,A_{q+s},B_{q+s}$ where every $A_i$ is a
$K_k$-chain (possibly empty), and every $B_i$ is either an internal
edge or an internal $K_k$-copy. Since $v(A_i) = O(\ln n)$ and $v(B_i)
= O(1)$ for every $i$, and $q,s$ are fixed, we get $v(\GG) = O(\ln
n)$.

Now let us bound from above the number of different graphs we can
obtain in this way. For every $A_i$ there are $O(\lnn)$ options to
choose the length of the chain, and recall that there is only one type of
$K_k$-chain of each length. For each of the $q$ internal edges there
are at most $\binom {v(\GG)}{2} = O(\ln^2n)$ options to choose its
endpoints. For each of the $s$ internal $K_k$-copies there are at most
$\sum_{r = 2}^{k-1}\binom {v(\GG)}{r} = O((\ln n)^{k-1})$ options to
choose the existing vertices it contains. Multiplying all these
factors we get the desired result.
\end{proof}

Using the previous three claims we obtain our main result for clique games,
which we use extensively in the following sections.

\begin{claim}\label{cl:noNegativeExponent}
Let $\ggn{cn^{-2/k}}$ for some constant $0 < c < 1$. The following holds w.h.p.: all families $\xqts$ for which at least one of their members appears in $G$ satisfy $2q + (k-2)s \le k$.
\end{claim}

\begin{proof}
By Claim~\ref{cl:NoChains} we may assume that $G$ contains no
$K_k$-chains of length more than $-\frac{1}{\ln c} \ln n$, thus we
only consider the families $\yqts$ and $Y_{q,s}$. Let $q,t,s$ be
three integers such that $2q + (k-2)s > k$, let $\GG \in \yqts$, and
consider an arbitrary exploration of $\GG$. Denote by $q_j$ and $s_j$ the
number of internal edges and internal copies added during the first
$j$ steps of the process, respectively. Let $i$ be the maximal
integer such that after $i$ steps in the process $2q_i + (k-2)s_i
\le k$ holds (such an $i$ exists since $q_1=s_1=0$). It follows immediately that $q_i \le k/2$ and $s_i \le
k/(k-2) \le 3$. Let
$$Y_k = \bigcup_{q \le k,  s \le 4 \textrm{ s.t.} \atop  2q + (k-2)s > k} Y_{q,s}.$$
It follows that $\GG_{i+1} \in Y_k$, and so in order to prove the
claim it suffices to show that \whp $G$ is $Y_k$-free. By Claim~\ref{cl:qsBound} we have that $\prgg \le
n^{-\frac 1k}$ for every $\GG \in Y_k$. For any $q \le k$ and $s \le
4$ we get by Claim~\ref{cl:yqsIsSmall} that $|Y_{q,s}| =
O\left((\ln n)^{7k}\right)$, and thus $|Y_k| = O\left((\ln
n)^{7k}\right)$ for any fixed $k$. A simple union bound now yields
\begin{equation*}
\Pr[G \textrm{ is not $Y_k$-free}] \le \sum_{\GG
\in Y_k}\prgg \le |Y_k|n^{-\frac 1k} = o(1).\qedhere
\end{equation*}
\end{proof}

In the light of Claim~\ref{cl:noNegativeExponent}, we show the strong correlation between dangerous edges and stable graphs, when considering only the plausible subgraphs of the random graph.

\begin{lemma}\label{lem:noDang}
Let $q,t,s$ be integers satisfying $2q + (k-2)s \le k$, and let $\GG \in \xqts$ be a $(K_k,1)$-stable connected graph. Then either $\GG$ contains at least one dangerous edge, or $k=3$ and $\GG$ is a $K_3$-cycle of length at least four.
\end{lemma}

\begin{proof}
Assume that $\GG$ contains no dangerous edges and consider
the last addition in an arbitrary exploration of $\GG$. By
Corollary~\ref{cor:noExtrnlCps} it is not the addition of an
external $K_k$-copy. However, it cannot be an addition of an internal
edge $e$ either. Indeed, by stability $e$ must be a part of some
$K_k$-copy $\hat{H} \sbst \GG$.
By the assumption that $\GG$ contains no
dangerous edges, $\hat H$ cannot use any edges from other $K_k$-copies, and therefore all the edges of $\hat{H}$ must have been added to
$\GG$ previously as internal edges. But that would imply that $q \ge \binom k2
> k/2$, a contradiction. Hence, the last step of the
exploration must be an addition of an internal $K_k$-copy $\hat{H}$.
Since $\hat{H}$ contains at least two existing vertices by definition, and since $\GG$ contains no dangerous edges, it follows that at least
one internal edge $uv$ was added to $\GG$ prior to the addition of
$\hat{H}$. We therefore get $q \ge 1$ and $s \ge 1$, which by the restriction $2q + (k-2)s \le k$ leads to $q = s = 1$. That is, every addition during the exploration is of an external copy, except for $uv$ and $\hat H$.

Note that no external $K_k$-copies can be added between the additions of $uv$ and
$\hat{H}$.  Otherwise,  $\hat{H}$ must contain a vertex $w$ from the last external copy, and also use the edge $uv$. However, as $\hat H$ is a clique, the edges $uw$ and $vw$ also need to be added as internal edges, which is impossible since $q=1$.  We can now describe the entire exploration process: it
starts with a $K_k$-chain, then an internal edge $uv$ is added, then
an internal $K_k$-copy $\hat{H}$ is added, containing $u,v$ and $k-2$
new vertices, and then the exploration terminates.

If $k \ge 4$, the internal $K_k$-copy $\hat{H}$ contains $k - 2 \ge 2$
vertices which only belong to $\hat{H}$. Since any two of these new
vertices form a bad pair, $\GG$ is not stable, which leads to the
conclusion that there are no stable graphs with no dangerous edges
in this case.

For $k = 3$, note that just before the addition of $uv$ both the first and last triangles in the triangle chain contain a bad pair. By stability, the set $\{u,v\}$ contains exactly one
vertex from each of these pairs. The internal triangle $wuv$ that is
added in the last step (where $w$ is a new vertex) completes the
creation of a $K_3$-cycle. It is immediate to see that a
$K_3$-cycle contains no dangerous edges if and only if it is of
length at least four.
\end{proof}
%

\subsubsection{\normalfont\bfseries{Results for triangle games}}
For the remainder of the section we focus on the case $k=3$. Here the term $2q + (k-2)s$ translates to $2q + s$. We first observe that this term depends on $\GG$ alone and not on the way we explore it.

\begin{observation}\label{obs:k3invariant}
Let $\GG$ be a connected graph belonging to both $X^{K_3}_{q_1,t_1,s_1}$ and $X^{K_3}_{q_2,t_2,s_2}$. Then $2q_1 + s_1 = 2q_2 + s_2$.
\end{observation}

\begin{proof}
Since in any exploration of $\GG$ every internal triangle contains exactly one new vertex and two new edges, we get that if $\GG \in \xqts$ then $v(\GG) = 1 + 2t + s$ and $e(\GG) = q + 3t + 2s$. The first equation yields $t_1-t_2 = (s_2-s_1)/2$ and the second one yields $t_1-t_2 = (2s_2-2s_1 + q_2 - q_1)/3$. Comparing the right hand side of the last two equations, and rearranging, we get the desired result.
\end{proof}

As already mentioned, greedy explorations (whenever possible) significantly simplify the case analysis in our proofs. When $H$ is a triangle, the existence of a dangerous edge in the explored component guarantees the existence of greedy explorations, since we can always start the exploration with two triangles containing this edge (if $H$ is not a clique we might have to add internal edges first). Since we typically only investigate stable components, it is natural to define the following.

\begin{definition}\label{def:zqts}
For any three integers $q,t,s$, let $\zqts$ be the family of all $(K_3,1)$-stable graphs in $X^{K_3}_{q,t,s}$ obtained via a greedy exploration.
\end{definition}

We often perform case analysis on the possible connected graphs under some restriction on the value of $2q + s$. Observation~\ref{obs:k3invariant} and Definition~\ref{def:zqts} lead to the following, extremely useful, corollary.


\begin{corollary}\label{cor:exploreK3}
Let $\GG \in \xqts$ be a $(K_3,1)$-stable graph containing at least one dangerous edge. Then $\GG \in Z_{q',t's'}$ for some integers $q',t',s'$ satisfying $2q' + s' = 2q + s$.
\end{corollary}

Following Claim~\ref{cl:noNegativeExponent}, in Sections~\ref{sec:clique} and~\ref{sec:others} we investigate stable components $\GG \in \xqts$ for which $2q+s \le 3$, and we use the next two claims to narrow down the case analysis there. The first claim deals with greedy explorations when $2q+s = 3$, and the second deals with all explorations for which $2q+s < 3$.

\begin{claim}\label{cl:11is30}
$Z_{1,t+1,1} \sbst \z 03$ for every integer $t \ge 0$.
\end{claim}

\begin{proof}
Let $\GG \in Z_{1,t+1,1}$. By definition, we start its exploration with two triangles sharing an edge. By the fact that $q = s = 1$ and by Corollary~\ref{cor:noExtrnlCps}, the remainder of the exploration must be the addition of a $K_3$-chain of length $t$, followed by the addition of an internal edge as the final step. Clearly $t > 0$ as otherwise $\GG = K_4$, which is a small component and thus not stable. Let $abc$ be the last external triangle added in the exploration, where $a$ is its existing vertex. Since $\{b,c\}$ is a bad pair when $abc$ is added, the addition of the internal edge must create a new triangle containing exactly one of them. WLOG, this triangle must be $abw$ (the internal edge added is $bw$), where $w \neq c$ is some neighbor of $a$.

Therefore, we can explore $\GG$ in a slightly different way: we simply change the last two steps of the exploration. Instead of adding the external triangle $abc$ and then the edge $bw$, we first add the internal triangle $abw$ (in this triangle $b$ is the new vertex) and then the internal triangle $abc$ (here $c$ is the new vertex). Note that these two triangles are indeed internal with respect to this process and that the beginning of the exploration remains the same, and thus $\GG \in \z 03$, as we had to show.
\end{proof}


\begin{claim}\label{cl:2qsle2}
Let $q,t,s$ be integers satisfying $2q + s \le 2$ and let $\GG \in
\xqts$ be a $(K_3,1)$-stable graph. Then $\GG$ is either a {\ttt} or a $DD_t$.
\end{claim}

\begin{proof}
All the arguments in this proof are done with respect to graph
isomorphism. First observe that if $\GG$ does not contain dangerous
edges then it is a $K_3$-cycle by Lemma~\ref{lem:noDang}. However, the
only way to explore a $K_3$-cycle is to start with a $K_3$-chain,
then add an internal edge $e$ between the first and last triangles
in the chain, and terminate with an internal triangle containing
$e$. That would imply $q=s=1$, in contradiction to the
assumption.

It follows that $\GG$ must contain a dangerous edge, thus by Corollary~\ref{cor:exploreK3} we can explore $\GG$ greedily, that is, the explored subgraph after two steps $\GG_2$ consists of two triangles sharing an edge, $xyz_1$ and $xyz_2$. We have $s \ge 1$ by the fact that the triangle $xyz_2$ is internal and hence by the assumption we get $q = 0$ and $s \le 2$. By Corollary~\ref{cor:noExtrnlCps}, the remainder of the exploration must be the addition of a $K_3$-chain of length $t - 1$, followed by an addition of an internal triangle $T$ in the last step. We distinguish between two cases, keeping in mind that $\{x,y\}$ is a bad pair in $\GG_2$.

If $t = 1$, that is $\GG = \GG_2 \cup T$, then $T$ must contain an existing edge other than $xy$, thus $T = wxz_1$ (where $w$ is a new vertex), and $\GG$ is a $\ttt$.

If $t > 1$, denote the vertices of the last external triangle added by $a,b,c$, where $a$ is the existing vertex. Since $\{b,c\}$ is a bad pair when added, $T$ must contain an existing edge incident to exactly one of them by stability, and so $T = abw$, where $w$ is a new vertex. Furthermore, the $K_3$-chain must intersect $V(\GG_2)$ at $x$, otherwise $\{x,y\}$ remains a bad pair in $\GG$. Thus $\GG$ is a $DD_t$.
\end{proof}

We conclude this section with the following claim, which is fundamental for the analysis of unbiased triangle games involving random graph processes.

\begin{claim}\label{cl:ConComp}
Let $\F$ be the family of all graphs on less than 25 vertices with
density larger than $10/7$, and let $G$ be an $\F$-free, $K_3$-stable
graph. Then every connected component of $G$ is either a {\ttt} or a
$DD$.
\end{claim}

\begin{proof}
Let $\GG$ be a connected component of $G$ and let $q,t,s$ such that $\GG \in \xqts$. Observe that $v(\GG) = 1 + 2t + s$ and $e(\GG) = q + 3t + 2s$, and that $d(\GG) = e(\GG)/v(\GG) \le 10/7$ if and only if $7q + t + 4s
\le 10$. It follows that when exploring $\GG$ there can be at most
ten addition steps as after eleven steps we get a component
$\GG_{11} \sbst \GG$ on at most 23 vertices and of density larger
than $10/7$ in contradiction to the assumption. It means that
\textbf{every} connected component of $G$ has at most 21 vertices
and therefore density at most $10/7$. Thus, if $\GG \in \xqts$ is a
component of $G$ then $7q + t + 4s \le 10$ must hold, and in
particular $2q + s \le 2$ must hold as well (since either $q = 1$
and $s = 0$ or $q = 0$ and $s \le 2$).

We can therefore apply Claim~\ref{cl:2qsle2} and deduce that every
component in $G$ is either a ${\ttt}$ or a $DD_t$. In the latter case we
have $s = 2$ and $t \ge 2$, and by the restriction $7q + t + 4s \le
10$ the $DD_t$ must be of length 2, i.e., a $DD$.
\end{proof}

\section{Maker-Breaker clique games}\label{sec:clique}
In this section we deal with the various cases of Maker-Breaker $(1:b)$ $H$-games involving cliques. We first prove the $0$-statement of Theorem~\ref{thm:main} for $H' = K_k$ (recall that $H'$ is a subgraph of $H$ of maximal 1-density). We divide the proof into the two cases $k \ge 4,~b \ge 1$ (large cliques) and $k = 3,~b \ge 2$ (biased triangle games). We then prove Theorem~\ref{thm:triangle11} which covers the remaining case $k=3,~b=1$. Finally, we focus on the constants $c,C$ appearing in Theorem~\ref{thm:main} for the case $H'=K_k$.

\subsection{Cliques of size $k\geq 4$}\label{sec:largeCliques}
Here $H' = K_k$ for some fixed $k \ge 4$. If Breaker prevents Maker from claiming a $K_k$-copy then Maker cannot occupy any $H$-copy, hence it is enough to prove Breaker's side in the $K_k$-game. By bias monotonicity, it is enough to prove Breaker's
side for $b=1$. Now let $0 < c < 1$ be a constant and let $G\sim G(n,cn^{-2/k})$. By Lemma~\ref{lem:BwinG'} it suffices to show that \whp Breaker can win the game played on the vertex set of $G^*$, the $(K_k,1)$-core of $G$. Clearly, is suffices to show that \whp $G^*$ is empty. Recall that Claim~\ref{cl:noNegativeExponent} states that \whp $G$ (and therefore $G^*$) contains only members of families $\xqts$ for which $2q + (k-2)s \le k$. We are therefore done by the following lemma.

\begin{lemma}\label{lem:GisB'swinKk}
Let $k \ge 4$ and let $\GG$ be a non-empty, $(K_k,1)$-stable, connected graph. Then there exist integers $q,t,s$ such that $\GG \in \xqts$ and $2q + (k-2)s > k$.
\end{lemma}

\begin{proof}
Assume for contradiction that in every exploration of $\GG$ we have $2q + (k-2)s \leq k$.
Lemma~\ref{lem:noDang} implies that $\GG$ contains a dangerous edge and therefore we can explore $\GG$ greedily. Let $H_1, H_2$ be two $K_k$-copies in $\GG$ sharing at least one edge and let $v\in V(H_1)$.
We start the exploration of $\GG$ with $\GG_0 = \{v\}$, $\GG_1 = H_1$, and $\GG_2 = H_1\cup H_2$. Since we add $H_2$ as an internal copy we have $s\geq 1$. Note that $v(\GG_2) \le 2k - 2$, i.e. $\GG_2$ is a small component, and so $\GG_2 \neq \GG$. By Corollary~\ref{cor:noExtrnlCps} the exploration of $\GG$ cannot end with the addition of an external copy. Hence, by the assumption $2q + (k-2)s \leq k$, and since $q=0$ and $s=1$ is not possible, we only have to consider the following two cases.
\begin{enumerate}[1.]
\item $q = s = 1$.
\item $q=0$, $s=2$ and $k=4$.
\end{enumerate}
That is, after the addition of $H_2$ we must continue with the addition of a $K_k$-chain of length $t-1$ (recall that $H_1$ was an external copy), followed by the addition of either an internal edge $e$ or an internal copy $H_3$ as the last step of the exploration, where the latter case is only possible if $k=4$. We analyze each case separately, showing that none of them is possible, thus completing the proof.

\case{1}{q=s=1}
In this case we have $t > 1$ as otherwise $V(\GG) = V(\GG_2)$ and $\GG$ is a small component. Let $\hat H$ be the last external $K_k$-copy added in the exploration. Since it is a clique, the edge $e$ (which completes the exploration of $\GG$) contains at most one of its vertices. It follows that there is a set of $k-2\geq 2$ vertices of $\hat H$ (that is, all of the vertices of $\hat H$ but at most two) not contained in any other $K_k$-copy in $\GG$, as they each have degree $k-1$. Each pair in this set forms a bad pair, in contradiction to the stability of $\GG$.

\case{2}{q=0,~s=2,~k = 4}
Here we have $H = K_4$. Observe first that whenever an $H$-copy $H'$ is added during the exploration, no new $H$-copies appear in the explored component other than $H'$. Indeed, every new vertex of $H'$ has degree 3 and so it belongs only to $H'$, and all new edges are incident to new vertices, so none of them can be a part of any other $H$-copy as well. It follows that $H_3$, the $K_4$-copy added in the last step of the exploration, contains exactly one new vertex, as two new vertices would form a bad pair. Furthermore, the existing three vertices must form a triangle prior to the addition of $H_3$.

Now, assume that $t > 1$ and let $\hat H$ be the external $H$-copy added just before $H_3$. Since $H_3$ must contain a new vertex of $\hat H$, and two of its neighbors, $H_3$ in fact contains (at least) two of the new vertices of $\hat H$. These two vertices form a bad pair in $\GG$ as they both belong to the same set of $H$-copies (namely, $\hat H$ and $H_3$), in contradiction the stability of $\GG$.

We thus have $t = 1$, and $\Gamma = \Gamma_3 = H_1\cup H_2\cup H_3$.
Recall that $v(\GG_2) \le 6$. Since $H_3$ adds one new vertex to $\GG$, and since $v(\GG) > 6$ or otherwise it would be a small component, we conclude that $v(\GG_2) = 6$, and we can write
$V(H_1) = \{v_1, v_2, v_3, v_4\}$ and $V(H_2) = \{v_3, v_4, v_5, v_6\}$.
Note that in $\GG_2$ there is no edge between the pairs $\{v_1, v_2\}$ and $\{v_5, v_6\}$, so $H_3$ contains vertices from only one of these pairs. However, both pairs are bad pairs in $\GG_2$, so at least one of them remains a bad pair in $\GG$, which once again contradicts its stability.
\end{proof}

\subsection{Biased Maker-Breaker triangle games}\label{sec:(1:2)triangle}
Here we have $H' = K_3$ and $b \ge 2$. Similarly to the proof in Section~\ref{sec:largeCliques}, it suffices to prove that \whp Breaker wins in the $(1:2)$ triangle-game played on the vertex set of $G\sim G(n,cn^{-2/3})$ where $0 < c < 1$ is an arbitrary constant. Once again, by Lemma~\ref{lem:BwinG'} it suffices to show that \whp $G^*$, the $(K_3,2)$-core of $G$, is empty. The proof is complete by Claim~\ref{cl:noNegativeExponent} and the following lemma.

\begin{lemma}\label{lem:GisB'swinK3}
Let $\GG \in \xqts$ be a $(K_3,2)$-stable graph. Then $2q + s > 3$.
\end{lemma}

\begin{proof}
Assume for contradiction that $2q + s \le 3$. If $\GG$ does not contain any dangerous edges then it is a $K_3$-cycle of length at least four by Lemma~\ref{lem:noDang}. However, this graph is not stable: any vertex of degree 4 and its two neighbors of degree 2 form a bad set. Assume then that $\GG$ contains a dangerous edge. By Remark~\ref{rmk:1stable} and Corollary~\ref{cor:exploreK3} we may consider only greedy explorations of $\GG$. It follows that $s\ge 1$, and therefore either $q=s=1$ or $q=0,~s \le 3$. Claim~\ref{cl:11is30} then implies that we may consider only the latter case. We make use of the following observation.

\begin{observation}\label{obs:badTriplet}
If $U \sbst V(\GG)$ is a set of size 3, such that every vertex in $U$ has at most one neighbor outside $U$, then $U$ is a bad set. Indeed, every triangle containing a member of $U$ must contain two of its neighbors, and at least one of them belongs to $U$ as well.
\end{observation}

The exploration of $\GG$ starts with two triangles $H_1$ and $H_2$ sharing an edge, that is, $\GG_2$ is a diamond. Note that $v(\GG)  = 1+2t+s > 6$ or otherwise $\GG$ would be a small component. Since $s \le 3$ we get $t > 1$, meaning there is an addition of an external copy after $\GG_2$. Let $H'$ be the last external $K_3$-copy added in the exploration, and let $x,y$ be its new vertices. By Corollary~\ref{cor:noExtrnlCps} and the rules of the exploration, we must continue with the addition of an internal $K_3$-copy $H_3$, containing one new vertex $z$, and an existing edge $e \in E(H')$. By Observation~\ref{obs:badTriplet} the set $\{x,y,z\}$ is a bad set, thus the exploration cannot end after the addition of $H_3$.

Since at this point only internal copies can be added, and since $s \le 3$, we conclude that the remainder of the exploration after $\GG_2$ is the addition of a nonempty $K_3$-chain, originating at some $w \in V(\GG_2)$, followed by the addition of internal copies $H_3$ and $H_4$ in the last two steps. Note that before the last step both $\{x,y,z\}$ and $V(\GG_2) \stm \{w\}$ are bad sets by Observation~\ref{obs:badTriplet}, and that there are no edges between these triplets. Since $H_4$ must contain an existing edge when added, it cannot contain vertices from both triplets, so at least one set remains bad in $\GG$, a contradiction.
\end{proof}

\subsection{The unbiased Maker-Breaker triangle game}\label{sec:(1:1)triangle}
\begin{proof}[{Proof of Theorem~\ref{thm:triangle11}}]
%
    We begin our proof with two simple sufficient conditions for Maker's
    win in the $(1:1)$ triangle game played on the vertex set of an
    arbitrary graph $G$ -- one for the case he is the first player and
    one for the case he is the second player. If $G$ contains a
    $DD$-copy $\D$ and Maker plays first, he can claim the center of
    $\D$ in his first move and apply the natural pairing strategy on the
    remaining vertices of $\D$, thus ensuring having at least one vertex
    from each pair. This is indeed a winning strategy by
    Observation~\ref{obs:DDpairing}. Similarly, if $G$ contains two
    vertex disjoint $DD$-copies then Maker (even as a second player) can
    win by applying a pairing strategy on the two centers of these
    copies and on all of their natural pairs.

    In fact, Maker can play straightforward and claim $x$, $y_i$ and $z_j$ in his first three moves (in this order, and using the labeling of Figure~\ref{fig:DD}), where $i\in [2]$ and $j \in [4]$ are possibly determined by Breaker's moves. If Maker plays second, the $DD$-copy in which he claims his vertices is a copy whose all its vertices are still free after Breaker's first move. We later use the fact that Maker can achieve his goal already in his third move in the proof of Theorem~\ref{thm:infamily}.

    For the main part of the proof, let $\tilde G = \{G_i \}$ be a
    random graph process, let $i_1 = \tau (\tilde G, \mathcal
    G_{DD} ) - 1$ and $i_2 = \tau (\tilde G, \G_{2DD} )
    - 1$, and let $G^*_1$ and $G^*_2$ be the $(K_3, 1)$-cores of
    $G_{i_1}$ and $G_{i_2}$, respectively. Maker, as a first
    player, has a winning strategy for the game played on
    $V(G_{i_1+1})$. Note that since $DD$ is strictly balanced, the first
    two $DD$-copies appearing during the random graph process are \whp
    vertex disjoint by Corollary~\ref{cor:balancedDisjoint}, and
    therefore \whp Maker, as a second player, has a winning strategy for
    the game played on $V(G_{i_2+1})$. Since ``being Breaker's win" is a
    monotone decreasing graph property, in order to complete the proof
    of the theorem it is enough to show that \whp Breaker has a winning
    strategy as a second player in the game played on $V(G_{i_1})$, and
    as a first player in the game played on $V(G_{i_2})$. By
    Lemma~\ref{lem:BwinG'} it is enough to show that w.h.p., Breaker has a
    winning strategy as a second player for the game played on
    $V(G^*_1)$, and as a first player for the game
    played on $V(G^*_2)$.

    The key ingredient for the proof is Claim~\ref{cl:ConComp}.
    Recall the family $\F$ defined in that claim, of all graphs on less than 25 vertices with
    density larger than $10/7$.
    Since $\F$ is finite, and as $m(DD) = 10/7$, the graph $G_{i_2}$ is \whp $\F$-free by
    Claim~\ref{cl:sparseBeforeDense} and by definition of $i_2$. We therefore finish the proof by showing that
    if $G_{i_2}$ is indeed $\F$-free, then Breaker has winning strategies for
    the two games in discussion.

    Assume then that $G_{i_2}$ is $\mathcal F$-free, and note that by containment $G^*_1$ and $G^*_2$ are $\F$-free
    as well. Since $G^*_1$ is in addition $DD$-free, and $K_3$-stable by
    definition, Claim~\ref{cl:ConComp} implies that every connected
    component in $G^*_1$ is \whp a $\ttt$-copy, for which there exists a
    natural pairing strategy (given in Definition~\ref{def:trioPairing}). Similarly, $G_{i_2}$ has exactly one
    $DD$-copy $\D$, and so every connected component in $G^*_2$ other
    than $\D$ is a $\ttt$-copy. Breaker, as a first player, can claim the
    center of $\D$ in his first move (thus ensuring that Maker cannot
    claim any triangle contained in $\D$), and then apply the natural
    pairing strategy on all remaining components.
\end{proof}

\subsection{Estimating the constants in Theorem~\ref{thm:main}}\label{sec:constants}
We conclude this section with a short discussion about the two constants $c$ and $C$ appearing in Theorem~\ref{thm:main} for the case $H'=K_k$. We later refer to this discussion in Section~\ref{sec:concluding}. The proofs of the two parts of the 0-statement given in Sections~\ref{sec:largeCliques} and~\ref{sec:(1:2)triangle} work for any $c < 1$. It turns out that for any given $C > 1$, the 1-statement holds for any $k \ge k_0$, where $k_0$ is determined by $b$ and $C$.

More precisely, recall that the proof of the 1-statement of Theorem~\ref{thm:main} is based on the argument that  for some constant $C = C(b,k)$, if $p \ge Cn^{-2/k}$, then \whp in $\ggnp$ every $\frac
1{b+1}$-fraction of $V(G)$ induces a $K_k$-copy in $G$ and thus \whp Maker wins
the $(1:b)$ $K_k$-game no matter how he plays in this case. Let us
examine this constant $C$ a bit more carefully. Although not very
complicated, we omit most of the technical details below. In the proof
of Theorem~\ref{thm:LRV} the authors of~\cite{LRV} used the following
theorem, which was originally stated in a more general form.

\begin{theorem}[\cite{JLR}]\label{thm:JLR}
For every strictly balanced graph $H$ there exists a constant $c(H)$
such that for $\ggnp$ the probability that $G$ is $H$-free is at
most $e^{-c(H)\mu(H)}$, where $\mu(H)$ is the expected number of
$H$-copies in $G$.
\end{theorem}

One can verify that if $H$ is strictly 1-balanced and $p =
\Theta(n^{-1/m_1(H)})$, then the constant $c(H)$ from Theorem~\ref{thm:JLR} can be set
arbitrarily close to 1 from below (this follows from Janson's
inequality). Using this fact and conducting the calculations in the
proof of Theorem~\ref{thm:LRV}  more carefully, one obtains that for
$p \ge Cn^{-2/k}$ (where $C$ is undetermined yet), the probability
that there exists a $K_k$-free subgraph of $G$ induced by a $\frac
1{b+1}$-fraction of its vertices is at most $\exp\left(n\left(1 -
(bk)^{-k}C^{\binom k2}\right)\right)$, i.e.~this
probability is exponentially small whenever $C >
\left(bk\right)^{\frac{2}{k-1}}$. For any fixed $b$, the right hand side of the last inequality is a
monotone decreasing function of $k$, tending to 1 as $k$ tends to
infinity. It follows that indeed, for any fixed $C > 1$ and $b \ge 1$, the 1-statement holds if $k$ is large enough.

\section{Maker-Breaker cycle games}\label{sec:cycle}
In this section we prove the 0-statement of Theorem~\ref{thm:main} for $H'=C_k,~k \ge 4$. At the beginning we follow the ideas of Section~\ref{sec:genCliques} while paying attention to the computational and structural differences between the clique game and the cycle game, showing eventually that the stable components which are likely to appear in the random graph are quite limited in their structure. The analysis is much more delicate in this case though. From now on, fix an integer $k \ge 4$, let  $c < k^{-1/k}$ be a constant, and consider the random graph $\ggnp$ for $p = cn^{-1/m_1(C_k)} = cn^{-(k-1)/k}$.

We start with an equivalent of Claim~\ref{cl:NoChains}, showing that \whp $G$ does not contain any $C_k$-chain of length $d\lnn$ for $d = -2/\ln\left(kc^k\right)$ (note that $d$ is a positive constant as $kc^k < 1$).
First observe that for any given $t$ there are less than $k^t$ different $C_k$-chains of length $t$.
Indeed, given $t$, if for every $2 \le i \le t-1$ we denote by $d_i$ the distance on the $i$th cycle between the two vertices it shares with the previous and next cycles in the chain, then the structure of the chain is determined by $d_2,\dots,d_{t-1}$.
Since there are $\lfloor k/2 \rfloor$ options for each $d_i$ the upper bound on the number of different chains is established.
Now, if $\GG$ is a $C_k$-chain of length $t$ than we have $v(\GG) = 1 + t(k-1)$ and $e(\GG) = kt$.
Therefore, if $\T$ denotes the family of all $C_k$-chains of length $t = d\lnn$ we get
\begin{align*}
	\Pr[\textrm{$G$ is not $\T$-free}] & \le \sum_{\GG \in \T}\prgg \le \sum_{\GG \in \T}n^{v(\GG)}p^{e(\GG)} \\& \le k^tn^{1 + t(k-1)}\left(cn^{-(k-1)/k}\right)^{kt} = n\left(kc^k\right)^{d\lnn} = 1/n = o(1).
\end{align*}

Next, for fixed $q$ and $s$ we wish to bound the size of the families $\xqts$ containing no $C_k$-chains of length $d\lnn$, which are once again denoted by $\yqts$. This is done similarly to Claim~\ref{cl:yqsIsSmall}, but with a few differences. First, as already argued, determining the length of each of the $C_k$-chains in a component $\GG \in \yqts$ is not enough and we have to account for the number of all different choices for their structures. For a given $t$ (that is, the sum of lengths of all $C_k$ chains in the exploration), by using basically the same argument as above, we can upper bound this number by $k^t$. Consequently, we bound separately the size of each family $\yqts$, and not the size of their union over all values of $t$. However, while $q$ and $s$ are fixed, we allow $t$ to grow to infinity with $n$. Lastly, when choosing existing vertices for an internal cycle we also have to choose their location on the cycle. This results in a factor of at most $(k!)^s$.
For all other considerations, almost identical calculations to those in Claim~\ref{cl:yqsIsSmall} yield the factor $O\left((\ln n)^{3q + ks}\right)$.
All in all we get that
\begin{equation}\label{eq:yqtscycle}
	|\yqts| = O\left(k^t(\ln n)^{3q + ks}\right)
\end{equation} holds for any two fixed integers $q,s$ and any integer $t:=t(n)$.

When considering the cycle game, knowing that a component belongs to $\xqts$ is not enough for us and we need more information.
For every non-negative integer $s$ and for every $s$-tuple $\sss = ((v_1,e_1),\dots,(v_s,e_s))$ (including the empty tuple if $s = 0$), let $\xqtss$ be the subset of all graphs in $\xqts$ with the additional condition that the $i$th internal $C_k$-copy added during the exploration process contains exactly $v_i$ existing vertices and $e_i$ existing edges. The families $\yqtss$ are defined in a similar fashion. Mind the abuse of notation here, where in order to reduce the number of different variables used, the integer $s$ denotes the length of $\sss$. Whenever discussing such an $s$-tuple, we assume it is ``legal". In particular, for every $i \in [s]$ we have $1 < v_i < k$, since it represents an internal copy, and $v_i > e_i$, since the existing vertices and edges of each internal $C_k$-copy always form vertex-disjoint paths.

%

Now we wish to prove an analogue of Claim~\ref{cl:noNegativeExponent} that characterizes the families $\xqtss$ which are likely to appear in $G$. Once again, we do this by showing that there exists a family of relatively small size, such that $G$ \whp does not contain any of its members, but each graph obtained via an exploration process which does not meet some condition on $q$ and $\sss$ must contain one of its members as a subgraph. In order to formulate this condition we define $$f(\sss) := \sum_{i=1}^s (k(v_i-e_i-1) + e_i).$$
Note that we cannot have $v_i=1,~e_i=0$ for the same $i\in[s]$, as these values represent the addition of an \emph{external} copy, and thus for any given $\sss$, every summand in $f(\sss)$ is a positive integer.


\begin{claim}\label{Ap:cl:noNegativeExponentCk}
	The following holds w.h.p.: all families $\xqtss$ for which at least one of their members appears in $G$ satisfy $(k-1)q + f(\sss) \leq k$.
\end{claim}

\begin{proof}
	First, we consider only the families $\yqtss$, since \whp there exist no long $C_k$-chains in $G$.
	Now, for every non-negative integer $s$ and a corresponding $s$-tuple $\sss$, and for every $\GG \in \yqtss$, we have $v(\GG) = 1 + t(k-1) + \sum_{i=1}^s (k-v_i)$ and $e(\GG) = tk + q + \sum_{i=1}^s (k-e_i)$.
	Hence,
	\begin{align}\label{eq:cycleBound}
		\prgg & \le  n^{v(\GG)}p^{e(\GG)} \nonumber \\
		& =  c^{e(\GG)}n^{1+t(k-1)+\sum_{i=1}^s (k-v_i) - \frac{k-1}{k}\left(tk + q + \sum_{i=1}^s k-e_i \right)} \nonumber \\
		& \le  c^{kt}n^{1-\frac 1k\left[(k-1)q + \sum_{i=1}^s ((k-1)(k-e_i)-k(k-v_i))\right]} \nonumber \\
		& =  c^{kt}n^{1-\frac 1k\left[(k-1)q + f(\sss)\right]}.
	\end{align}
	
	Let $q$, $t$ and $\sss  = ((v_1,e_1),\dots,(v_s,e_s))$ such that $(k-1)q + f(\sss) > k$, and let  $\GG \in \yqtss$.
	For every $i \in [s]$ let $\ssi{i} = ((v_1,e_1),\dots,(v_i,e_i))$ be the $i$th prefix of $\sss$.
	Consider an arbitrary exploration of $\GG$, and let $q_j$ and $\gamma_j$ be the number of internal edges and internal copies added during the first $j$ steps of the exploration, respectively.
	Let $i$ be the maximal integer such that $(k-1)q_i + f(\ssi{\gamma_i}) \le k$. Such an $i$ exists since the first step in any exploration is the addition of an external copy. 
	Clearly we have $q_i \le 1$ and $f(\ssi{\gamma_i}) \le k$. Recalling that every summand in $f(\sss)$ is a positive integer, the latter inequality implies $\gamma_i \le k$.
	By these observations and the definition of $i$, we have $(k-1)q_{i+1} + f(\ssi{\gamma_{i+1}}) > k$, while $q_{i+1} \le 2$ and $\gamma_{i+1} \le k+1$.
	Let
	$$Y_{k,t} = \bigcup_{\substack{q, \sss \textrm{~s.t.~} q \le 2, s \le k+1 \\ \textrm{and~} (k-1)q + f(\sss) > k}} \yqtss$$
	and
	$$Y_k = \bigcup_{t \ge 1} Y_{k,t}.$$
	Then $\GG_{i+1} \in Y_k$, and so it suffices to show that \whp $G$ is $Y_k$-free.
	Clearly, for every $t$ we have
	$$Y_{k,t} \sbst \bigcup_{q \le 2,\ s\le k+1} \yqts,$$
	and so $|Y_{k,t}| \le k^t(\ln n)^{k^3}$ by~(\ref{eq:yqtscycle}). Note also that $Y_{k,t} = \emptyset$ for every $t > (q+s)d\lnn$ by the restriction on the maximal length of a $C_k$-chain.
	By~(\ref{eq:cycleBound}) we have that $\prgg \le c^{kt}n^{-1/k}$ for every $\GG \in Y_{k,t}$. Applying the union bound  and using the fact that $kc^k < 1$ we obtain
	$$\Pr[\textrm{$G$ is not $Y_k$-free}]
	\le \sum_{t \ge 1}\sum_{\GG \in Y_{k,t}}\prgg \le \sum_{t =
		1}^{(q+s)d\lnn}(\lnn)^{k^3}k^tc^{kt}n^{-\frac 1k}
	\le (\lnn)^{k^4}n^{-\frac 1k} = o(1),$$
	which completes the proof.
\end{proof}

In the light of Claim~\ref{Ap:cl:noNegativeExponentCk} we define the following.

\begin{definition}\label{Ap:def:feasible}
	A nonempty connected component $\GG$ is called \emph{feasible} if it is $(C_k,1)$-stable, and any $q,t,\sss$ for which $\GG \in \xqtss$ satisfy $(k-1)q + f(\sss) \leq k$.
\end{definition}

Now, in a similar fashion to that of the proof given in Section~\ref{sec:largeCliques}, we focus on the unbiased $C_k$-game for $k\geq 4$ played on a feasible component, and from this we eventually deduce the more general result. Our main and most basic tool is the fact that stable components cannot contain any bad pairs. In our analysis we encounter two main types of bad pairs, the first is defined as follows.

\begin{definition}
	Let $G$ be a graph.	Two vertices $u,v \in V(G)$ are called \emph{evil twins} if $uv \in E(G)$ and $d_G(u)=d_G(v)=2$.
\end{definition}

It is easy to see that if $u,v$ are evil twins then they indeed form a bad pair: any $C_k$-copy containing one of these vertices must use the two edges incident to it, and specifically the edge $uv$, implying that no $C_k$-copy can contain exactly one of $u,v$. Note that evil twins remain such as long as no edge incident to one of them is added to the graph, which is not the case for bad pairs in general. Let $\GG_i \sbst \GG_j$ be two graphs obtained during an exploration of a feasible component $\GG$ for some $i<j$. If $\GG_i$ contains a pair of evil twins, which remain evil twins in $\GG_j$, we say that the pair \emph{survives} in $\GG_j$.

The other type of bad pairs we wish to introduce is described in the following claim.

\begin{claim}\label{Ap:cl:generalizedTwins}
	Let $G$ be a graph and let $x,y \in V(G)$ be two vertices of degree 3, such that $G$ contains two distinct paths $P_1,P_2$ between $x$ and $y$, and none of these paths contains an internal vertex of degree larger than 2. Then $\{x,y\}$ is a bad pair.
\end{claim}

\begin{proof}
	Every $C_k$-copy containing either $x$ or $y$ contains two of the three edges incident to that vertex, and so must contain at least one edge belonging to either $P_1$ or $P_2$. Since that path is either a single edge or contains only internal vertices of degree 2, the $C_k$-copy must contain the entire path, and in particular both its endpoints $x$ and $y$. It means that no $C_k$-copy contains exactly one of $\{x,y\}$, which is what we had to show.
\end{proof}

We now begin with a series of claims and corollaries, gradually revealing more and more information about the possible structures of feasible components and the restrictions applied to them. First we wish to make the restriction $(k-1)q + f(\sss) \leq k$ much more explicit and convenient.

\begin{claim}\label{Ap:cl:sgeq1}
	Let $\GG$ be a feasible component. In any exploration of $\GG$ in which $s\geq 1$, exactly one of the following holds.
	\begin{enumerate}[(i)]
		\item $q=0,~s=1,~v_1=2,~e_1=0$.
		\item $q=0,~\sum_{i=1}^s e_i \leq k$, and $e_i = v_i-1 \ge 1$ for every $i\in [s]$.
		\item $q=1,~s=1,~v_1=2,~e_1=1$.
	\end{enumerate}
\end{claim}

\begin{proof}
	Recall that $f(\sss) = \sum_{i=1}^s (k(v_i - e_i - 1) + e_i)$.
	If $v_{i_0} > e_{i_0}+1$ for some $i_0\in [s]$, then $f(\sss) \geq k(v_{i_0} - e_{i_0}-1) \geq k$. Since $\GG$ is feasible and therefore $(k-1)q + f(\sss) \le k$, it follows that in this case $q=0$ and $f(\sss)=k$, and thus by the positivity of the summands of $f(\sss)$ we conclude that $s=1,~v_1=2,~e_1=0$, which is exactly Option~$(i)$.
	
	Assume now that $v_i = e_i + 1$ for every $i\in [s]$, and therefore $f(\sss) = \sum_{i=1}^{s} e_i$, and recall that $e_i \ge 1$ for every $i \in [s]$ in this case.
	Since $(k-1)q + \sum_{i=1}^{s} e_i \leq k$ holds there are only two possibilities.
	If $q=0$ then the restriction $\sum_{i=1}^s e_i \leq k$ remains as is, and we got Option~$(ii)$.
	If $q=1$ then $\sum_{i=1}^s e_i \leq 1$, which can only happen if $s=1,~v_1 = 2,~e_1=1$, that is, Option~$(iii)$.
\end{proof}

Our next step is to show that not surprisingly, dangerous edges are crucial for stability.

\begin{claim}\label{Ap:cl:CycleDanEdge}
	Any feasible component contains at least one dangerous edge.
\end{claim}

\begin{proof}
	Assume for contradiction that $\GG$ is a feasible component containing no dangerous edges.
	Recall that by Corollary~\ref{cor:noExtrnlCps} the last step of the exploration cannot be an addition of an external $C_k$-copy.
	However, it cannot be an addition of an internal edge $e$ either.
	Indeed, by stability $e$ must be a part of some $C_k$-copy, which contains an existing edge $e'$. Since  $q\leq 1$ by the feasibility of $\GG$, the edge $e'$ must have been added to $\GG$ as a new edge with some previous $C_k$-copy. This makes $e'$ a dangerous edge, a contradiction.
	
	The exploration therefore terminates with the addition of an internal $C_k$-copy $H'$, and in particular we get $s\geq 1$.
	Moreover, when adding an internal $C_k$-copy, every existing edge that this copy contains must have been added previously as an internal edge, or otherwise we get a dangerous edge.
	Hence, if $e_i \geq 1$ for some $i\in [s]$, then $q\geq 1$ must hold as well.
	Claim~\ref{Ap:cl:sgeq1} therefore implies that $s=1$ and $v_1=2$, since Option~$(ii)$ of the claim leads to a contradiction.
	
	It follows that $H'$ which is added as the last step of the exploration contains $k-2$ new vertices. Let $u,v$ be two of these vertices, and note that they both have degree 2 in $\GG$. The two edges incident to each of them belong to $H'$ and to no other $H$-copy in $\GG$ by assumption, implying that the same holds for $u$ and $v$ as well, which makes them a bad pair in contradiction.
\end{proof}

The existence of dangerous edges in feasible components suggests we may focus on greedy explorations, in which the second exploration step is the addition of an internal $C_k$-copy containing at least one existing edge (recall Definition~\ref{def:greedy}). For this we need the following claim.




\begin{claim}\label{Ap:cl:greedy}
	There exists a greedy exploration for any feasible component.
\end{claim}
\begin{proof}
	Let $\GG$ be a feasible component. Since $\GG$ contains a dangerous edge by Claim~\ref{Ap:cl:CycleDanEdge}, we can start the exploration with two $C_k$-copies $H_1,H_2$ containing this edge. It only remains to verify that no internal edges need to be added prior to $H_2$. Assume otherwise, and note that in this case we must have $q=s=1$ by Claim~\ref{Ap:cl:sgeq1}. That is, one internal edge is added to $H_1$, and immediately afterwards $H_2$ is added. But at this point only external copies may be added, implying that $\GG_3 = \GG$, which is impossible since $\GG_3$ is a small component.
\end{proof}

Being restricted to greedy explorations may affect the possible values of $q,t,s$, but since we consider only feasible components, and all of them, this restriction does not affect the analysis. Therefore, from now on we only consider greedy explorations. Since the second step ensures that $s\geq 1$ and $e_1\geq 1$, Option~$(i)$ of Claim~\ref{Ap:cl:sgeq1} is not possible. It turns out that the same holds for Option~$(iii)$.

\begin{claim}\label{Ap:cl:Ckq=0}
	Let $\GG$ be a feasible component. Then $q=0$ holds for any greedy exploration of $\GG$.
\end{claim}

\begin{proof}
	Assume otherwise for contradiction. Then $q=1, s=1, v_1=2, e_1=1$ must hold by Claim~\ref{Ap:cl:sgeq1}, and $t>1$ must hold as otherwise $\GG$ would be a small component. By definition of greedy explorations, and since no exploration can end with the addition of an external copy, we conclude that any greedy exploration of $\GG$ with $q>0$ must go as follows. First the external copy $H_1$ is added, then the internal copy $H_2$ is added, sharing exactly one edge and its two endpoints with $H_1$. Then a non-empty $C_k$-chain $\C$ is added, ending with an external copy $H_3$, and finally an internal edge $e$ is added. Recall that $\GG_2 = H_1 \cup H_2$, and let $\GG'$ be the explored graph just before the addition of $e$.
	
	For $i=1,2$, let $V_i \sbst V(H_i)$ be the set of $k-2$ vertices of degree 2 in $\GG_2$, and let $z$ and $V_3$ be the existing vertex and the set of new vertices added by $H_3$, respectively. Note that both $V_1$ and $V_2$ contain evil twins in $\GG_2$, and since the intersection of $\C$ and $\GG_2$ is a single vertex $w$ (satisfying $w=z$ in case $t=2$), at least one of these sets --- assume WLOG it is $V_1$ --- still contains evil twins in $\GG'$. Since $V_3$ also contains evil twins in $\GG'$, and since $\GG$ cannot contain any evil twins, it follows that $w \in V_2$, and that $e = uv$ for some $u \in V_1$ and $v \in V_3$. Moreover, we may assume that $k=4$ and that $v$ is the non-neighbor of $z$ in $H_3$, as otherwise $V_3$ still contains evil twins in $\GG$.
	
	But now we have a contradiction: by stability $uv$ must be part of a $C_4$-copy, meaning that in $\GG'$ there exists a path of length 3 between $u$ and $v$. However, any such path must contain $z$, which is of distance at least two from both $u$ and $v$ (since there are no edges between $V_1$ and $V_2$, and therefore no edges between $u$ and any vertex of $\C$).
\end{proof}


It remains to analyze greedy explorations in which $q=0,~\sum_{i=1}^s e_i \leq k$, and $e_i = v_i-1$ for every $i\in [s]$. The fact that $q=0$ means that each step of the exploration is the addition of a $C_k$-copy, either external or internal. The restriction $e_i=v_i-1$ means that the existing part of each internal copy is a single path, which allows us to describe the addition of internal copies very precisely. Suppose that the $i$th internal $C_k$-copy $H'$ is added during the exploration to $\GG' \sbst \GG$. This must be done by adding a path of $k-e_i \ge 2$ new edges (recall that $e_i <v_i < k$), where all internal vertices of the path are new vertices, and its two endpoints $x,y$ are existing vertices such that $\GG'$ contains a path of length $e_i$ between them. We say that $H'$ is \emph{attached} to $x$ and $y$. As an immediate corollary of Claim~\ref{Ap:cl:generalizedTwins}, we get the following restriction on the vertices the last added copy is attached to.

\begin{corollary}\label{Ap:cor:attachedTo1}
	Let $\GG$ be a feasible component, and let $H'$ and $H''$ be the penultimate and last $C_k$-copies added in a greedy exploration of it, respectively. Then $H''$ is attached to at most one new vertex of $H'$.
\end{corollary}

\begin{proof}
	Assume otherwise, and let $x,y$ be the two vertices $H''$ is attached to. Then $d_\GG(x) = d_\GG(y) = 3$, and there exist two paths $P_1,P_2 \sbst \GG$ between $x$ and $y$ with no internal vertices of degree larger than~2, the one added by $H'$ and the one added by $H''$. Thus, the conditions for Claim~\ref{Ap:cl:generalizedTwins} hold, and $\{x,y\}$ is a bad pair, a contradiction.
\end{proof}

By the general restrictions of the exploration process, we have that in any greedy exploration the first step is the addition of an external copy, while the second step and the last step (which must be two different steps because otherwise $\GG = \GG_2$ is a small component) are the additions of internal copies, implying in particular $s>1$. Given the above characterization of internal copy additions, we immediately deduce that $e_s = k-2$, that is, only one new vertex is added in the last step of the exploration, as otherwise this step will add evil twins to the graph (in any step of the exploration, each pair of adjacent new vertices is a pair of evil twins). It follows that $\sum_{i=1}^{s-1} e_i \leq 2$, leaving no much options for $\sss$.

\begin{corollary}\label{Ap:cor:3cases}
	In any greedy exploration of a feasible component exactly one of the following holds.
	\begin{enumerate}[$(a)$]
		\item $s=2,~e_1 = 2,~e_2 = k-2,~t>1$.
		\item $s=2,~e_1 = 1,~e_2 = k-2$.
		\item $s=3,~e_1 = 1,~e_2 = 1,~e_3 = k-2$.
	\end{enumerate}
\end{corollary}

\begin{proof}
	We already know that only Option~$(ii)$ of Claim~\ref{Ap:cl:sgeq1} is possible, and that $e_s = k-2$. Thus it only remains to explain why $t>1$ in Case~$(a)$. If not, then the exploration lasts only three steps, where the second step adds $k-3$ new vertices and the third step adds just one, resulting in a small component in contradiction.
\end{proof}

The significant next claim narrows down our analysis to the case $k=4$.

\begin{claim}\label{Ap:cl:noFeasible}
	There exist no feasible components for $k\geq 5$.
\end{claim}

\begin{proof}
	Let $k \ge 5$, assume for contradiction that $\GG$ is a feasible component and consider a greedy exploration of $\GG$. Let $H'$ and $H''$ be the penultimate and last $C_k$-copies added, respectively. First observe that Case~$(a)$ of Corollary~\ref{Ap:cor:3cases} can be excluded. Indeed, in this case $H'$ is an external copy, meaning it adds a path of $k-1 \ge 4$ new vertices. By Corollary~\ref{Ap:cor:attachedTo1} the last copy $H''$ is attached to at most one of them, implying there remain evil twins in $\GG$.
	
	Since only Cases~$(b)$ and~$(c)$ are possible, and in both $e_1 = 1$, it follows that in \emph{any} greedy exploration of $\GG$, the first $C_k$-copies $H_1$ and $H_2$ share exactly one edge. Hence, $H'$ and $H''$ also share at most one edge, as otherwise we could start the exploration with these two copies (this is guaranteed by Claim~\ref{Ap:cl:Ckq=0}). Using this fact, Corollary~\ref{Ap:cor:attachedTo1}, and the fact that $e_{s-1} = 1$, we get that $H'$ adds a path of at least $k-2 \ge 3$ new vertices, and either $H''$ is attached to none of them, or it is attached to a new vertex adjacent to one of the existing vertices of $H'$. In either case we get that at least one pair of evil twins added by $H'$ survives in $\GG$, a contradiction.
\end{proof}

Before performing a detailed case analysis for $k=4$, and in order to simplify it, we first consider the endings of greedy explorations in the following two simple claims.

\begin{claim}\label{Ap:cl:penExternal}
	Let $\GG$ be a feasible component. Suppose that in a greedy exploration of $\GG$ the penultimate step is the addition of an external $C_4$-copy $H'$. Let $x$ be the existing vertex of $H'$, and let $y$ be the new vertex of $H'$ who is not adjacent to $x$. Then the internal $C_4$-copy which is added in the last step is attached to $x$ and $y$.
\end{claim}

\begin{proof}
	By Corollary~\ref{Ap:cor:attachedTo1}, the last copy added $H''$ is attached to at most one of the new vertices of $H'$, and so unless $H''$ is attached to $y$, evil twins will remain in $\GG$ (the vertex $y$ and one of its neighbors). Additionally, before $H''$ is added there must be a path of length 2 between $y$ and the other vertex $H''$ is attached to, and $x$ is the only vertex meeting this requirement.
\end{proof}

\begin{claim}\label{Ap:cl:sameDegree}
	Let $\GG$ be a feasible component. Suppose that $\GG' \sbst \GG$ is a graph obtained during a greedy exploration of $\GG$, such that the remainder of the exploration is the addition of a nonempty $C_4$-chain and then the addition of an internal $C_4$-copy. Let $u$ be the existing vertex of the first external copy in the chain. Then $d_{\GG}(v) = d_{\GG'}(v)$ for every $v \in V(\GG') \stm \{u\}$.
\end{claim}

\begin{proof}
	Clearly $u$ is the only vertex in $V(\GG')$ whose degree is increased by the addition of the $C_4$-chain. By Claim~\ref{Ap:cl:penExternal}, the internal $C_4$-copy added last is attached to two vertices belonging to the vertex set of the last external copy in the chain. So except perhaps $u$, no vertex of $V(\GG')$ is affected by the last step as well.
\end{proof}

We have come to the last claim of this section, showing that although feasible components do exist for $k=4$, this is no obstacle for Breaker.

\begin{claim}\label{Ap:cl:BwinC4}
	Breaker wins the unbiased $C_4$-game played on any feasible component.
\end{claim}

\begin{proof}
	Let $k=4$ and let $\GG$ be a feasible component. We consider all possible greedy explorations of $\GG$ and separate the proof according to the three cases of Corollary~\ref{Ap:cor:3cases}. We show that there exists a unique feasible component for each case, which is an easy win for Breaker via a pairing strategy. Recall that
	$H_1$ is the external copy initiating the exploration, $H_2$ is the internal copy added in the second step, and $\GG_2=H_1 \cup H_2$. Let $H_3$ be the second internal copy added during the exploration, and in case $s=3$ (that is, Case~$(c)$), let $H_4$ be the third internal copy added (and finishes the exploration). Keep in mind that $e_s=k-2=2$.
	
	\case{(a)}{s=2,~e_1 = e_2 = 2,~t>1}
	Since $t>1$ and $s=2$, and the exploration terminates with the addition of the internal copy $H_3$, an external $C_4$-copy $H'$ is added in the third step of the exploration. Let $w$ be the existing vertex of $H'$ and let $u,v$ be the two vertices of degree 3 in $\GG_2$. First observe that $w \in \{u,v\}$. Indeed, $\GG_2$ consists of three paths of length 2 between $u$ and $v$, where the internal vertex in each path is of degree 2. By Claim~\ref{Ap:cl:sameDegree}, every vertex in $V(\GG_2) \stm \{w\}$ remains with the same degree in $\GG$. If $w \not\in \{u,v\}$, then in $\GG$ all conditions of Claim~\ref{Ap:cl:generalizedTwins} hold, and $\{u,v\}$ is a bad pair, a contradiction.
	%
	%
	%
	%
	
	Next, assume for contradiction that $t>2$ and thus an external copy $H''$ is added in the fourth step. By the restrictions of the exploration process, the existing vertex of $H''$ is one of the new vertices of $H'$. Let $V(H') = \{w,x,y,z\}$, where $w$ and $z$ denote the existing vertices of $H'$ and $H''$, respectively. By Claim~\ref{Ap:cl:sameDegree} we have $d_\GG(x) = d_\GG(y) = 2$, making $x$ and $y$ evil twins if they are adjacent. If they are not, then every $C_4$-copy containing one of them must contain its only two neighbors $w$ and $z$ as well. On the other hand, as $wz \not\in E(\GG)$ in this case, any $C_4$-copy containing both $w$ and $z$ must contain their only two common neighbors $x$ and $y$. It follows that both $x$ and $y$ belong to no other $C_4$-copy but $H'$, which makes them a bad pair.
	
	In conclusion, if an external copy is added after $H'$ then two of the new vertices of $H'$ will be a bad pair in $\GG$, a contradiction. Thus $t=2$ and $H_3$ is added in the fourth and last step, and by Claim~\ref{Ap:cl:penExternal} it is attached to the existing vertex $w$ of $H'$ and to the non-neighbor of $w$ in $H'$. The resulting graph is $\GG^{(a)}$ shown in Figure~\ref{fig:feasible}. Breaker can win the game played on this graph by following the pairing strategy with respect to the pairs $\Lambda = \left\{\{x_1, x_2\}, \{y_1, y_2\}, \{z_1, z_2\} \right\}$.
	
	\case{(b)}{s=2,~e_1 = 1,~e_2 = 2}
	Since $e_1=1$, in both $H_1$ and $H_2$ the two vertices not in the intersection are evil twins. By Claim~\ref{Ap:cl:sameDegree}, if the third exploration step is the addition of an external $C_4$-copy, then at least one of these pairs survives in $\GG$. Thus, in the third and last step $H_3$ is added, and in such a way it is attached to one vertex from each pair of evil twins. Clearly the only possible option is $\GG^{(b)}$ shown in Figure~\ref{fig:feasible}.
	Breaker can win the game by following the pairing strategy with respect to the pairs $\Lambda = \left\{\{x_1, x_2\}, \{y_1, y_2\} \right\}$.
	
	\case{(c)}{s=3,~e_1 = e_2 = 1,~e_3 = 2}
	As in Case~$(b)$, in $\GG_2$ there exist two pairs of evil twins, $T_1$ and $T_2$, which are the two vertices of degree 2 in $H_1$ and $H_2$, respectively. At least one of these pairs survives in $\GG'$, the graph obtained after the addition of $H_3$. Indeed, $e_2 = 1$ and so $H_3$ is attached to two adjacent vertices. Now, if $H_3$ is added in the third step it cannot be attached to one vertex from each $T_i$. Otherwise, $H_3$ is added after an external copy $H'$, and by the restriction of the exploration must be attached to a new vertex of $H'$ and one of its neighbors, which obviously must also belong to $V(H')$. Thus, in this case only one vertex in $V(\GG_2)$ has larger degree in $\GG'$ than in $\GG_2$.
	
	Assume then WLOG that the pair $T_1$ survives in $\GG'$, and note that the two new vertices added by $H_3$ are evil twins in $\GG'$ as well. Denote this pair by $T_3$. By Claim~\ref{Ap:cl:sameDegree}, if the next addition to $\GG'$ is of an external $C_4$-copy, then at least one of the pairs $T_1$ and $T_3$ survives in $\GG$. It follows that $H_4$ is added immediately after $H_3$, and must be attached to one vertex from each of the pairs $T_1$ and $T_3$. Hence, there must be a path of length 2 in $\GG'$ between these two vertices, that is, they must have a common neighbor. This common neighbor must belong to $V(H_1)\cap V(H_3)$, since for $i=1,3$, the vertices of $T_i$ do not have any neighbors outside $H_i$.
	
	The only way to meet these requirements, as well as the requirement that $T_2$ does not survive in $\GG$, is to add $H_3$ in the third step, and to attach it to a vertex of $T_2$ and to its neighbor in $V(H_1)\cap V(H_2)$. This yields the graph $\GG^{(c)}$ as shown in Figure~\ref{fig:feasible}. Breaker can win the game by following the pairing strategy with respect to the pairs $\Lambda = \left\{\{x_1, x_2\}, \{y_1, y_2\}, \{z_1, z_2\} \right\}$.
\end{proof}

\iffigure
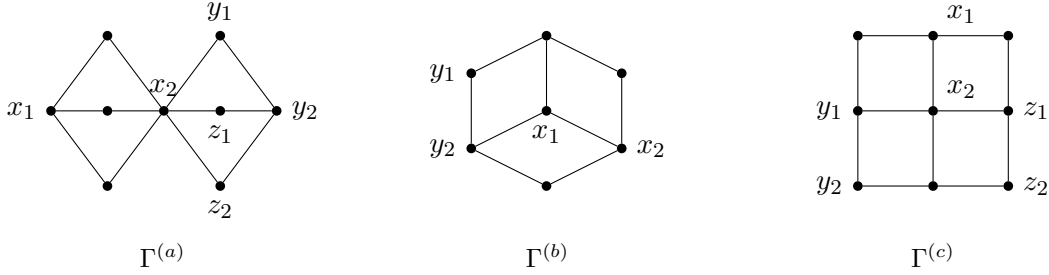
\begin{figure}
		\centering
	\begin{subfigure}[t]{0.3\textwidth}
			\centering
		\begin{tikzpicture}[auto, vertex/.style={circle,draw=black!100,fill=black!100, thick,
			inner sep=0pt,minimum size=1mm}]
		\node (x1) at ( 0.75,1) [vertex] {};
		\node (v1) at ( 0,0) [vertex,label=left:$x_1$] {};
		\node (x2) at ( 0.75,0) [vertex] {};
		\node (v2) at ( 1.5,0) [vertex,label=above:$x_2$] {};
		\node (z) at ( 2.25,0) [vertex,label=below:$z_1$] {};
		\node (x3) at ( 0.75,-1) [vertex] {};
		\node (x) at ( 2.25,1) [vertex,label=above:$y_1$] {};
		\node (w) at ( 3,0) [vertex,label=right:$y_2$] {};
		\node (y) at ( 2.25,-1) [vertex,label=below:$z_2$] {};
		
		\draw [-] (x1) --node[inner sep=0pt,swap]{} (v1);
		\draw [-] (v1) --node[inner sep=2pt,swap]{} (x2);
		\draw [-] (x2) --node[inner sep=2pt,swap]{} (v2);
		\draw [-] (x1) --node[inner sep=0pt,swap]{} (v2);
		\draw [-] (v1) --node[inner sep=0pt,swap]{} (x3);
		\draw [-] (x3) --node[inner sep=0pt,swap]{} (v2);
		\draw [-] (x) --node[inner sep=0pt,swap]{} (v2);
		\draw [-] (y) --node[inner sep=0pt,swap]{} (v2);
		\draw [-] (w) --node[inner sep=0pt,swap]{} (z);
		\draw [-] (z) --node[inner sep=0pt,swap]{} (v2);
		\draw [-] (x) --node[inner sep=0pt,swap]{} (w);
		\draw [-] (y) --node[inner sep=0pt,swap]{} (w);
		\end{tikzpicture}
		\caption*{$\GG^{(a)}$}
	\end{subfigure}
	\begin{subfigure}[t]{0.3\textwidth}
		\centering
		\begin{tikzpicture}[auto, vertex/.style={circle,draw=black!100,fill=black!100, thick,
			inner sep=0pt,minimum size=1mm}]
		\node (v0) at (0,0) [vertex,label=below:$x_1$] {};
		\node (v1) at (0,1) [vertex] {};
		\node (v2) at (1,0.5) [vertex] {};
		\node (v3) at (1,-0.5) [vertex,label=right:$x_2$] {};
		\node (v4) at (0,-1) [vertex,label=below:\phantom{$x_3$}] {};
		\node (v5) at (-1,-0.5) [vertex,label=left:$y_2$] {};
		\node (v6) at (-1,0.5) [vertex,label=left:$y_1$] {};
		
		\draw [-] (v0) --node[inner sep=0pt,swap]{} (v1);
		\draw [-] (v0) --node[inner sep=0pt,swap]{} (v3);
		\draw [-] (v0) --node[inner sep=0pt,swap]{} (v5);
		\draw [-] (v1) --node[inner sep=0pt,swap]{} (v2);
		\draw [-] (v2) --node[inner sep=0pt,swap]{} (v3);
		\draw [-] (v3) --node[inner sep=0pt,swap]{} (v4);
		\draw [-] (v4) --node[inner sep=0pt,swap]{} (v5);
		\draw [-] (v5) --node[inner sep=0pt,swap]{} (v6);
		\draw [-] (v6) --node[inner sep=0pt,swap]{} (v1);
		\end{tikzpicture}
		\caption*{$\GG^{(b)}$}
	\end{subfigure}
	\begin{subfigure}[t]{0.3\textwidth}
		\centering
		\begin{tikzpicture}[auto, vertex/.style={circle,draw=black!100,fill=black!100, thick,
			inner sep=0pt,minimum size=1mm}]
		\node (v1) at ( -1,1) [vertex] {};
		\node (v2) at ( 0,1) [vertex,label=400:$x_1$] {};
		\node (v3) at ( 1,1) [vertex] {};
		\node (v4) at ( -1,0) [vertex,label=left:$y_1$] {};
		\node (v5) at ( 0,0) [vertex,label=400:$x_2$] {};
		\node (v6) at ( 1,0) [vertex,label=right:$z_1$] {};
		\node (v7) at ( -1,-1) [vertex,label=left:$y_2$] {};
		\node (v8) at ( 0,-1) [vertex,label=below:\phantom{$x_3$}] {};
		\node (v9) at ( 1,-1) [vertex,label=right:$z_2$] {};
		
		\draw [-] (v1) --node[inner sep=0pt,swap]{} (v2);
		\draw [-] (v1) --node[inner sep=2pt,swap]{} (v4);
		\draw [-] (v2) --node[inner sep=0pt,swap]{} (v3);
		\draw [-] (v2) --node[inner sep=0pt,swap]{} (v5);
		\draw [-] (v3) --node[inner sep=0pt,swap]{} (v6);
		\draw [-] (v4) --node[inner sep=0pt,swap]{} (v5);
		\draw [-] (v4) --node[inner sep=0pt,swap]{} (v7);
		\draw [-] (v5) --node[inner sep=0pt,swap]{} (v6);
		\draw [-] (v5) --node[inner sep=0pt,swap]{} (v8);
		\draw [-] (v6) --node[inner sep=0pt,swap]{} (v9);
		\draw [-] (v7) --node[inner sep=0pt,swap]{} (v8);
		\draw [-] (v8) --node[inner sep=0pt,swap]{} (v9);
		\end{tikzpicture}
		\caption*{$\GG^{(c)}$}
	\end{subfigure}
\caption{The three feasible components for $k=4$}\label{fig:feasible}
\end{figure}
\fi

\vspace{3mm}

Claim~\ref{Ap:cl:BwinC4} provides the last missing piece of the puzzle and we can finally prove formally the main result of this section.

\begin{proof}[Proof of the $0$-statement of Theorem~\ref{thm:main} for $H' = C_k,~k \ge 4$]
	Recall that we consider $(1:b)$ $H$-games, and that $H'$ is a subgraph of $H$ of maximal 1-density.
	It suffices to show that breaker can win the $C_k$-game, as this ensures his win in the $H$-game as well.
	By bias monotonicity, it is also sufficient to prove Breaker's win for the case $b=1$.
	Fix a positive constant $c < k^{-1/k}$ and let $\ggn{cn^{-(k-1)/k}}$. By Lemma~\ref{lem:BwinG'} it suffices to show that \whp Breaker can win the game played on the vertex set of $G^*$, the $(C_k,1)$-core of $G$.
	By Claim~\ref{Ap:cl:noNegativeExponentCk} and Definition~\ref{Ap:def:feasible}, \whp $G^*$ is either empty or contains only feasible components. Claim~\ref{Ap:cl:noFeasible} implies that for $k \ge 5$ \whp $G^*$ is empty, and thus Breaker wins trivially. Claim~\ref{Ap:cl:BwinC4} shows that for $k=4$ \whp Breaker has a winning strategy for the game played on any component of $G^*$, which implies Breaker's win in the game played on $G^*$ itself by Observation~\ref{obs:playSeparately}.
\end{proof}

\section{Trees and Forests}\label{sec:forests}
We begin with the proofs of Theorems~\ref{thm:MBforest} and~\ref{thm:CWforest}.
%
%
%
\begin{proof}[Proof of Theorem~\ref{thm:MBforest}]
Note that we only need the first part of the theorem in order to state the second part. That is, to ensure the existence of the trees $T_{\min}^{(i)}$. It is therefore suffices to prove the first part under the assumption that Maker plays first. We do not make this assumption in the proof of the second part.
\begin{enumerate}
\item We do not attempt to optimize the size of $T$. Let $\Delta = \Delta(H)$, let $h = v(H)$, let $d = 2b\Delta^h$, and let $T$ be the $d$-ary tree with $h$ levels rooted at $r\in V(T)$.
Maker's strategy when playing on the vertices of $T$ is as follows.
He begins by claiming $r$, then in his next $\Delta$ moves he claims
$\Delta$ arbitrary children of $r$, then $\Delta$ arbitrary children
for each of them and so on until he reaches the leaves level. When
he finishes he owns a $\Delta$-ary tree with $h$ levels rooted at
$r$, which clearly contains an $H$-copy. It only remains to observe
that by the time Maker finishes to claim his $\Delta$-ary tree,
Breaker claims $b\sum_{i=0}^{h-1}\Delta^i \le b\Delta^h$ vertices,
and so Maker has enough children to claim for each vertex in his
tree at any point of the game, and he can follow this strategy.

\item Let $i$ such that $T_{\max} = T_{\min}^{(i)}$ and denote $v = v(T_{\max})$, $e = e(T_{\max})$. If $p =
o\left(n^{-v/e}\right)$ then by Claim~\ref{cl:smallTrees} every
connected component of $\ggnp$ is \whp a tree with less than $v$
vertices. By assumption, Breaker can prevent Maker from claiming a
$T_i$-copy on any such component and thus Maker cannot claim an
$H$-copy and loses the game.

Let $\F = \big\{T_{\min}^{(i)}\big\}_{i=1}^{k}$ and let $F$ be the graph consisting of $2bkv$ copies of each $T \in \F$, where all copies of all trees are vertex disjoint. Since $\F$ is finite
and every $T \in \F$ satisfies $m(T) \le e/v$, it follows by Theorem~\ref{thm:HinGnp} that if $p = \omega\left(n^{-v/e}\right)$ then \whp $\ggnp$ contains a copy of the (fixed) graph $F$. Maker can claim all trees $T_i$, one at the time, by playing $k$ separate
games. For each $i \in [k]$ he plays the $(1:b)$ $T_i$-game on the
vertex set of some $T_{\min}^{(i)}$-copy whose all its vertices are
still free at the moment Maker starts. Since Maker can win each such
game, and in no more then $v$ moves, by the time he finishes
claiming a copy of $H$, Breaker claims at most $bkv$ vertices, and so
for every $i$ Maker can always find a $T_{\min}^{(i)}$-copy with all
of its vertices still free and he can follow this strategy. \qedhere
\end{enumerate}
\end{proof}

%

\begin{proof}[Proof of Theorem~\ref{thm:CWforest}]
As in Theorem~\ref{thm:MBforest}, we only need the first part of the theorem in order to state the second part, and do not attempt to optimize the size of $T$ in the proof of the first part.
\begin{enumerate}
  \item Dean and Krivelevich showed in the proof of Proposition~9
    in~\cite{DK} that for any two integers $b$ and $k$, Client has a
    strategy to build a copy of the $k$-ary tree of height $k$ when
    playing the Client-Waiter $(1:b)$ game on the edge set of the
    $m$-ary tree of height $k$, where $m = (k(b+1))^2$. Client can use
    the same strategy in the vertex version by replacing each
    parent-child edge with the vertex of the child (and ignore the root vertex) and build a $k$-ary
    tree of height $k-1$. Clearly building
    such a tree for $k = v(H)$ is a winning strategy for Client in the
    $H$-game.
  \item As in the proof of Theorem~\ref{thm:MBforest}, let $i$ such that
    $T_{\max} = T_{\min}^{(i)}$, denote $v = v(T_{\max})$, $e = e(T_{\max})$, and note that if $p = o\left(n^{-v/e}\right)$ then by Claim~\ref{cl:smallTrees} every connected
    component of $\ggnp$ is \whp a tree with less than $v$ vertices. By
    offering in every move vertices from the same component, Waiter can
    prevent Client from claiming a $T_i$-copy, and thus Client cannot
    claim an $H$-copy and loses the game.

    For every $i \in [k]$ let $n_i = \left((bv)^2 + 1\right)^{i-1}$, let
    $F_i$ be the forest consisting of $n_i$ vertex disjoint copies of
    $\tmi$, denoted by $\big\{\tmij\big\}_{j = 1}^{n_i}$, let $F =
    \bigcup_{i \in [k]} F_i$ and finally let $F_{<i} = \bigcup_{\ell < i} F_\ell$. We
    now show that every graph containing an $F$-copy is Client's win,
    which completes the proof by Theorem~\ref{thm:HinGnp} and the fact that $F$ is a fixed
    graph satisfying $m(F) = e/v$.

    Given a graph $G$ containing $F$, Client plays as follows. He
    ignores all vertices of $V(G) \stm V(F)$ (if he is offered only
    vertices of this sort he chooses one arbitrarily). Whenever he is
    offered a set of vertices $U$ intersecting $F$, he considers only
    the vertices belonging to $\tmij$ for the minimal pair $(i,j)$ which
    appears in $U$, where he uses the natural lexicographical ordering
    of pairs: $(i,j) < (i',j')$ if either $i < i'$ or $i = i'$ and $j <
    j'$. He chooses a vertex from that tree according to his winning
    strategy in the $T_i$-game when playing on $V\left(\tmi\right)$, and
    deletes from $F$ all other trees which intersect $U$.

    Note that this strategy ensures that for every $i$ and $j$, whenever
    Waiter offers a vertex of $\tmij$, either Client claims a vertex in
    that tree according to his winning strategy on it, or he deletes it
    from $F$. It follows that at the end of the game, if a tree $\tmij$
    is still in $F$, then Client has claimed a $T_i$-copy in it. It
    remains to show that for every $i$ at least one $\tmij$ survives in
    $F$ until the end of the game.

    Now let $i \in [k]$, and note that a tree $\tmij$ is deleted from $F$ only if one of its
    vertices is offered to Client in a set containing a vertex from
    $T_{\min}^{(i',j')}$ for some pair $(i',j') < (i,j)$. Since all
    trees in $F$ have at most $v$ vertices, every tree in $F_{<i}$ causes the deletion of at most $bv$ trees in $F_i$. Furthermore, each deleted tree in $F_i$ can cause the deletion of at most $b(v-1)$ additional
    trees from $F_i$ (one of its vertices is the one causing
    its own deletion). Therefore, the number of trees deleted from $F_i$
    throughout the game is at most
    $$bv(1 + b(v-1))\sum_{\ell = 1}^{i - 1}n_\ell \le
    (bv)^2\sum_{\ell = 1}^{i - 1}\left((bv)^2 + 1\right)^{\ell-1} =
    (bv)^2\frac{\left((bv)^2 + 1\right)^{i-1} - 1}{\left((bv)^2 +
        1\right) - 1} = n_i - 1,$$ meaning that at least one tree survives
    in each $F_i$ until the end of the game, and thus Client wins. \qedhere
\end{enumerate}
\end{proof}

\begin{remark}\label{rmk:milos}
In~\cite[Lemma 36]{Sthesis}, Stojakovi\'{c} analyzed the unbiased Maker-Breaker $H$-game played on the edge set of $\ggnp$, when $H$ is a tree. So did Dean and Krivelevich for the biased Client-Waiter game in~\cite[Proposition 9]{DK}. It is immediate to see that their results can be generalized to biased games on forests by using essentially the proofs of Theorems~\ref{thm:MBforest} and~\ref{thm:CWforest}, with slight modifications.
\end{remark}

\begin{remark}
It is immediate to see from the proofs of Theorems~\ref{thm:MBforest} and~\ref{thm:CWforest} that if $H$ is a tree, then the second part in each theorem may be refined to a hitting time result (assuming Maker plays first in Theorem~\ref{thm:MBforest}). That is, if $\F$ is the family of all minimal sized trees on which the builder wins the game (these families may be different for Maker and Client), then in the random graph process \whp the graph becomes the builder's win at the same moment the first member of $\F$ appears. The same is true for the edge versions of these games.
\end{remark}

We finish this section by showing some more equivalencies between the edge and vertex versions of $H$-games where $H$ is either a path or a star. We follow the terminology of Theorems~\ref{thm:MBforest} and~\ref{thm:CWforest} and assume for simplicity Maker plays first. For the unbiased $P_\ell$-game, by using an almost identical proof to that of~\cite[Proposition~37]{Sthesis}, one can show that $v(T_{\min}) = \Theta\left(2^{\ell/2}\right)$ when considering the Maker-Breaker game. The argument in the proof of~\cite[Claim~33]{DK}
shows that $v(T_{\min}) = \Omega\left(2^{\ell/2}\right)$ when
considering the Client-Waiter game. 

We continue with the case $H=S_d$, the star with $d$ edges. Here,~\cite[Proposition~38]{Sthesis} (for the Maker-Breaker
game) and~\cite[Claim~32]{DK} (for the Client-Waiter game) show
that $S_{2d-1}$ is a minimal sized tree for which the builder of the game
has a winning strategy in the unbiased edge
$S_d$-game. It is easy to see that if we consider the $(1:b)$ versions of these games, $S_{(d-1)(b+1)+1}$ is a minimal sized tree required for the builder to win.

Arguments of the same sort show that for any two
positive integers $b$ and $d$, the star $S_{d(b+1)}$ is a minimal sized tree required for the builder in the $(1:b)$ vertex $S_{d}$-game, for both Maker and Client. More explicitly, the builder has at most $d$ vertices at the end of the game when playing on any tree with at most $d(b+1)$ vertices and trivially loses; when playing on $S_{d(b+1)}$, Maker claims the star center in his first move and continues arbitrarily, and Client claims the star center whenever it is
offered to him, and plays arbitrarily in any other
move. In particular, by Theorems~\ref{thm:MBforest} and~\ref{thm:CWforest}, both the Maker-Breaker and
Client-Waiter $(1:b)$ vertex $S_d$-games have a threshold at $p = n^{-\frac {d(b+1)+1}{d(b+1)}}$.

Note that the difference in the size of the star required for the builder between the edge and vertex versions comes from the fact that in the edge version the builder has to claim $d$ edges, while in the vertex version he has to claim $d+1$ vertices. In general, if the builder has to claim $t$ elements (either edges or vertices) in order to build $S_d$, then a minimal sized tree on which he wins is the star with $(t-1)(b+1)+1$ elements of the same type.
\section{Graphs containing a triangle}\label{sec:infamily}
\begin{proof}[Proof of Theorem~\ref{thm:infamily}]
Throughout this section $G$ stands for the random graph $\gnp$, where the values of $p$ vary according to the different parts of the theorem.

We begin the proof with the 0-statements. First, if $\alpha \leq 10/7$ and $p = o(n^{-7/10})$, then by Corollary~\ref{cor:triangle11} Breaker \whp can prevent Maker from claiming a triangle, and thus wins the $H$-game in this case. Next, if $\alpha > 10/7$ and $p \le cn^{-7/10}$, then the assertion of Conjecture~\ref{conj:main} implies that Breaker can prevent Maker from claiming an $H'$-copy, and once again win the $H$-game.

Moving to the 1-statements, if $\alpha \ge 3/2$ then Maker's win follows immediately from Remark~\ref{rmk:MakerSideConj}. Assume then that $\alpha < 3/2$, set $C$ to be the constant guaranteed in Theorem~\ref{thm:RamseyVertex} for $H'$ and $r = 20$, and note that whether $\alpha > 10/7$ or not, we have the following: $p = \omega(n^{-7/10})$, thus $G$ \whp contains two vertex disjoint $DD$-copies by Theorem~\ref{thm:HinGnp}; and $p \ge Cn^{-1/\alpha}$, thus \whp every induced subgraph of $G$ with $n/20$ vertices contains an $H'$-copy. From now on we assume $G$ satisfies these two properties.

The remainder of the proof goes along the same lines as the proof of \cite[Theorem~4]{HgameRandomGraphs}. For completeness we repeat it here with the necessary modifications. We first need the following claim about the expansion of $G$.

\begin{claim}[\cite{HgameRandomGraphs}, proof of Theorem~4]\label{cl:P1}
Let $p=\omega(n^{-7/10})$ and $G\sim G(n,p)$. Then w.h.p, for every subset $U\sbst V(G)$ of size $|U|\leq 1/(2p)$, we have $|N_G(U)|\geq |U|np/4$.
\end{claim}

We now assume $G$ possesses this expansion property, and describe Maker's winning strategy, which we divide into five phases (after these phases are complete Maker plays arbitrarily until the end of the game). For $i \in [5]$ let $M_i$ denote the set of vertices claimed by Maker during the $i$th phase.

\begin{description}
\item [Phase 1] In his first three moves, Maker claims a triangle $v_1v_2v_3$.
\item [Phase 2] In his next $\left\lceil 50/(n^2p^3)\right\rceil$ moves, Maker claims arbitrary free vertices from $N_G(v_1)$.
\item [Phase 3] In his next $\left\lceil 5/(np^2)\right\rceil$ moves, Maker claims arbitrary free vertices from $N_G(M_2)$.
\item [Phase 4] In his next $\left\lceil 1/(2p)\right\rceil$ moves, Maker claims arbitrary free vertices from $N_G(M_3)$.
\item [Phase 5] In his next $n/20$ moves, Maker claims arbitrary free vertices from $N_G(M_4)$.
\end{description}

We first show that Maker can follow this strategy. He can follow Phase 1 even as a second player since $G$ contains two vertex disjoint $DD$-copies (recall the proof of Theorem~\ref{thm:triangle11}). Since $p = \omega(n^{-7/10})$ we have $|N_G(v_1)| \ge np/4 = \omega(1/(n^2p^3))$, implying Maker can follow Phase 2. For $i=2,3,4$ we have $|M_{i}| \le 1/(2p)$ and so $N_G(M_{i}) \ge |M_{i}|np/4$. In addition, denoting the number of vertices claimed by both players during the first $i$ phases by $n_i$, we have $n_i = o(|N_G(M_{i}|)$ since $np = \omega(1)$.
It follows that at least $|M_{i}|np/5$ of the vertices in $N_G(M_{i})$ are still free at the end of the $i$th phase, and Maker can claim $|M_{i}|np/10$ of them and complete the next phase.

It remains to observe that Maker wins the game by following the proposed strategy. Indeed, $|M_5| = n/20$ and so by assumption $G[M_5]$ contains an $H'$-copy. Since $v_1v_2v_3$ is a triangle, and since every vertex in $M_5$ is connected to $v_1$ via a path of length 4 disjoint from $\{v_2,v_3\}$ by construction, Maker's graph contains an $H$-copy at the end of Phase~5.
\end{proof}

\section{Other game types}\label{sec:others}

We begin this section with the description and analysis of a positional game in which all target sets are pairwise disjoint. This game is later used as an auxiliary game in many of the proofs in this section.

\subsection{Box games}
In their seminal paper~\cite{CE}, Chv\'atal and Erd\H{o}s introduced
the \emph{box game}. This is essentially a Maker-Breaker game where
all board elements are partitioned into element disjoint winning
sets. Each winning set is referred to as a \emph{box}, and the two
players are denoted by BoxMaker and BoxBreaker. Chv\'atal and
Erd\H{o}s used this game in their analysis of the connectivity game
as part of Breaker's strategy, who pretends to be BoxMaker in an
auxiliary game, and by this isolates a vertex in Maker's graph. This
is of course a winning strategy for Breaker in other games for which
positive minimum degree in his graph is a necessary condition for
Maker (called \emph{spanning} games), such as the Hamiltonicity game and the perfect matching game.

Chv\'atal and Erd\H{o}s were interested in the $(a:1)$ box game where the sizes of the smallest and largest
winning sets differ by at most one. Later, in~\cite{HLV}, Hamidoune and Las Vergnas analyzed the box game in
full generality (and also corrected a mistake in one of the proofs in~\cite{CE}). In this section we use
different variations of box games, in all of them we are only interested
in \emph{uniform} box games, that is, all boxes are of the same size. We denote this setting by $n \times k$, where $n$
indicates the number of boxes and $k$ denotes the size of each box. We first state and prove a trivial result for the Waiter-Client version of the game, that is, BoxWaiter is trying to force BoxClient to fully claim a box, while BoxClient is trying to avoid it.
We consider the $(a:b)$ uniform game, abbreviated to $WCBox(n \times k,(a:b))$.

\begin{claim}\label{cl:WCBox}
    Let $a, b, k$ be three positive integers. There exists an integer $N
    = N(a, b, k)$ such that for every $n \ge N$ BoxWaiter has a winning
    strategy in the game $WCBox(n \times k,(a:b))$.
\end{claim}

\begin{proof}
We do not wish to optimize the bound on $N$ and rather show that given $a$ and $b$, BoxWaiter has a winning
strategy in the game $WCBox(n_k \times k,(a:b))$, where $n_k := (a+b)^k$. Clearly, this strategy is also
applicable to a game containing $n > n_k$ boxes since BoxWaiter can offer elements only from the first
$n_k$ boxes in his first $n_{k-1}$ moves, achieve his goal in the game, and then offer all remaining elements arbitrarily. We prove the
claim by induction on $k$. For $k = 1$ BoxWaiter simply offers all $(a+b)$ elements in the game, one from each
box. BoxClient then claims at least one of them and loses. Now Assume that BoxWaiter has a winning strategy in
the game $WCBox(n_k \times k,(a:b))$ and consider the game $WCBox(n_{k+1} \times (k+1),(a:b))$. In the first
$n_k$ rounds of the game, BoxWaiter offers arbitrary $a+b$ elements in every move, each from a different box
for which none of its elements has been offered yet. After these rounds there are $an_k \ge n_k$ boxes in
which Client has claimed an element, and no other element in them has been offered yet. BoxWaiter can then
apply his strategy for the game $WCBox(n_k \times k,(a:b))$ on these boxes and win.
\end{proof}

In~\cite{FKN}, Ferber, Krivelevich and Naor analyzed the Avoider-Enforcer version of the box game (where
BoxEnforcer is trying to force BoxAvoider to fully claim a box).
Following the ideas of~\cite{CE}, they also used the monotone version of the game as an auxiliary game in
order to provide Avoider with a strategy to isolate a vertex in his graph and win various spanning games. This
was indeed Avoider's strategy in the $k$-connectivity, Hamiltonicity, and perfect matching games later considered
in~\cite{FGKN}. We abbreviate the $(a:b)$ Avoider-Enforcer uniform box game to ${AEBox(n \times k,(a:b))}$. When
considering the monotone game, the situation is very simple. Similarly to BoxWaiter, BoxEnforcer wins if there
are sufficiently many boxes, where the number of required boxes depends on the bias of the players and the box
size, all of which are fixed.

\begin{theorem}[Theorem~1.7 in \cite{FKN}]\label{thm:AEBoxMonotone}
    Let $a, b, k$ be three positive integers. There exists an integer $N
    = N(a, b, k)$ such that for every $n \ge N$ BoxEnforcer has a
    winning strategy in the monotone game $AEBox(n \times k,(a:b))$ as a
    first or a second player.
\end{theorem}

However, when considering the strict game, things are more
complicated and the outcome of the game depends on some divisibility
conditions. The following theorem appears in a slightly weaker form
as Corollary~1.4 in~\cite{FKN}, and can be deduced from Theorems~1.2
and~1.3 of that paper.

\begin{theorem}[\cite{FKN}]\label{thm:AEBoxStrict}
    Let $a, b, k$ be three positive integers. If $\gcd(a + b, \ell) \le
    a$ holds for every $2 \le \ell \le k$, then there exists an integer
    $N = N(a, b, k)$ such that for every $n \ge N$ BoxEnforcer has a
    winning strategy in the strict game $AEBox(n \times k,(a:b))$ as a
    first or a second player. Otherwise, BoxAvoider wins this game for
    every $n$.
\end{theorem}


We conclude this subsection with the following related claim which will be very useful in the proof of
Theorem~\ref{thm:AEstrict}.

\begin{claim}\label{cl:AEAux}
    Let $n$ and $m$ be two non-negative integers and let $\cH = (X,\F)$
    be a hypergraph where $X = \{a_1,\dots,a_n,b_1,\dots,b_n,
    c_1,\dots,c_m\}$ and $\F =
    \big\{\{a_1,b_1\},\dots,\{a_n,b_n\}\big\}$.
    \begin{enumerate}[$(a)$]
        \item Avoider wins the strict $(1:1)$ Avoider-Enforcer game played on
        $\cH$ as a first or a second player.
        \item If, in addition, Avoider makes the last move in the game (that is,
        either $m$ is odd and Avoider plays first, or $m$ is even and
        Avoider plays second), he also wins the strict $(1:1)$ game played
        on $\cH' = (X', \F')$ where $X' = X \cup \{d\}$ and $\F' = \F \cup
        \big\{\{d\}\big\}$ for a new element $d \not\in X$ (note that when
        playing on $\cH'$ Enforcer makes the last move in the game).
    \end{enumerate}
\end{claim}

\begin{proof}
    Avoider uses the same strategy for both variants of the game. As
    long as there exists a free vertex $c_i$, or a free vertex of one of
    the pairs $\{a_i,b_i\}$ such that the other vertex in that pair is
    already claimed by Enforcer, then Avoider claims one of these
    vertices arbitrarily. If before one of his moves no such vertex
    exists, then at this point only pairs $\{a_i,b_i\}$ in which both
    vertices are free remain. Indeed, it is trivial when playing on
    $\cH$, and for $\cH'$ it follows from the fact that there is an even
    number of free vertices before each of Avoider's moves by
    assumption. From this moment until the end of the game Avoider plays
    as BoxAvoider in the strict game $AEBox(n' \times 2,(1:1))$, where
    $n' \le n$ is the number of unclaimed pairs. By
    Theorem~\ref{thm:AEBoxStrict} he can Avoid claiming both elements
    of a pair (as $\gcd(1+1,2) = 2 > 1)$, and by doing so win the game.
\end{proof}

%
%
%

\subsection{Proofs of Theorems~\ref{thm:trivialWins},~\ref{thm:AEstrict},~\ref{thm:WC},~\ref{thm:CWsimple} and~\ref{thm:CWtriangle}}
We begin this subsection with the proofs of Theorems~\ref{thm:trivialWins} and~\ref{thm:WC}. Since the
statements of the theorems are almost identical, it is not surprising that their proofs are also very similar.

\begin{proof}[Proof of Theorems~\ref{thm:trivialWins} and~\ref{thm:WC}]
    Let $G' \sbst G$ be a collection of vertex disjoint $H$-copies of
    maximal size. Enforcer and Waiter win their respective $H$-games if
    $G'$ is sufficiently large, in which case they can treat each
    $H$-copy in it as a box (where each vertex in the copy is one
    element in the box) and simulate some box game.

    By the proof of Claim~\ref{cl:WCBox}, Waiter has a winning strategy if $G'$
    contains $(a+b)^{v(H)}$ copies, because he simply plays first the
    $WCBox$ game on $G'$ and after achieving his goal offers all other
    vertices in $G$ arbitrarily. Enforcer's strategy is straight
    forward as well. We assume for simplicity that he plays first, and
    the proof is very similar when he is the second player. In his first
    move he claims all of $V(G) \stm V(G')$ and from that
    moment he plays as BoxEnforcer in the $(a:b)$ mis\`{e}re box game.
    By Theorem~\ref{thm:AEBoxMonotone}, BoxEnforcer wins as long as
    there are enough boxes, where enough means some constant depending
    only on $a$, $b$ and $v(H)$. Winning in this game means of course
    that Avoider is forced to claim all vertices of some $H$-copy.

    Since Client and Avoider win trivially in their respective games if
    $G$ is $H$-free, the result for a general graph $H$ in both theorems
    follows by Theorem~\ref{thm:HinGnp}.

    For a strictly balanced graph $H$ the result follows immediately by
    Corollary~\ref{cor:balancedDisjoint}, where the integer $N$ from the
    statements of the theorems represent the minimal number of boxes
    required for the wins of BoxWaiter and BoxEnforcer in
    the games described in this proof.
\end{proof}

It is immediate to see that by replacing the vertices of the
$H$-copies with their edges in the proof above, we can obtain the
same results for Avoider-Enforcer and Waiter-Client $(a:b)$
$H$-games played on the edge set of random graphs. The
Avoider-Enforcer result can also be obtained from Corollary~1.8
in~\cite{FKN}.

\begin{corollary}
    Theorems~\ref{thm:trivialWins} and~\ref{thm:WC} are still valid when
    stated for the edge versions of the games, with the only minor
    difference that the parameter $N$ in their statements depends (in
    addition to the bias of the players) on $e(H)$ rather than on
    $v(H)$.
\end{corollary}

We continue with the strict Avoider-Enforcer game played on the vertex set of $\ggnp$. Before proving
Theorem~\ref{thm:AEstrict}, we would like to shortly discuss the $H$-game for a given arbitrary graph $H$.
Theorem~\ref{thm:AEBoxStrict} imply that there are infinitely many pairs of integers $(a,b)$, for which a
strict game equivalent of Theorem~\ref{thm:trivialWins} holds. These are all pairs satisfying
$\gcd(a + b, \ell) \le a$ for every $2 \le \ell \le v(H)$. For example, two obvious families of such pairs are
all pairs $(a,b)$ such that $a \ge v(H)$, and all pairs $(a,b)$ such that $a+b$ is a prime number greater than
$v(H)$. There is one small difference in the statement of the hitting time result, though. The number $N$ of
required $H$-copies in the graph does not depend only on $a$, $b$ and $v(H)$, but also on the identity of the
first player and the residue of $n \mod (a+b)$, where $n$ is the number of vertices in the graph process. It
does not depend on the value of $n$ itself, however, and therefore it is a constant. One simple example is
that if $n \md {0}{(a+b)}$, Enforcer plays first and $a \ge v(H)$, then $N = 1$ (Enforcer plays arbitrarily
and only avoiding the vertices of the single $H$-copy in $G$, which will be fully claimed by Avoider
eventually).

At the same time, there exist infinitely many integers $b$ such that \whp Avoider wins the $(1:b)$ $H$-game for every $p = O\left(n^{-1/\mu(H)}\right)$, where $\mu(H) = m(H) + \frac 1{2v(H)^2}$. Indeed, let $\hat{H} \sbst H$ be a strictly balanced graph with $m(\hat{H}) = m(H)$. By Claim~\ref{cl:balancedDense}, Theorem~\ref{thm:HinGnp},
the assumption on $p$, and the fact that two $\hat{H}$-copies can intersect in finitely many ways, all
$\hat{H}$-copies in $G$ are \whp vertex disjoint. Thus, if $b+1$ is divisible by an integer $2 \le \ell \le
v(\hat{H})$, then Avoider can win the $H$-game. His strategy would be to first play arbitrarily as long as he
does not claim the first unclaimed vertex of any $\hat{H}$-copy, and if it at some point he can no longer do
so, play until the end of the game as BoxAvoider where each box is an $\hat{H}$-copy. He wins this game by
Theorem~\ref{thm:AEBoxStrict}. This strategy ensures that he avoids fully claiming any $\hat{H}$-copy and thus
wins the $H$-game. For example, Avoider can apply this strategy in any $(1:b)$ $H$-game whenever $b$ is odd
(and thus $b+1$ is divisible by $2 \le v(\hat{H})$). In particular, the unbiased $H$-game is never an ``easy"
win for Enforcer, in the sense that the threshold probability $p^*_{1,H}$ in this case satisfies $p^*_{1,H} =
\omega\left(n^{-1/\mu(H)}\right)$ for every $H$, including, of course, the case $H = K_3$, to which we now
refer.

\begin{proof}[{Proof of Theorem~\ref{thm:AEstrict}}]
Not surprisingly, the proof goes along the lines of the proof of Theorem~\ref{thm:triangle11}. We omit or abbreviate some of the repeated arguments. We begin with sufficient conditions for Enforcer's win when playing the game on $V(G)$ for an arbitrary graph $G$. Assume first that Avoider makes the last move of the game and that $G$ contains a $DD$-copy $\D$. By Part~$(b)$ of Claim~\ref{cl:AEAux}, Enforcer can play an auxiliary game (as \emph{AuxAvoider}) and avoid claiming the center of $\mathcal D$ and any full pair of the natural pairs of $\D$. It means that Avoider, as \emph{AuxEnforcer}, must claim the center and at least one vertex form each pair, which leads to Enforcer's win in the original game by Observation~\ref{obs:DDpairing}. Similarly, Part~$(a)$ of Claim~\ref{cl:AEAux} implies that if $G$ contains two vertex disjoint $DD$-copies, then Enforcer (even if he makes the last move), can win by ensuring that Avoider claims at least one vertex from each of their natural pairs, and at least one of the centers (by pairing up the two centers as well).

Now let $\tilde G = \{G_i \}$ be a random graph process, and recall that $\mathcal G_{DD}$, $\mathcal
G_{2DD}$, $\E_{K_3}^A$ and $\E_{K_3}^E$ are the graph properties of containing one or two vertex disjoint
$DD$-copies, respectively, and being Enforcer's win in the strict $(1:1)$ triangle game, where Avoider or
Enforcer, respectively, moves last. Then $\tau(\tilde G,\E_{K_3}^A) \le \tau(\tilde G,\mathcal G_{DD})$
trivially holds, and, since the first two $DD$-copies appearing during the random graph process are \whp
vertex disjoint, $\tau(\tilde G,\E_{K_3}^E) \le \tau(\tilde G,\mathcal G_{2DD})$ holds \whp as well. Let $i_1
= \tau(\tilde G, \mathcal G_{DD}) - 1$ and $i_2 = \tau(\tilde G,\mathcal G_{2DD}) - 1$, and let $G^*_1$ and
$G^*_2$ be the $(K_3, 1)$-cores of $G_{i_1}$ and $G_{i_2}$, respectively. It remains to show that \whp Avoider
has a winning strategy in the game played on $V(G_{i_1})$, and, provided he is not the last player to play, in
the game played on $V(G_{i_2})$. Following the argument presented in the proof of
Theorem~\ref{thm:triangle11}, from now on we assume that $G_{i_2}$ is $\F$-free (and therefore $G^*_1$ and
$G^*_2$ are $\F$-free as well), where $\F$ is the family defined in Claim~\ref{cl:ConComp}. Since this is
true w.h.p., and since all remaining arguments are deterministic, the proof holds.

When playing the game on $V(G_{i_1})$, Avoider first runs the deletion algorithm to obtain an output
$U,W,G^*_1$. By our assumption and by Claim~\ref{cl:ConComp}, every connected component of $G^*_1$ is a
$\ttt$-copy. Avoider pairs up the two vertices of any bad pair $U_i \in U$, and for any $W_j \in W$ he chooses
arbitrary $\lfloor W_j / 2 \rfloor$ disjoint pairs contained in $W_j$. By Part~$(a)$ of Claim~\ref{cl:AEAux}
he can avoid claiming all these pairs and all the natural pairs of all components of $G^*_1$ (given in Definition~\ref{def:trioPairing}. It follows that
by the end of the game he claims at most one vertex from each bad pair $U_i$ and at most two vertices from
each small component $W_j$, and thus by Claim~\ref{cl:G*} and the fact that he avoids all triangles in
$G^*_1$, he wins the game.

When playing on $V(G_{i_2})$, every connected component of $G^*_2$ is a ${\ttt}$-copy, except for one
$DD$-copy $\D$. If Enforcer makes the last move of the game, then by Part~$(b)$ of Claim~\ref{cl:AEAux}
Avoider can avoid claiming all the pairs as in the previous case (that is, all the pairs generated by the
deletion algorithm and all natural pairs of the $\ttt$-copies in $G^*_2$), as well as the center of $\D$, and
win.
\end{proof}

It remains to prove the theorems regarding Client-Waiter games. The
1-statement of Theorem~\ref{thm:CWsimple} follows from
Theorem~\ref{thm:RamseyVertex}. The 0-statement of the theorem is an
immediate corollary of Lemma~\ref{lem:BwinG'}, the proof of the 0-statement of Theorem~\ref{thm:main} (given in Sections~\ref{sec:largeCliques},~\ref{sec:(1:2)triangle} and~\ref{sec:cycle}), and the following trivial
observation.

\begin{observation}
For any two graphs $G,H$ and integer $b$, if Breaker can win the $(1:b)$ Maker-Breaker $H$-game played on
$V(G)$ by using a pairing strategy (with respect to $\Lambda$), then Waiter can win the $(1:b)$ Client-Waiter
$H$-game played on $V(G)$ by using the same pairing strategy. That is, Waiter offers each set of at most $b+1$
vertices from $\Lambda$ in one move, and offers all other vertices (if there are any) arbitrarily.
\end{observation}

We finish this section with the proof of Theorem~\ref{thm:CWtriangle}.

\begin{proof}[{Proof of Theorem~\ref{thm:CWtriangle}}]
We set $p = an^{-2/3}$ where $a < 1$ is undetermined at this point, and consider the $(1:1)$ Client-Waiter triangle-game played on $\ggnp$. We begin with Waiter's side, and so by Lemma~\ref{lem:BwinG'} we may focus on $G^*$, the $(K_3,1)$-core of $G$. By Observation~\ref{obs:playSeparately} we may consider each connected component of $G^*$ separately, and by Claim~\ref{cl:noNegativeExponent} we may assume that all families $\xqts$ for which $G^*$ contains at least one of their members satisfy $2q + s \le 3$. Recall Definition~\ref{def:zqts} of the families $\zqts$ of $(K_3,1)$-stable graphs obtained via greedy explorations. Lemma~\ref{lem:noDang} and Corollary~\ref{cor:exploreK3} imply that any connected component of $G^*$ not contained in any $\zqts$ is a $K_3$-cycle of length at least four, for which there exists a natural pairing strategy for Waiter (see Definition~\ref{def:K3cyclePairing}).

It follows that if Client has a winning strategy in the game, then he must have a winning strategy when playing on the vertex set of some $\GG \in \zqts$, where $q,t,s$ are three integers satisfying $2q + s \le 3$. In fact, the last condition may be replaced with $2q + s = 3$. Indeed, if $2q + s \le 2$ and $\GG \in \zqts$, then $\GG$ is either a ${\ttt}$ or a $DD_t$ by Claim~\ref{cl:2qsle2}. In either case, Waiter can apply the natural pairing strategy (see Definition~\ref{def:trioPairing} and Claim~\ref{cl:DDtPairing}). It remains to consider the members of the families $\zqts$ for which $2q + s = 3$. As $s \ge 1$, there are only two cases to consider: $q = s = 1$ and $q = 0,~s = 3$. Claim~\ref{cl:11is30} shows that it suffices to consider only the latter case, that is, we only have to consider $\z03$. The next natural step is to bound the size of this family.

\begin{claim}\label{cl:y0t3}
    $|\z 03| \le 12 \binom {t+3}{2}$ holds for every integer $t$.
\end{claim}

\begin{proof}
    Consider a graph $\GG \in \z 03$. By definition we start the
    exploration of $\GG$ greedily, with two triangles sharing an edge, $xyz_1$ and
    $xyz_2$. By
    Corollary~\ref{cor:noExtrnlCps} the remainder of the exploration
    goes as follows: $t_1$ additions of external triangles for some $0
    \le t_1 \le t-1$, an addition of an internal triangle $T_1$
    containing an existing edge $e_1$, the addition of $t_2 = t-1-t_1$
    external triangles, and finally the addition of an internal triangle
    $T_2$ containing an existing edge $e_2$. Furthermore, for $i = 1,2$,
    if $t_i > 0$ then $e_i$ must be part of the last external triangle
    that was added before it by the restriction of the exploration
    process.

    Given $t_1,t_2$, let us bound from above the number of graphs that can be
    created in this manner (some of them may be not stable), where all
    the arguments are done with respect to isomorphism. If $t_1 > 0$ then the
    first triangle in the chain contains either $x$ or $z_1$; there is
    only one option for the external chain; there are two options for
    $e_1$ whether $t_1 = 0$ or not; if $t_2 > 0$ there are $5 +2t_1$
    options for the existing vertex contained in the external chain, and
    then one option for the chain itself and by stability one for $e_2$;
    if $t_2 = 0$ there are $7 + 3t_1 $ options for $e_2$. All in all we
    can bound from above the number of graphs with $4(7 + 3t_1) \le
    12(t_1 + 3)$.  Hence we get
    \begin{equation*}
    |\z 03| \le \sum_{t_1 = 0}^{t-1} 12(t_1 + 3) \le 12\sum_{t_1 =
        0}^{t+2} t_1 = 12 \binom {t+3}{2}. \qedhere
    \end{equation*}
\end{proof}

We are now ready to prove Waiter's side. The arguments above show that Waiter wins the game if $G^*$ contains no members of $\z 03$. For every $\GG \in \z 03$ we get by Equation~(\ref{eq:xqts}) in the proof of Claim~\ref{cl:qsBound} that $\Pr[\Gamma \sbst G] \le a^{e(\Gamma)} = a^{3t +6}$. Using the fact that $\sum_{n=0}^\infty x^n = \frac 1{1-x}$ for any $|x| < 1$ we obtain that
$$\sum_{n=0}^\infty \binom n2 x^n = \frac 12{x^2}\sum_{n=0}^\infty n(n-1)
x^{n-2} = \frac 12{x^2}\frac {d^2}{dx^2}\left(\sum_{n=0}^\infty
x^{n}\right) = \frac 12{x^2}\frac {d^2}{dx^2}\left(\frac 1{1-x}
\right) = \frac {x^2}{(1-x)^3}$$
for every such $x$.

For any $a < 1$ the probability that $G^*$ contains a component $\GG \in \z 03$ (for any $t$) can therefore be bounded from above by
$$\sum_{t=1}^{\infty}\sum_{\GG \in \z 03}\Pr[\Gamma \sbst G] \le
\sum_{t=1}^{\infty} 12 \binom {t+3}{2} a^{3t +6} = 12
\sum_{t=4}^{\infty} \binom {t}{2} a^{3t-3}  \le 12
\sum_{t=0}^{\infty} \binom {t}{2} a^t = \frac {12a^2}{(1-a)^3}.$$

For $a = o(1)$  it follows that Waiter wins the game with
probability at least $1 - \frac {12a^2}{(1-a)^3} = 1 - o(1)$, and
thus Claim~$(1)$ of the theorem holds.

For every constant $a$ such that $\frac {12a^2}{(1-a)^3} < 1$ the probability Waiter wins is bounded away
from 0, thus Claim~$(2)$ of the theorem holds. In the terms of the theorem, one can set, for example, $c = 0.2$ (this is not the maximal choice for $c$) and $\alpha = 1/16$.

%
Since Client claims half the vertices of $G$ by the end of the game
no matter how he plays, Claim~$(4)$ of the theorem trivially holds
by Theorem~\ref{thm:RamseyVertex}.

It remains to prove Claim~$(3)$ of the theorem. For
this we describe a strictly balanced fixed graph with density $3/2$, to which we refer as a triple diamond, and show that if $G$ contains a copy of this
graph Client wins the game. The desired result then follows by
Theorem~\ref{thm:fixedGraphDist}. In the terms of the theorem, for any $d>0$ one can set $\beta = 1-e^{-d^{15}/16}$.

Before getting to the triple diamond, we first make some
observations and introduce new terminology. For any triangle $xyz$
in $G$, if at any point during the game Client claims the vertex $x$
while the vertices $y$ and $z$ have not been offered yet, we say
that the pair $\{y,z\}$ is \emph{forced} on Waiter, because if this
pair will not be offered later in the game, then at the moment the first vertex from the pair
is offered Client can abandon any other strategy, claim it, and then whenever the other vertex in the pair is offered claim it as well, and win.

Let $D$ be a diamond on the vertex set $\{x,y,z,w\}$, i.e., a copy of $K_4$ on this vertex set minus the edge $zw$. We say that $D$ is a \emph{winning diamond} if the
vertices $\{y,z,w\}$ have not been offered yet in the game, and in
addition either $x$ was already claimed by Client, or that a pair
$\{x,u\}$ is forced on Waiter for some vertex $u \not\in \{y,z,w\}$.
It is easy to see that the existence of such a diamond guarantees
Client's win in the game: assuming the pair $\{x,u\}$ will be
offered, Client can claim the vertices $x$, $y$ and either $z$ or
$w$, as $y$ cannot be offered in a pair with both of them. We now
define a triple diamond and show that if $G$ contains a copy of it then
Client has a winning strategy in the game.

A triple diamond is a diamond-chain of length 3, where each two consecutive diamonds intersect in a vertex of maximal degree in each diamond.
Let $H \sbst G$ be a triple diamond composed of the three diamonds
$D_L$, $D_M$ and $D_R$, with vertex sets $V_L =
\{x_1,y_1,z_1,z_2\}$, $V_M = \{x_1,x_2,w_1,w_2\}$, $V_R =
\{x_2,y_2,z_3,z_4\}$, respectively (see Figure~\ref{fig:DDD}).
Client's strategy goes as follows. He plays arbitrarily whenever the
pair offered to him is disjoint from $V(H)$. Let $\{v,u\}$ denote
the first pair offered to him which intersects $V(H)$. Client then
plays according to these four cases (all other options are
isomorphic to those described here).
\begin{enumerate}
    \item If $v = x_1$ Client claims $v$ and then either $D_L$ or $D_M$
    is a winning diamond (since $u$ belongs to at most one of them).
    \item Otherwise, if $v = y_1$ Client claims $v$. If $u \not\in V_L$
then $D_L$ is a winning diamond. If $u \in V_L$, then $u = z_1$, and
thus the pair $\{x_1,z_2\}$ is forced on Waiter, and so $D_M$ is a
winning diamond.
    \item Otherwise, if $v = w_1$ Client claims $v$, and the pair
    $\{x_1,x_2\}$ is then forced on Waiter. This means that either
    $D_L$ or $D_R$ is a winning diamond (since $u$ belongs to at most
    one of them).
    \item Otherwise, $v = z_1$ and $u \not\in V_M \cup \{y_1\}$. Client claims $v$,
    forces the pair $\{x_1,y_1\}$ on Waiter and $D_M$ becomes a winning
    diamond.\qedhere
\end{enumerate}

\iffigure
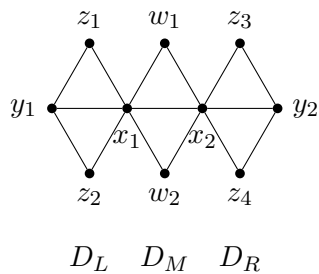
\begin{figure}
    \centering
    \begin{tikzpicture}[auto, vertex/.style={circle,draw=black!100,fill=black!100, thick,
        inner sep=0pt,minimum size=1mm}]
    \node (x1) at (-0.5,0) [vertex,label={[label distance=1mm]270:$x_1$}] {};
    \node (x2) at (0.5,0) [vertex,label={[label distance=1mm]270:$x_2$}] {};
    \node (y1) at (-1.5,0) [vertex,label=left:$y_1$] {};
    \node (y2) at (1.5,0) [vertex,label=right:$y_2$] {};
    \node (z1) at (-1,0.866) [vertex,label=above:$z_1$] {};
    \node (z2) at (-1,-0.866) [vertex,label=below:$z_2$] {};
    \node (z3) at (1,0.866) [vertex,label=above:$z_3$] {};
    \node (z4) at (1,-0.866) [vertex,label=below:$z_4$] {};
    \node (w1) at (0,0.866) [vertex,label=above:$w_1$] {};
    \node (w2) at (0,-0.866) [vertex,label=below:$w_2$] {};
    \node at (-1,-2) {$D_L$};
    \node at (0,-2) {$D_M$};
    \node at (1,-2) {$D_R$};

    \draw [-] (y1) --node[inner sep=0pt,swap]{} (x1);
    \draw [-] (y1) --node[inner sep=0pt,swap]{} (z1);
    \draw [-] (y1) --node[inner sep=0pt,swap]{} (z2);
    \draw [-] (x1) --node[inner sep=0pt,swap]{} (z1);
    \draw [-] (x1) --node[inner sep=0pt,swap]{} (z2);
    \draw [-] (x1) --node[inner sep=0pt,swap]{} (x2);
    \draw [-] (x1) --node[inner sep=0pt,swap]{} (w1);
    \draw [-] (x1) --node[inner sep=0pt,swap]{} (w2);
    \draw [-] (x2) --node[inner sep=0pt,swap]{} (w1);
    \draw [-] (x2) --node[inner sep=0pt,swap]{} (w2);
    \draw [-] (x2) --node[inner sep=0pt,swap]{} (y2);
    \draw [-] (x2) --node[inner sep=0pt,swap]{} (z3);
    \draw [-] (x2) --node[inner sep=0pt,swap]{} (z4);
    \draw [-] (y2) --node[inner sep=0pt,swap]{} (z3);
    \draw [-] (y2) --node[inner sep=0pt,swap]{} (z4);
    \end{tikzpicture}
    \caption{Triple Diamond}\label{fig:DDD}
\end{figure}

\fi

\end{proof}

\section{Concluding Remarks and Open Problems}\label{sec:concluding}
The main open problem raised in this paper is proving Breaker's side in Conjecture~\ref{conj:main}.  In the previous sections we provided two families of graphs for which Breaker's side of the conjecture holds. Our method was to first apply the general deletion algorithm that breaks down $G$ into small components with strong limitations on their structure, and then perform some case analysis on these components. We believe that the same can be done for any given graph $H$, although the case analysis might be very exhausting (as the cycle case shows). We also believe that the deletion algorithm, as described here, could be helpful in proving the conjecture in its general form. However, the specific case analysis should be replaced by some general argument about the possible structures of the surviving connected components, if there are any.


Another question concerns the sharpness of the threshold probabilities of $H$-games (see Chapter 1.2 in \cite{FrBook} for more information). We know that whenever Maker ``wins locally", i.e. wins as soon as some fixed graph appears in $G$, the threshold is coarse. This is the case for the unbiased triangle game and for forest games, which are not covered by Conjecture~\ref{conj:main}. From the discussion in Section~\ref{sec:constants}, we have that for any integer $b$ and
every $\eps > 0$ there exists a constant $k_0 := k_0(b,\varepsilon)$ such that for any $k
\ge k_0$, when playing the $(1:b)$ $K_k$-game on the vertex set of
$\ggnp$, Maker wins \whp for $p \ge (1 + \eps)n^{-2/k}$, and Breaker
wins \whp for $p \le (1 - \eps)n^{-2/k}$. Of course, ideally we could switch the order of the quantifiers in the above statement to establish a sharp threshold. More formally, we would like to know the following.

\begin{question}
Is there a constant $k_0$ such that for
every $k \ge k_0$ the $(1:b)$ Maker-Breaker $K_k$-game has a sharp
threshold at $p = n^{-2/k}$? If so, does $k_0$ equal $4$? Are there any other $H$-games with
a sharp threshold at $p = n^{-1/m_1(H)}$?
\end{question}

We conclude this section with a short discussion about the new type of positional games introduced in this paper. That is, given a graph $G$, the two players claim its vertices, where the outcome of the game is determined by  the subgraph of $G$ induced by the vertices claimed by one of the players (in some games the induced subgraphs of both players matter). We considered $H$-games, and it would be interesting to investigate other classical games such as the $k$-connectivity, the perfect-matching and the Hamiltonicity games.  Note that if the graph properties we consider are spanning, like in these examples, then  the whole nature of the game changes dramatically in comparison to the edge version of these games.

To begin with, the very meaning of the word ``spanning" is different. For instance, in the Maker-Breaker Hamiltonicity game, it is obviously impossible for Maker to claim a Hamilton cycle of $G$. Instead, Maker wins if at the end of the game the graph induced by the set of vertices he has claimed throughout the game is Hamiltonian. For the same reason, these games are not bias monotone -- claiming more vertices can harm both Maker and Breaker. Another difference of this sort, is the fact that for any given monotone increasing graph property $\mathcal P$, when playing on the edge set of a graph $G$,  the family of target sets $\mathcal T$ (either winning or losing) is closed upwards, that is, if $E(G_1)\in \mathcal T$ and $G_1\sbst G_2 \sbst G$, then $E(G_2)\in \mathcal T$. Clearly, this is not the case when playing on the vertex set of $G$ and $\mathcal P$ is a spanning graph property: returning to the Hamiltonicity example, if $U \sbst W \sbst V(G)$, there is no relation whatsoever between the Hamiltonicity of $G[U]$ and that of $G[W]$. Considering all this, it seems that these games might be hard to analyze (and perhaps it is not obvious how to even define them), and it is not clear whether their behavior is analogous in some sense to the behavior of their respective edge version games, like we have shown in this paper for $H$-games.

\subsection*{Acknowledgements}
The authors wish to thank Michael Krivelevich for suggesting the new setting of positional games introduced in this paper and for
helpful discussions.


\begin{thebibliography}{99}

    \bibitem{BeckBook}
    J.\ Beck, \textbf{Combinatorial Games: Tic-Tac-Toe Theory},
    Cambridge University Press, 2008.


    \bibitem{Beck}
    J. Beck, ``Positional games and the second moment method'',
    Combinatorica 22 (2002), 169--216.


    \bibitem{BedAE}
    M. Bednarska-Bzd\c{e}ga, ``Avoider-Forcer games on hypergraphs with
    small rank'', Electronic Journal of Combinatorics, 21
    (2014), P1.2.


    \bibitem{BedCW}
    M. Bednarska-Bzd\c{e}ga, ``On weight function methods in
    Chooser-Picker games'', Theoretical Computer Science 475 (2013),
    21--33.

    \bibitem{BBGT}
    M. Bednarska-Bzd\c{e}ga, O. Ben-Eliezer, L. Gishboliner, and T. Tran, ``On the separation conjecture in Avoider-Enforcer games", manuscript, arXiv:1709.09065 [math.CO]



    \bibitem{BHL}
    M. Bednarska-Bzd\c{e}ga, D. Hefetz, and T. Luczak, ``Picker-Chooser
    fixed graph games'', Journal of Combinatorial Theory Series B 119
    (2016), 122--154.


    \bibitem{BL}
    M.\ Bednarska  and T.\  \L uczak, ``Biased positional games for
    which random strategies are nearly optimal'', Combinatorica 20
    (2000), 477--488.


    \bibitem{BFHK}
    S. Ben-Shimon, A. Ferber, D. Hefetz, and M. Krivelevich, ``Hitting time results for Maker-Breaker games'', Random Structures and Algorithms 41 (2012), 23--46.

    \bibitem{Bol98}
    B. Bollob\'as. \textbf{Random graphs}, $2^{nd}$ ed. Cambridge University
    Press, 2001.

    \bibitem{BT}
    B. Bollob\'as and A. Thomason, ``Threshold functions", Combinatorica 7 (1987), 35-–38.


    \bibitem{CE}
    V. Chv\'atal and P. Erd\H{o}s, ``Biased positional games'', Annals of
    Discrete Mathematics 2 (1978), 221--228.




    \bibitem{CFKL}
    D. Clemens, A. Ferber, M. Krivelevich, and A. Liebenau, ``Fast Strategies In Maker-Breaker Games Played on Random Boards'', Combinatorics, Probability and Computing 21 (2012), 897--915.



    \bibitem{DK}
    O. Dean and M. Krivelevich, ``Client-Waiter games on complete and
    random graphs'', Electronic Journal of Combinatorics 23 (2016),
    P4.38.


    \bibitem{FGKN}
    A. Ferber, R. Glebov, M. Krivelevich, and A. Naor, ``Biased games on
    random graphs'', Random Structures and Algorithms 46 (2015), 651--676.




    \bibitem{FKN}
    A. Ferber, M. Krivelevich, and A. Naor, ``Avoider-Enforcer games
    played on edge disjoint hypergraphs'', Discrete Mathematics 313
    (2013), 2932--2941.


    \bibitem{FrBook} A. Frieze, and M.  Karo\'nski. \textbf{Introduction to random graphs}. Cambridge University Press, 2015.

    \bibitem{GS}
     H. Gebauer and T. Szab\'o, ``Asymptotic random graph intuition for the biased connectivity game", Random Structures and Algorithms 35 (2009), 431–-443.

    \bibitem{AEstars}
    A. Grzesik, M. Mikalacki, Zolt\'{a}n L\'{o}r\'{a}nt Nagy,
    Alon Naor, Bal\'{a}zs Patk\'{o}s, and Fiona Skerman, ``Avoider-Enforcer
    star games'', Discrete Mathematics and Theoretical Computer Science
    17 (2015), 145--160.


    \bibitem{HLV}
    Y. O. Hamidoune and M. Las Vergnas, ``A solution to the box game'',
    Discrete Mathematics 65 (1987), 157--171.


    \bibitem{HKSS10}
    D. Hefetz, M. Krivelevich, M. Stojakovi\'c, and T. Szab\'o,
    ``Avoider-Enforcer: The rules of the game'', Journal of Combinatorial Theory Series A 117 (2010), 152--163.


    \bibitem{HKSS2009b}
    D.~Hefetz, M.~Krivelevich, M.~Stojakovi\'{c}, and T.~Szab\'{o}, ``Fast winning strategies in Maker-Breaker games'',
    Journal of Combinatorial Theory Series B 99 (2009), 39--47.


    \bibitem{HKSSbook}
    D. Hefetz, M. Krivelevich, M. Stojakovi\'c, and T. Szab\'o,
    \textbf{Positional Games}, Birkh\"{a}user, 2014.


    \bibitem{HKS07}
    D. Hefetz, M. Krivelevich, and T. Szab\'o, ``Avoider-Enforcer games'',
    Journal of Combinatorial Theory Series A 114 (2007),
    840--853.


    \bibitem{HKT}
    D. Hefetz, M. Krivelevich, and W. E. Tan, ``Waiter-Client and
    Client-Waiter Hamiltonicity games on random graphs'', European Journal
    of Combinatorics 63 (2017), 26--43.

    \bibitem{HS}
    D. Hefetz and S. Stich, ``On two problems regarding the Hamilton cycle game", The Electronic Journal of Combinatorics, Vol 16 (1) (2009), R28.

    \bibitem{JLR}
    S. Janson, T. \L uczak, and A. Rucinski, {\bf Random graphs}, John Wiley $\&$
    Sons, 2000.


    \bibitem{K}
    M. Krivelevich, ``The critical bias for the Hamiltonicity game is $(1+o(1))n/\ln n$", Journal of the American Mathematical Society 24 (2011), 125--131.

    \bibitem{Lehman}
    A.~Lehman,
    ``A solution of the Shannon switching game'',
    Journal of the Society for Industrial and Applied Mathematics
    12 (1964), 687--725.

    \bibitem{LRV}
    T. \L uczak, A. Ruci\'{n}ski, and B. Voigt, ``Ramsey properties of random graphs'',  Journal of Combinatorial Theory, Series B 56 (1992), 55--68.


    \bibitem{MullerStojakovic}
    T. M\"{u}ller and M. Stojakovi\'{c}, ``A threshold for the Maker-Breaker clique game'' Random structures and algorithms 45 (2014), 318--341.


    \bibitem{HgameRandomGraphs}
    R. Nenadov,  A. Steger, and M. Stojakovi\'c, ``On the threshold for the Maker-Breaker H-game'', Random Structures and Algorithms 49 (2016), 558--578.


    \bibitem{RR1}
    V. R\"odl and A. Ruci\'nski, ``Lower bounds on probability thresholds for Ramsey properties'',
    In Combinatorics, Paul Erd\H os is eighty, Vol. 1, Bolyai Soc. Math. Stud. 1993, pp. 317--346.


    \bibitem{RR2}
    V. R\"odl and A. Ruci\'nski, ``Random graphs with monochromatic triangles in every edge
    coloring'', Random Structures Algorithms 5 (1994), 253--270.


    \bibitem{RR3}
    V. R\"odl and A. Ruci\'nski, ``Threshold functions for Ramsey properties'', J. Amer. Math.
    Soc. 8 (1995), 917--942.


    \bibitem{Sthesis}
    M. Stojakovi\'c, ``Games on Graphs'', PhD thesis, ETH Zurich,
    Switzerland, September 2005.

    \bibitem{PGonRandomGraphs}
    M. Stojakovi\'c and T. Szab\'o, ``Positional games on random graphs'', Random Structures and Algorithms 26 (2005), 204--223.


    \bibitem{IntroGraphTheory}
    D. B. West, {\bf Introduction to Graph Theory}, $2^{nd}$ ed. Prentice Hall, 2001.

\end{thebibliography}
\end{document}